\documentclass{amsart}

\DeclareSymbolFont{bbold}{U}{bbold}{m}{n}
\DeclareSymbolFontAlphabet{\mathbbold}{bbold}

\usepackage{amsmath}%
\usepackage{todonotes}

\usepackage{amsfonts}%
\usepackage{amssymb}%
\usepackage{graphicx}
\usepackage{bm}

\usepackage{mathabx}
\usepackage{rotating}
\usepackage[all,cmtip]{xy}
\usepackage{cite}
\usepackage{hyperref}
\usepackage{geometry}
\usepackage{bbm}
\usepackage{stmaryrd}
\usepackage{tensor}
\usepackage{wasysym}

\newtheorem{dummy}{dummy}[section]              
\newtheorem{lemma}[dummy]{Lemma}
\newtheorem{theorem}[dummy]{Theorem}
\newtheorem{corollary}[dummy]{Corollary}
\newtheorem{proposition}[dummy]{Proposition}

\newtheorem{conjecture}[dummy]{Conjecture}

\newtheorem{claim}[dummy]{Claim}
\theoremstyle{definition}                                  
\newtheorem{definition}[dummy]{Definition}
\newtheorem{example}[dummy]{Example}
\newtheorem{remark}[dummy]{Remark}

\newtheorem{defn}[dummy]{Definition}

\newcommand{\Vect}{\mathbf{Vect}}

\DeclareMathOperator{\Spec}{Spec}

\DeclareMathOperator{\Hom}{Hom}

\DeclareMathOperator{\End}{End}

\newcommand{\Mod}{\mathbf{Mod}}

\newcommand{\mhyphen}{{\hbox{-}}}
\newcommand{\module}{\mhyphen\mathrm{mod}}
\newcommand{\rmodule}{\mathrm{mod}\mhyphen}
\newcommand{\modcat}{\mhyphen\mathcal{M}od}
\newcommand{\comod}{\mhyphen\mathrm{comod}}
\newcommand{\bimod}{\mhyphen\mathrm{bimod}}

\DeclareMathOperator{\Sym}{Sym}

\newcommand{\QC}{QC}

\DeclareMathOperator{\Rep}{Rep}

\newcommand{\DGCat}{\mathbf{DGCat}}

\newcommand{\too}{\longrightarrow}
\newcommand{\Id}{Id}

\newcommand{\aff}{\textit{aff}}

\newcommand{\bs}{\backslash}

\newcommand{\ls}[2]{\tensor*[^{#1 }]{{#2}}{}}

\newcommand{\lsub}[2]{\tensor*[_{#1}]{#2}{}}
\newcommand{\GIT}{{/\! /}}
\newcommand{\lGIT}{\bs \! \bs}
\newcommand{\GITquot}[3]{{#1}\lGIT {#2} \GIT {#3}}

\DeclareMathOperator{\Ind}{Ind}
\DeclareMathOperator{\Res}{Res}

\newcommand{\hc}{hc}
\newcommand{\ch}{\mathfrak{ch}}
\newcommand{\ad}{\mathrm{ad}}
\newcommand{\HC}{\mathcal{HC}}

\newcommand{\Loc}{\mathcal Loc}

\newcommand{\reg}{\mathrm{reg}}
\newcommand{\Whit}{Whit}
\newcommand{\Kost}{Kost}
\newcommand{\Char}{Char}
\newcommand{\Wh}{Wh}

\newcommand{\Nil}{Nil}
\newcommand{\QHR}{QHR}
\newcommand{\forg}{forg}

\newcommand{\oN}{\overline{N}}

\newcommand{\fc}{\mathfrak{c}}

\newcommand{\fg}{\mathfrak{g}}
\newcommand{\fgx}{\mathfrak{g}^\star}

\newcommand{\fa}{\mathfrak{a}}
\newcommand{\fh}{\mathfrak{h}}
\newcommand{\ft}{\mathfrak{t}}

\newcommand{\fb}{\mathfrak{b}}

\newcommand{\fM}{\mathfrak{M}}

\newcommand{\fL}{\mathfrak {L}}

\newcommand{\fD}{\mathfrak{D}}
\newcommand{\fX}{\mathfrak{X}}

\newcommand{\fz}{\mathfrak{z}}
\newcommand{\fd}{\mathfrak{d}}

\newcommand{\fE}{\mathfrak{E}}
\newcommand{\fB}{\mathfrak{B}}
\newcommand{\fO}{\mathfrak{O}}

\newcommand{\fWh}{\mathfrak{Wh}}

\newcommand{\cA}{\mathcal A}

\newcommand{\cC}{\mathcal C}
\newcommand{\cD}{\mathcal D}

\newcommand{\cF}{\mathcal F}
\newcommand{\cG}{\mathcal G}
\newcommand{\cH}{\mathcal H}
\newcommand{\cI}{\mathcal I}

\newcommand{\cK}{\mathcal K}

\newcommand{\cM}{\mathcal M}
\newcommand{\cN}{\mathcal N}
\newcommand{\cO}{\mathcal O}
\newcommand{\cP}{\mathcal P}

\newcommand{\cR}{\mathcal R}

\newcommand{\cW}{\mathcal W}
\newcommand{\cX}{\mathcal X}

\newcommand{\cZ}{\mathcal Z}

\newcommand{\cWh}{\mathcal{W}h}

\newcommand{\bG}{\mathbf{G}}
\newcommand{\bB}{\mathbf{B}}

\newcommand{\bW}{\mathbf{W}}

\newcommand{\bP}{\mathbf P}

\newcommand{\bh}{\mathbf{h}}

\newcommand{\A}{\mathbb A}

\newcommand{\C}{\mathbb C}

\newcommand{\F}{\mathbb F}
\newcommand{\G}{\mathbb G}

\newcommand{\Z}{\mathbb Z}

\newcommand{\uG}{{\underline{G}}}

\newcommand{\uP}{\underline{P}}

\newcommand{\uH}{\underline{H}}

\newcommand{\XYX}{X\times_Y X}

\newcommand{\ot}{\otimes}
\newcommand{\oo}{\infty}

\newcommand{\CC}{\mathbb C}

\def\fc{\mathfrak c}

\def\fU{\mathfrak U}
\def\fZ{\mathfrak Z}
\def\fn{\mathfrak n}
\def\fb{\mathfrak b}

\def\ft{\mathfrak t}

\def\fgxr{\mathfrak g^{\ast}_{\reg}}
\def\fhx{\mathfrak h^\ast}
\def\fgx{\mathfrak g^\ast}
\def\fg{\mathfrak g}
\def\fa{\mathfrak a}

\newcommand{\actson}{\circlearrowright}
\newcommand{\onacts}{\circlearrowleft}

\newcommand{\Cx}{{\C^\times}}
\def\Gv{{G^{\vee}}}
\def\Tv{{T^{\vee}}}
\def\cGr{{\cG r^{\vee}}}
\def\cDv{\breve{\cD}}
\def\Dhol{\cDv_{hol}}
\def\LGv{{LG^\vee}}
\def\LGpv{{LG^\vee_+}}

\def\uGr{\underline{\mathcal{G}r}}
\def\uGrv{\underline{\mathcal{G}r}^\vee}

\def\Gm{\G_m}

\newcommand{\mc}{\mathcal}

\makeatletter
\newcommand*\leftdash{\rotatebox[origin=c]{-45}{$\dabar@\dabar@\dabar@$}}
\newcommand*\rightdash{\rotatebox[origin=c]{45}{$\dabar@\dabar@\dabar@$}}
\makeatother

\newcommand{\quot}[3]{{#1}\backslash{#2}/{#3}} 
\newcommand{\wqw}[3]{{#1}\leftdash{#2}\rightdash{#3}} 
\newcommand{\wq}[2]{{#1}\leftdash{#2}}
\newcommand{\qw}[2]{{#1}\rightdash{#2}}

\newcommand{\NGNpsi}{N {}_\psi{\backslash} G /_{\psi} N}

\newcommand{\Ngo}{Ng\hat{o}}
\newcommand{\tNgo}{\Ngo^\sim}

\newcommand{\adjquot}{{/_{\hspace{-0.2em}ad}\hspace{0.1em}}}

\calclayout 

\title[Symmetries of Categorical Representations and the Quantum Ng\^o Action]{Symmetries of Categorical Representations\\and the Quantum Ng\^o Action}

\author{David Ben-Zvi} \address{Department of Mathematics\\University
  of Texas\\Austin, TX 78712-0257} \email{benzvi@math.utexas.edu}
\author{Sam Gunningham}\address{Department of Mathematics\\University
  of Texas\\Austin, TX 78712-0257} \email{gunningham@math.utexas.edu}

\begin{document}

\begin{abstract}
We observe that all classical Hamiltonian systems coming from the invariant polynomials on a reductive Lie algebra $\fg$ can be integrated in a universal way. This is a consequence of 
Ng\^o's action of the group scheme $J$ of regular centralizers in $G$ on all centralizers: the Hamiltonian flows associated to invariant polynomials integrate to an action of $J$ as commutative symplectic groupoid. 
We quantize the Ng\^o action, providing a universal integration for all quantum Hamiltonian systems coming from the center $\fZ\fg=\fZ(\fU\fg)$ of the enveloping algebra (after a cohomological regrading in the spirit of cyclic homology and supersymmetric gauge theory). Namely
Kostant's Whittaker description of $\fZ\fg$ 
integrates to the action of a commutative quantum groupoid $\fWh$, 
the bi-Whittaker Hamiltonian reduction of $\fD_G$, as do quantum Hamiltonian systems coming from the action of $\fZ\fg$ (Harish-Chandra's higher order Laplacians). These actions come from a braided tensor functor, the {\em quantum Ng\^o map}, from the $\cW$-category $\cWh=\fWh\module$ to adjoint-equivariant $\cD$-modules $\cD(G\adjquot G)$ (the center of the convolution category $(\cD(G),\ast)$), which gives a categorical family of $G$-invariant commuting operators on any strong $G$-category. This action also leads to a notion of Langlands parameters (or refined central character) for categorical representations of $G$ and character sheaves, and a new commutative symmetry of homology of character varieties of surfaces. 

We derive our construction as the Langlands dual form of a simple symmetry principle for groupoids. Namely the symmetric monoidal category of equivariant sheaves, i.e., modules for the convolution algebra $H$, acts centrally on the corresponding convolution category $\cH$, i.e., we have a braided functor $H\module\to \cZ(\cH)$. In particular modules for the nil-Hecke algebra for any Kac-Moody group act centrally on the corresponding Iwahori-Hecke category. We use the renormalized Geometric Satake theorem of Bezrukavnikov-Finkelberg to identify $H\module$ for the equivariant affine Grassmannian for $\Gv$ with the $\cW$-category for $G$, and the corresponding central action gives both the Ng\^o action for $G$ and its quantization.


  \end{abstract}
\maketitle
\tableofcontents

\section{Overview}

The purpose of this paper is to describe a general mechanism for constructing large commuting families of operators in the setting of geometric representation theory. We first sketch the underlying formal mechanism and then describe its primary application.

\subsection{Symmetries of Convolution Categories}\label{convolution categories}

Let $X$ denote a stack and $\cG\actson X$ an ind-proper groupoid acting on $X$. Let $\cR=Shv(X)$ be the symmetric monoidal
category of sheaves on $X$. In all our examples, $\cR=R\module$ is described as modules for a commutative algebra $R=\omega(X)$.
Let $\cK=Shv(X)^\cG$ the symmetric monoidal category of $\cG$-equivariant sheaves. We have an equivalence
$$\cK\simeq H\module$$ where the Hecke algebra $H=(\omega(\cG),\ast)$ is the associated groupoid algebra (concretely a cocommutative Hopf algebroid over $R$). Now consider the Hecke category $\cH=Shv(\cG)$ of sheaves on $\cG$. 
It forms a categorical cocommutative Hopf algebroid over $\cR=Shv(X)$: in addition to the convolution monoidal structure and the diagonal action of $(Shv(X),\otimes)$, it carries a commutative pointwise tensor product operation. $\cH$-modules represent $\cG$-equivariant sheaves of categories on $X$, hence are naturally linear over $\cG$-equivariant sheaves on $X$. 
As a consequence of this structure we find the following:

\begin{theorem}[Informal]\label{main intro}
There is a 
braided monoidal functor $\fz:\cK\to \cZ(\cH)$ from the category of equivariant sheaves
to the Drinfeld center of the convolution category, lifting the diagonal embedding $\fd:\cR\to \cH$ and 
admitting a monoidal left inverse $\fa:\cZ(\cH)\to \cK$.
Thus we have a diagram with commutative square as follows, with morphisms labeled by their level of monoidal structure:

\begin{equation}\label{basic diagram}\xymatrix{\cK\ar[r]_{E_2}\ar[d]_-{E_\infty}& \cZ(\cH) \ar@/_1pc/_-{E_1}[l]  \ar[d]^{E_1}\\
\cR \ar[r]_{E_1} & \cH}\end{equation}
\end{theorem}

Theorem~\ref{main intro} applies to any setting where we have a category of geometric objects (stacks) theory of sheaves $X\mapsto Shv(X)$ admitting ($*$-)pushforward and ($!$-)pullback functors, satisfying base change and $(p_*,p^!)$ adjunction for ind-proper maps $p$. Together such a sheaf theory defines a functor from a correspondence category of stacks to a 2-category of categories. Such sheaf theories are one of the main objects of study of the book~\cite{GR} of Gaitsgory and Rozenblyum, which in particular develops two main examples of sheaf theories:
\begin{itemize}
\item the theory of ind-coherent sheaves $X\mapsto \QC^!(X)$, a ``renormalized" variant of the theory of quasicoherent sheaves
\item the theory of $\cD$-modules $X\mapsto \cD(X) = \QC^!(X_{dR})$.
\end{itemize}

We will mostly be interested in applying the theorem to a mild variant of the theory of $\cD$-modules, the theory of {\em ind-holonomic $\cD$-modules} on a class of ind-algebraic stacks, which we describe in Section~\ref{sheaf theory} using the formalism of~\cite{GR}. In the examples we study (equivariant flag varieties and in particular affine Grassmannians), this theory produces simply (the ind-completed version of) the familiar categories of equivariant constructible complexes. 
The ordinary equivariant $\cD$-module categories (where equivariant holonomic sheaves are not necessarily compact) can be recovered as a completion with respect to the equivariant cohomology of a point. In Section~\ref{KM section} we describe two related applications of this result, in which the category of modules for a nil-Hecke algebra acts centrally on convolution categories built out of flag varieties for the  corresponding Kac-Moody group or the reflection representation of the corresponding Coxeter group.

\subsection{The Quantum Ng\^o Action}

Our motivation stems from the following application. Let $G$ be a complex reductive group, and consider the spherical Hecke category $\cH=\cH_{sph}$ associated to the Langlands dual group $\Gv$: the category of sheaves on the equivariant Grassmannian 
\[
\uGrv = \quot{\LGpv}{\LGv}{\LGpv}
\] which we consider as a groupoid stack 
\[
\cG=\uGrv\; \actson \; X=pt/\LGpv
\]

We may apply Theorem \ref{main intro} in this setting, obtaining a diagram of the form of Diagram~\ref{basic diagram}. Langlands duality, in particular the renormalized geometric Satake theorem of Bezrukavnikov-Finkelberg~\cite{BezFink}, leads to interpretations of the various parts of the diagram in terms of the original group $G$, which naturally appear in a cohomologically graded form.
Bezrukavnikov, Finkelberg and Mirkovic~\cite{BFM} identifed the ring $H=H_\ast(\uGrv)$ with the coordinate ring of the commutative group scheme $J$ of regular centralizers (see also the influential works of Teleman~\cite{teleman} where this construction is applied to categorical representation theory and symplectic topology and Braverman-Finkelberg-Nakajima~\cite{BraFinkNa} where it is generalized to a construction of Coulomb branches of 3d $\cN=4$ supersymmetric gauge theory).  We then identify the symmetric monoidal category $\cK=H\module$ with the category of quasi-coherent sheaves on $J$ under convolution. The functor $\fz: \cK \to \cZ(\cH)$ can be interpreted in terms of the \emph{Ng\^o homomorphism} from regular centralizers to all centralizers, leading to a new conceptual construction (the original construction was via a Hartog's lemma argument). The Ng\^o homomorphism is best known for its central role in the proof of the Fundamental Lemma~\cite{Ngo}. As we explain below, it also gives rise to a canonical integration of all ``$G$-integrable systems": the commuting flows on any Hamiltonian $G$-space $X$ (or any of its Hamiltonian reductions by subgroups of $G$) coming from the $G$-invariant Poisson map $X\to \fc=\Spec\C[\fgx]^G$ integrate to an action of the commutative symplectic groupoid $J\to \fc$. 

The multiplicative group acts on the equivariant Grassmannian by loop rotation; considering a loop rotation equivariant version of the spherical Hecke category leads to a deformation over $H^*(BS^1)=\C[\hbar]$ also described in~\cite{BezFink}. As is familiar from the theory of cyclic homology and the Nekrasov $\Omega$-background~\cite{NW} (in particular the theory of quantized Coulomb branches in 3d $\cN=4$ gauge theories~\cite{BraFinkNa}), the parameter $\hbar$ of the deformation appears in cohomological degree two, and as a result the familiar structures of representation theory appear in their ``cohomologically sheared" (or ``asymptotic"~\cite{BezFink}) avatars as differential graded algebras. Under this deformation, $\cK_\hbar=H_\hbar\module$ gets identified (as a monoidal category) with the Whittaker Hecke category $\cWh_\hbar$ of bi-Whittaker $\cD_\hbar$-modules on $G$ (which we will refer to as the $\cW$-category). The Hecke algebra $H$ itself is identified both with the spherical subalgebra of the nil-Hecke algebra associated to the affine Weyl group $W_{\aff}$ of $\Gv$, and with the bi-Whittaker differential operators $\fWh_\hbar$ on $G$. (The underlying category $\cWh_\hbar=\fWh_\hbar\module$ for $\hbar\neq 0$ has been recently described explicitly in~\cite{lonergan, ginzburg whittaker} as sheaves on the coarse quotient $\fhx//W_{\aff}$ of the Cartan by the affine Weyl group, an identification that is expected to respect the tensor structure.)
In particular, we deduce (from a general conceptual point of view) that the convolution structure on the Whittaker Hecke category is naturally symmetric monoidal, answering a question of Arinkin-Gaitsgory. Moreover the functor $\fz$ becomes a central action on the category of conjugation equivariant $\cD_\hbar$-modules on $G$ which is right inverse to the quantum Kostant section (Whittaker reduction); a quantum version of the Ng\^o homomorphism, conjectured by Nadler.

\begin{theorem}\label{central action intro} 
The $\cW$-category $\cWh_\hbar$ is naturally symmetric monoidal, and equipped with a braided monoidal functor
 $\Ngo_\hbar:\cWh_\hbar \to \cD_\hbar(G\adjquot G)$ lifting the quantum characteristic polynomial map $\Char_\hbar:\cZ_\hbar\to \HC_\hbar$ and 
admitting a monoidal left inverse $\Whit_\hbar:\cD_\hbar(G\adjquot G)\to \cWh_\hbar$.\footnote{The functors $\Whit_\hbar$ and $\Ngo_\hbar$ do not form an adjoint pair in general, although see Remark \ref{remark quantum affine}.}
Thus we have a commutative diagram:
\[
\xymatrix{
\cWh_{\hbar} \ar[r]_{\Ngo_\hbar}\ar[d]_-{}& \ar@/_1pc/_-{\Whit_\hbar}[l] \cD_{\hbar}(G\adjquot G) \ar[d]^{\Gamma}\\
\cZ_{\hbar} \ar[r]_{\Char_\hbar} & \HC_{\hbar}  }
\]
In particular, $\cWh_\hbar$ acts by $G$-endomorphisms on $\cD_\hbar(G)$-module categories.
\end{theorem}

In particular, passing to invariants for a subgroup $K\subset G$, the $\cW$-category acts on $(\fU_\hbar\fg,K)\module$ and on $\cD_\hbar(K\backslash M)$ for any $G$-variety $M$.

Thus the quantum Ng\^o functor defines a categorical counterpart to Harish-Chandra's construction of a $G$-invariant commuting family of operators on $G$-spaces, parametrized by sheaves on $\fhx\GIT W^{\aff}$.
It follows that one can consider categories of $\cWh_\hbar$-eigensheaves in any $G$-category---a more refined version of infinitesimal character. In particular this applies (for the conjugation action of $G$ on itself) to give a refined version of the notion of central character of character sheaves.
We expect that this structure will play a key role in better understanding the truncated Hecke and character sheaf categories defined by Lusztig \cite{lusztig cells, lusztig convolution} (see also \cite{BFO}). More generally, the entire character field theory of~\cite{character2} is linear over $\cWh_\hbar$, leading to a spectral decomposition of the homology of character varieties of surfaces over $\fhx\GIT W^{\aff}$.



While the spherical Hecke category corresponds to the cohomological Harish-Chandra bimodule category $\HC_\hbar$, it is desirable to have a version of Theorem \ref{central action intro} which applies to the usual category of Whittaker $\cD$-modules and Harish-Chandra bimodules:

\begin{theorem}\label{unsheared ngo intro}
There is a canonical $E_2$-morphism $\Ngo: \cWh \to \cD(G\adjquot G)$ which fits in to a diagram as in Theorem \ref{central action intro}. Moreover, this functor restricts to an exact functor of braided monoidal abelian categories (appearing as the heart of the natural $t$-structure on the source and target).
\end{theorem}
In Section \ref{ss graded lift} we sketch a proof of Theorem \ref{unsheared ngo intro}, by constructing a graded lift of the Ng\^o functor, i.e. a lift to the category consisting of objects with a compatible external grading. Geometrically, such a graded lift corresponds to a mixed version of the Satake category, in the sense of Beilinson-Ginzburg-Soergel \cite{BGS} (see also \cite{riche}).



Less categorically, the quantum Ng\^o action may be interpreted in terms of a {\em quantum integration} of all $G$-quantum Hamiltonian systems: $\fWh_\hbar$ forms a commutative quantum groupoid (cocommutative Hopf algebroid) quantizing the commutative symplectic groupoid $J$,
which acts on $\hbar$-differential operators $\fD_{\hbar,M}$ on any $G$-space $M$ (or any of its quantum Hamiltonian reductions) extending the Harish-Chandra higher Laplacians $\fZ_\hbar\fg\to \fD_{\hbar,M}$ (in particular $\fD_{\hbar,M}$ is naturally a $\fWh_\hbar$-comodule). This structure is closely related to the 
theory of shift maps for quantum integrable systems. (Examples include quantized Coulomb branches of 3d $\cN=4$ gauge theories, and more generally arbitrary supersymmetric reductions of 4d $\cN=4$ super-Yang Mills on an interval, see Section~\ref{SUSY}.)
Since the theory of Hopf algebroids in the $\oo$-categorical setting is not currently documented in the literature, we confine ourselves to remarks (see Remarks~\ref{quantum groupoid remark} and~\ref{abelian vs derived}) and defer a more detailed discussion to a future paper.

\subsection{Outline of paper} In the rest of the introduction, we review the idea behind Theorem~\ref{main intro} and some of its instances, the classical Ng\^o construction, its quantization and their applications. In Section~\ref{sheaf theory} we develop some basic sheaf theory functoriality in the setting relevant for the renormalized geometric Satake theorem of~\cite{BezFink}, i.e., equivariant sheaves on the affine Grassmannian. In Section~\ref{Hecke section} we prove Theorem~\ref{main intro}. In Section~\ref{filtered D-mod section} we review some aspects of categorical representations and filtered $\cD$-modules. Finally in 
Section~\ref{quantum section} we describe how specializing the theorem in the setting of the affine Grassmannian produces the quantum Ng\^o action. 

\medskip 

{\bf Setting:} Throughout the paper we work in the setting of derived algebraic geometry over a field $k$ of characteristic zero, following~\cite{HA, GR}. Thus ``category" indicates an $\infty$-category, commutative or symmetric monoidal indicate $E_\infty$, schemes are derived $k$-schemes, and so forth, unless explicitly noted otherwise.

\section{Introduction} 

We begin by describing the classical Ng\^o action and some of its applications (Section~\ref{classical Ngo}), followed by its sheaf-theoretic reinterpretation (Section~\ref{monoidal interp}) and its quantization (Sections~\ref{quantum Kostant section} and~\ref{quantum Ngo section}). We also explain some of the applications of the quantum Ng\^o action, in particular to quantum integrability, in Section~\ref{integrable section}.  In Section~\ref{parabolic induction section} we explain a perspective on our quantum Ng\^o map that is closer in spirit to the classical theory of character sheaves. In Section~\ref{KM section} we mention some other applications of Theorem~\ref{main intro} (to Kac-Moody groups and Coxeter systems) and a couple of toy examples. Finally in Section~\ref{further section} we outline some further directions and perspectives, including geometric Langlands, character varieties and supersymmetric gauge theory.


%
%

\subsection{The classical Ng\^o action}\label{classical Ngo}
In the following sections we provide some background for the primary applications of our results: the classical and quantum Ng\^o actions.

Fix a complex reductive group $G$ with Lie algebra $\fg$. We also fix a Borel subgroup $B$ with unipotent radical $N$, and write $H=B/N$ for the universal torus (with Lie algebras $\fb$, $\fn$ and $\fh$ respectively). Let $$\fc:=\Spec \left( \CC[\fgx]^G\right) \simeq \fhx\GIT W$$ denote the adjoint quotient scheme.
Recall the characteristic polynomial map and Kostant section 
\[
\xymatrix{
\fgx/G \ar[r]_{\chi} & \ar@/_1pc/[l]_{\kappa} \fc }
\]
The Kostant section $\kappa:\fc\to \fgx$ lands in the open substack $\fgxr/G\subset \fgx/G$ of regular elements - the locus of $x\in \fgx$ whose stabilizer $G_x$ has the minimal dimension $l=rk(\fg)$. It can be described in terms of Hamiltonian reduction: fix $\psi\in \fn^\ast\simeq \fg/\fb$ a non degenerate character of $\fn$.  Then the composite map $$\xymatrix{\fgx\GIT _{\psi} N \ar[r]& \fgx/G \ar[r]^-{\chi}& \fc}$$ is an isomorphism, and the Kostant section is its inverse.

We denote by $$I\simeq T^\ast(G/G)\longrightarrow \fgx/G$$ the inertia stack (or derived loop space) of the adjoint quotient: informally,
\[
I = \left\{(g,x) \in G\times \fg^\ast \mid coAd_g(x)=x \right\}/G.
\]
It can be identified as the cotangent stack to the stack of conjugacy classes. We can restrict $I$ over the Kostant section, resulting in the {\em group scheme of regular centralizers}
\[
J = \kappa^\ast I \longrightarrow \fc.
\] 
The group scheme $J$ also has a description as a Hamiltonian reduction of $T^\ast G$: 
$$J\simeq N{}_\psi \backslash \!\backslash T^\ast G \GIT _\psi N.$$
It has the natural structure of commutative {\em symplectic groupoid} over $\fc$ -- in particular $Lie(J)\simeq T^\ast\fc$ as commutative symplectic Lie algebroids. 

Note that there is an equivalence of group schemes over the regular locus
\[
\chi^\ast J|_{\fg^\ast_{\reg}/G} \simeq I|_{\fg^\ast_{\reg}/G} 
\]
In fact the Kostant section defines an equivalence 
\[
\fgxr/G\simeq B_\fc J\longrightarrow \fc
\]
of the regular adjoint quotient with the classifying stack of $J\to \fc$.

Ng\^o made the crucial observation that regular centralizers act canonically on all centralizers:

\begin{lemma}[Ng\^o]
	The equivalence above extends to a morphism of group schemes over $\fg^\ast/G$:
	\[
	\chi^\ast J \to I
	\]
\end{lemma}

The Lemma is a simple consequence of the Hartogs principle: the open substack $\fg^\ast_{\reg}/G \subseteq \fg^\ast/G$ has complement of codimension at least three.  For $GL_n$, this map can be described as the natural action of invertible functions on the spectrum of a matrix $M$ via operators commuting with $M$. In general, no direct description of the Ng\^o map was available.  



Ng\^o introduced his map as a universal ``mold" from which many more concrete actions are formed. Ng\^o applied it (extending the Donagi-Gaitsgory spectral theory for Higgs bundles~\cite{DG}) to give a new abelian symmetry group of the cohomology of Hitchin fibers, which plays a crucial role in his study of endoscopy and proof of the Fundamental Lemma.
Namely, given any variety $C$, the Ng\^o action (in its equivalent ``delooped" form, an action of the abelian group stack $BJ\to \fc$ of $\fgx/G$), gives an action of the commutative group-stack $Map(C,BJ)\to Map(C,\fc)$
on the stack $Map(C,\fgx/G)$ of $G$-Higgs bundles\footnote{By keeping track of $\Gm$-equivariant version of the above constructions, Ng\^o obtains also a more general version twisted by a line bundle on $C$, see Section~\ref{char poly section}.} on $C$.

We observe that the Ng\^o map has another concrete manifestation (which does not appear to have been discussed in the literature).
Given any Hamiltonian $G$-space $X$ with equivariant moment map 
$$\xymatrix{X\ar[r]^-{\mu}& \fgx \ar[r]^-{\chi}& \fc}$$ the induced map to $\fc$ defines a collection of 
Poisson-commuting Hamiltonians on $X$ or (thanks to $G$-invariance) on any Hamiltonian reduction $Y=X\GIT _{\mathbb O}K$ of $X$ by a subgroup of $G$ - a mechanism that was used to describe and solve Toda, Calogero-Moser and many other integrable systems (see e.g.~\cite{KKS,Kostant Toda} and~\cite{etingof}).
The Hamiltonian flows coming from the Poisson map $\chi\circ\mu:X\to \fc$ may be interpreted as defining an action of the trivial commutative Lie algebroid $$T^\ast\fc\simeq Lie(J)\longrightarrow \fc$$ on $X$. 
A simple consequence of the Ng\^o construction is the following:

\begin{proposition}\label{classical integration} The Hamiltonian flows (action of $Lie(J)$) on any Hamiltonian reduction $Y\to\fc$ of a Hamiltonian $G$-space $X$ integrate canonically to an action of the symplectic groupoid $J\to \fc$.
\end{proposition}

\begin{proof} 
A Hamiltonian $G$-action on $X$ is equivalent to an action of the symplectic groupoid $T^\ast G$ over $\fgx$. We may restrict this to an action of the inertia groupscheme $I$, and then use the Ng\^o map to induce an action of $J$ -- concretely, the action map is given as follows:
$$J \times_\mathfrak{c} X= (J \times_\mathfrak{c} \mathfrak{g}^\ast) \times_{\mathfrak{g}^\ast} X \to I \times_{\mathfrak{g}^\ast} X\to X$$
\end{proof}

\subsection{Monoidal interpretation}\label{monoidal interp}
In order to describe our construction of the Ng\^o action and its quantization, we first pass from spaces to tensor categories of sheaves. 


The category $\QC( I)=\cZ(\QC(\fgx/G))$ of sheaves on the inertia stack $ I\simeq T^\ast(G/G)$ of $\fgx/G$ (the Drinfeld center of $\QC(\fgx/G)$) is naturally braided under the convolution product.
Using the Ng\^o homomorphism, we may also define a braided\footnote{The braided structure can be seen by delooping the functor to an action of $(\QC(B_\fc J),\ast)$ on $\QC(\fgx/G)$.} monoidal functor (with respect to the convolution structures on both sides)
\[
\Ngo_0: \QC(J) \to \QC( I)
\]
given by the correspondence.
\begin{equation}\label{Ngo correspondence}
\xymatrix{
J &\ar[l] \chi^\ast(J)\ar[r] &  I
}
\end{equation}

Note that there is another monoidal functor $$\Whit_0:\QC( I)\to \QC(J)$$
given by the correspondence
\begin{equation}\label{Kostant correspondence}
\xymatrix{
J&\ar[l]_-{\sim} \kappa^\ast( I) \ar[r]&  I
}
\end{equation}
provided by the Kostant section. Moreover $\Whit_0$ is a left inverse to $\Ngo_0$, $\Whit_0\circ \Ngo_0\simeq \Id$. Thus we have a commutative diagram:

$$\xymatrix{\QC(J)\ar[r]_{\Ngo_0}\ar[d]_-{}& \QC( I) \ar@/_1pc/_-{\Whit_0}[l]  \ar[d]^{}\\
\QC(\fc) \ar[r]_{\Char_0} & \QC(\fgx/G)}$$
where $\Char_0 = \chi^\ast$.

One can check that the Ng\^o action is identified, via the renormalized Satake theorem of~\cite{BezFink}, with the construction of 
Theorem~\ref{main intro} applied to $\LGpv$-equivariant sheaves on the affine Grassmannian $\LGv/\LGpv$ for the Langlands dual group $\Gv$.\footnote{More precisely, the Satake Theorem of \cite{BezFink} gives a differential graded form of $\fg^\ast/G$ and the Kostant slice. To recover the statement above, one must consider some form of mixed sheaves on the affine Grassmannian as in~\cite{riche}, see Remark~\ref{mixed remark}.}

We can also describe Hamiltonian $G$-actions monoidally. Given a Hamiltonian $G$-space $X$, the action of the symplectic groupoid $T^\ast G$
endows $\QC(X)$ with the structure of module category over the the convolution category $\QC(T^\ast G)$. Equivalently, the equivariant moment map
$X/G\to \fgx/G$ makes $\QC(X/G)$ into a module category over $\QC(\fgx/G)$. This equivalence comes from a Morita equivalence
$$(\QC(\fgx/G),\otimes)\modcat\simeq (\QC(T^\ast G),\ast)\modcat,$$ an instance of Gaitsgory's 1-affineness theorem~\cite{1affine},
which in particular identifies the Drinfeld centers of the two categories
$$\cZ(\QC(\fgx/G),\otimes)\simeq \cZ(\QC(T^\ast G),\ast)\simeq \QC( I).$$
Thus the Ng\^o action gives rise to an action of $\QC(J)$ on $\QC(X)$ commuting with the $G$-action and moment map, and hence descending to any Hamiltonian reduction. 

\subsection{Quantum Kostant slice and geometric Satake}\label{quantum Kostant section}
The quantization of $\fgx$ is the algebra $\fU\fg$ or equivalently the (pointed or $E_0$) category $\fU\fg\module$. 
Recall that by the Harish-Chandra isomorphism the adjoint quotient scheme $\fc$ is identified with the spectrum of the center of the enveloping algebra, $$\fc\simeq \Spec \fZ\fg \simeq \fhx\GIT W.$$ 
The quantization of the Kostant section is given by the {\em Whittaker Hecke algebra}, the quantum Hamiltonian reduction $\fU\fg\GIT _\psi \fU\fn:$ the algebra which acts on the space of Whittaker vectors ($\fn$-eigenvectors with eigenvalue $\psi$) universally in any $\fU\fg$-module---in other words, the (principal) finite $\cW$-algebra associated to $\fg$. Kostant~\cite{Kostant Whittaker} then proved that the canonical map $\fZ\fg\to \fU\fg\GIT_\psi N$ is an isomorphism, in particular that the $\cW$-algebra $\fU\fg\GIT _\psi N$ is commutative.

The quantization of $\fgx/G$ is the monoidal (or $E_1$) category $\HC$ of {\em Harish-Chandra bimodules} 
 $\fU\fg$-bimodules integrable for the diagonal action of $G$ (or weakly $G$-equivariant $\fU\fg$-modules). It receives a monoidal functor 
$$\xymatrix{\cZ:=\fZ\fg\module \ar[rr]^-{\Char}&& \HC}$$ 
quantizing the characteristic polynomial map.
Its Drinfeld center $$\cZ(\HC)\simeq \cD(G\adjquot G)$$ is identified with the category of conjugation-equivariant $\cD$-modules on $G$, quantizing sheaves on the inertia stack $\QC( I)$. 

Thanks to the (derived, renormalized, loop rotation equivariant) Geometric Satake theorem of Bezrukavnikov-Finkelberg~\cite{BezFink}, Harish-Chandra bimodules and Whittaker reduction appear out of the equivariant geometry of the affine Grassmannian for the Langlands dual group. In this setup, the category of Harish-Chandra bimodules and its relatives appear with a cohomological degree shift (as is familiar from cyclic homology theory, see in particular the closely related~\cite{conns}, or from the Nekrasov $\Omega$-background in supersymmetric gauge theory~\cite{NW}). In particular, the quantization parameter $\hbar$ appears with cohomological degree two as the equivariant parameter $\C[\hbar]=H^\ast(BS^1)$ for loop rotation; the dual Lie algebra $\fgx$ is replaced by the 2-shifted Poisson variety $\fgx[2]$, which deforms over the $\C[\hbar]$ to the Rees dg-algebra $\fU_\hbar\fg$; and the 1-shifted symplectic stack $\fgx/G$ is replaced by the 3-shifted symplectic stack $\fgx[2]/G$, which deforms to the monoidal category $\HC_\hbar$ of $\fU_\hbar\fg$-Harish-Chandra bimodules.

Geometric Satake gives an equivalence of monoidal categories between $\HC_\hbar$ and the spherical Hecke category $\cH_\hbar=\Dhol(\uGr)$ of $\LGpv  \rtimes \G_m$-equivariant $D$-modules on the affine Grassmannian $\cGr =\LGv/\LGpv$. Moreover, this equivalence intertwines the Kostant-Whittaker action of $\HC_\hbar$ on $\cZ_\hbar = \fZ_\hbar\fg\module$ with the action of $\cH_\hbar$ on $\cR_\hbar := \Dhol(B\LGpv) \simeq H_{\Gv}(pt)\module$. We denote by $$\cK_\hbar=\End_{\cH_\hbar}(\cR_\hbar)$$ the monoidal category of Hecke-linear endomorphisms (compare our general notation of Section~\ref{convolution categories}, where the equivariant Grassmannian is playing the role of the groupoid $\cG$).


\subsection{The $\cW$-category and the quantum Ng\^o action}\label{quantum Ngo section}
Now we explain how the construction of Theorem~\ref{main intro} also gives rise to a quantization of the Ng\^o action (in its cohomologically sheared $\hbar$-form). 

The quantum analog of $J$ is given by the {\em $\cW$-category}, or Whittaker Hecke category of $G$
$$\cWh_\hbar:=End_{\HC_{\hbar}}(\fZ_\hbar \fg\module) \simeq \cD_\hbar(\NGNpsi),$$ given by the $\HC_\hbar$-endomorphisms of the category $\cZ_\hbar=\fZ_\hbar\fg\module$ of Whittaker modules; equivalently, it is the category of $\cD$-modules on $G$ equivariant with respect to the left and right action of $(N,\psi)$\footnote{The Whittaker equation defining $\psi$-twisted equivariance is not homogeneous with respect to the usual filtration on differential operators; one must use the \emph{Kazhdan filtration} to make sense of the Rees algebra constructions--see Remark \ref{remark Kazhdan}.}. 

According to geometric Satake, we have an equivalence of monoidal categories
\[
\cWh_\hbar = \End_{\HC_\hbar}(\cZ_\hbar) \simeq \End_{\cH_\hbar}(\cR_\hbar) = \cK_\hbar
\]
More concretely, $\cWh_\hbar$ is given by modules for the ring of bi-Whittaker differential operators $\fWh_\hbar$, obtained from $\fD_G$ by two-sided Hamiltonian reduction by $N$ at $\psi$, whereas $\cK_\hbar$ is given by modules for $H_\hbar =  H_\ast^{\LGpv\rtimes \CC^\times}(\cGr)$, the equivariant convolution homology ring appearing in \cite{BFM}. The equivalence of monoidal categories above may be interpreted as an isomorphism of bialgebroids $$\fWh_\hbar \simeq H_\hbar.$$


The $\cW$-category provides a deformation quantization of sheaves on the groupscheme $J,$ the bi-Whittaker reduction of $T^\ast G$.
As with all Hecke categories, the $\cW$-category is naturally monoidal. However it is surprising that it is in fact naturally symmetric monoidal (even on the derived level), and that the Ng\^o action quantizes: the construction of Theorem~\ref{main intro} applied to $\LGpv\rtimes \Gm$-equivariant sheaves on the affine Grassmannian $\LGv/\LGpv$ gives rise, in conjunction with the renormalized Satake theorem of~\cite{BezFink}, to Theorem \ref{central action intro} --- a central action of the $\cW$-category on Harish-Chandra bimodules, the {\em quantum Ng\^o action}.

\begin{remark}
The Drinfeld center $\cZ(\cC)$ of any monoidal category $\cC$ is nontrivially braided, so that the analog of Kostant's proof of his theorem fails: the canonical Kostant functor $\cZ(\cD(G))\to \cWh$ is far from an equivalence. In fact $\cWh_\hbar$ is closer to being a ``Lagrangian" in $\cZ(\cD_\hbar(G))$ - a maximal subcategory on which the braiding vanishes.
\end{remark}

\subsection{Spectral decomposition and quantum integrability}\label{integrable section}

One of the fundamental problems in harmonic analysis is spectral decomposition of functions on a symmetric space under
Harish-Chandra's commutative algebra of invariant differential operators, a collection of higher analogs of the Laplace operator for which we seek joint eigenfunctions. We now describe some immediate consequences of Theorem~\ref{central action intro} in this setting.

By a result of~\cite{1affine, dario}, the monoidal category $\HC$ of Harish-Chandra bimodules is Morita equivalent (as a monoidal category) to $\cD$-modules on $G$ with convolution, the ``de Rham group algebra" $(\cD(G),\ast)$.
The Morita equivalence relates a $\cD(G)$-category with its weak $G$-equivariants, and an $\HC$-category with its de-equivariantization:
$$\cD(G)\actson \cM \longleftrightarrow \HC\actson \cM^G, \hskip.3in \HC\actson \cN \longleftrightarrow \cD(G)\actson \left(\cN\otimes_{Rep(G)} Vect\right)$$
 It follows that module categories for $\HC$ are identified with $\cD(G)$-modules, also known as {\em de Rham} or {\em strong} $G$-categories. The theory of de Rham $G$-categories, or the equivalent theory of $\HC$-modules, is a natural realization of the notion of quantum Hamiltonian $G$-space (an algebraic variant of an idea of~\cite{teleman}). Examples include $A\module$ for algebras $A$ acted on by $G$, for which the Lie algebra action is made internal by means of a homomorphism $\mu^*:\fU\fg\to A$, e.g., $A=\fU\fg$ itself or $A=\fD_M$ for a $G$-space $M$. More abstractly the category $\cD(M)$ for any $G$-space $M$ is a de Rham $G$-category. 

For a $G$-space $M$, the composite map $$\xymatrix{\fZ\fg\ar[r]^-{\chi^*} &\fU\fg \ar[r]^-{\mu^*} & \fD_M}$$ provides a family of commuting $G$-invariant differential operators (similarly for any Hamiltonian $G$-algebra $(A,\mu^*)$ as above). These generalize the commuting $G$-invariant differential operators on symmetric spaces introduced by Harish-Chandra, and thanks to $G$-invariance descend to give commuting operators on any quantum Hamiltonian reduction (e.g., on locally symmetric spaces). This provides a source of many quantum integrable systems~\cite{etingof}.
In particular given $\lambda\in \fc\simeq \fhx\GIT W$ we can define the $\lambda$-eigensystem for the Harish-Chandra Laplacians in this setting, the quantum analog of the fibers of the classical Hamiltonians $\chi\circ\mu$. 

However, unlike in the classical setting, quantum Hamiltonian $G$-spaces $\cM$ do not ``live" over $\fc=\Spec(\fZ\fg)$: $\cM$ is {\em not} naturally a module category for $\cZ=\fZ\fg\module$. Thus unlike with spaces of functions, there is no spectral decomposition of $\cM$ over $\fc$: e.g., it does not make sense to ask for a category which is the ``quantum fiber" of $\cM$ over $\lambda\in \fc$. This is a manifestation of the well-known phenomena of shift maps and translation functors: the Harish Chandra systems associated to different $\lambda$ 
can be isomorphic.

\begin{example}
$\bullet$ The eigensystem $M_\lambda$ for the operator $z \frac{d}{dz}$ on $\Cx$ depends on $\lambda$ only up to translation. Indeed the category of $\cD$-modules on $\Cx$ is equivalent (by the Mellin transform) to the category of equivariant sheaves $\cWh_\Cx=\QC(\CC)^{\Z}$, which then acts on $\cM$ for any quantum Hamiltonian $\Cx$-space $\cM$. 

$\bullet$ The $G$-category $\fU\fg_\lambda\module$ of $\fg$-modules with a fixed central character depends on $\lambda$ only up to the action of translation functors. The corresponding $W_{\aff}$-orbit $[\lambda]\in \fhx\GIT W_{\aff}$ is an invariant of this $G$-category, but not $\lambda$ itself.
\end{example}

Thus we might instead hope to spectrally decompose quantum Hamiltonian $G$-spaces over $\fhx\GIT W^{\aff}$, and indeed our main result gives such a decomposition:

\begin{corollary}[Theorem~\ref{central action intro}]\label{quantum ham corollary} For any $\cD_\hbar(G)$-module $\cM$, there is an action of the tensor category $\cWh_\hbar\actson \cM$ commuting with the $\cD_\hbar(G)$ action (and hence descending to any quantum Hamiltonian reduction such as $\cD_\hbar(K\backslash G/ H)$ and $(\fg,K)\module_\hbar$). 
\end{corollary}

In other words, quantum Hamiltonian $G$-spaces may be spectrally decomposed under the ``categorical Harish-Chandra operators", i.e., the action of the commuting operators provided by the quantum Ng\^o map $\cWh_\hbar\simeq \QC^!(\fhx\GIT W^{\aff})\to \cD_\hbar(G/G)=\cZ(\cD_\hbar(G))$.
In other words, quantum Hamiltonian $G$-spaces may be spectrally decomposed under the ``categorical Harish-Chandra operators", i.e., the action of the commuting operators provided by the quantum Ng\^o map $\cWh_\hbar\simeq \QC(\fhx\GIT W^{\aff})\to \cD_\hbar(G/G)=\cZ(\cD_\hbar(G))$.

\begin{remark}[Integrating quantum Hamiltonian systems by a quantum groupoid]\label{quantum groupoid remark}
This categorical statement has a more concerete ``function-level" interpretation as follows in the spirit of Proposition~\ref{classical integration}. For a $G$-space $M$, the algebra $\fD_{\hbar,M}$ is a $\fD_{\hbar,G}$-comodule in $\fU\fg$-modules. Concretely, $\fD_{\hbar,M}$ carries an action of $G$ and an action of $\fU_\hbar\fg$ (from the moment map), making it a Harish-Chandra bimodule.
The function level quantization of the action of Proposition~\ref{classical integration} endows $\fD_{\hbar,M}$ with the structure of $\fWh_\hbar$-comodule--concretely the coaction map is given by
$$\fD_{\hbar,M}\rightarrow   \fD_{\hbar,G} \otimes_{\fU_\hbar\fg} \fD_{\hbar,M}  \rightarrow (\fWh_\hbar \otimes_{\fZ_\hbar\fg} \fU_\hbar\fg) \otimes_{\fU_\hbar\fg} \fD_{\hbar,M} =\fWh_\hbar \otimes_{\fZ_\hbar\fg} \fD_{\hbar,M} $$
Here we use the map $\fD_{\hbar,G}\rightarrow \fWh_\hbar\otimes_{\fZ\fg} \fU_\hbar \fg $ which arises from the action of $\fD_{\hbar,G}$ on $\fWh_\hbar\otimes_{\fZ\fg} \fU\fg$ defined by the Ng\^o action, and the Harish-Chandra bimodule structure map $\fD_{\hbar,M} \rightarrow \fD_{\hbar,M} \otimes_{\fU\fg} \fD_{\hbar,G}$ above.
This comodule structure on $\fD_{\hbar,M}$ underlies a structure of algebra in $\fWh_\hbar$-comodules, i.e., an action of $\fWh_\hbar$ on $\fD_{\hbar,M}$ as a {\em commutative quantum groupoid} over $\fZ_\hbar\fg$ (cocommutative Hopf algebroid, see~\cite{lu},~\cite{Bohm} and references therein), which is the natural quantum analog of the integration of Hamiltonian flows provided by Proposition~\ref{classical integration}. We postpone a discussion of Hopf algebroids in the $\oo$-categorical setting (and thus a precise formulation of this claim) to a future paper, though see Remark~\ref{abelian vs derived}.
\end{remark}

\begin{remark}[Conjectural Picture: Fukaya Quantization of the Ng\^o correspondence]
The Ng\^o correspondence~\ref{Ngo correspondence} has a Lagrangian structure, and defines a central action of the commutative symplectic groupoid $J$ on $T^\ast G$. This suggests a natural setting for quantization of the Ng\^o action: given a ``deformation quantization theory", a (lax) symmetric monoidal functor $\cF$ from the Lagrangian correspondence category of symplectic manifolds to dg categories, we obtain a symmetric monoidal category $\cF(J)$ together with a central action on the monoidal category $\cF(T^\ast G)$ associated to the symplectic groupoid $T^\ast G$ integrating $\fgx$. Informally speaking one expects suitable versions of the Fukaya category to define such a functor (as we learned from Teleman, Gualtieri and Pascaleff). In particular $\cF(T^\ast G)\actson\cF(M)$ for a Hamiltonian $G$-space $M$ (as explained in~\cite[Conjecture 2.9]{teleman}). Thus one would expect the Ng\^o action to define an action $$\cF(T^\ast G)\actson \cF(M) \onacts \cF(J)$$ of the Fukaya category of $J$, with the symmetric monoidal structure coming from convolution, by $G$-automorphisms of the Fukaya category of any Hamiltonian $G$-space. Moreover mirror symmetry should identify $\cF(J)$ in terms of the B-model on $H^\vee\GIT W$, providing a notion of spectral decomposition of $G$-categories. It would be very interesting to understand the relation of this picture to the remarkable comprehensive character theory for $G$-A-models developed by Teleman~\cite{teleman}. Note that Teleman's theory prominently features the (unquantized) groupscheme $J$ for the {\em Langlands dual} group, as the target for a spectral decomposition of a smarter ``decompleted" form of $\cF(T^\ast G)$-modules.
\end{remark}

\subsection{The quantum Ng\^o action}\label{parabolic induction section}
In this section, we give a number of conjectural interpretations of the functor $\Ngo: \cWh \to \cD(G)^G$ in terms of more familiar constructions arising in the theory of character sheaves. 

\subsubsection{The horocycle transform and parabolic induction/restriction}
First, let us consider the commutative diagram
\[
\xymatrix{
G\adjquot G  &\ar[l]_a G\adjquot B \ar[r]^b& Hor \\
G\adjquot G  \ar@{=}[u] &\ar[l]^r B\adjquot B \ar@{^{(}->}[u] \ar[r]_{s}& H\adjquot B \ar@{^{(}->}[u]
}
\]
where 
\[
Hor = (\quot NGN)\adjquot H = \quot{G}{(G/N\times G/N)}{H}
\]
is the horocycle stack.
This diagram gives rise to two pairs of adjoint functors, both of which have been studied extensively in the context of character sheaves (see e.g. \cite{lusztig character, Ginzburg admissible, Ginzburg parabolic}): we have the horocycle and character functors 
\[
\xymatrix{
\hc = b_\ast a^! :  \cD(G)^G \ar[r] & \ar[l] \cD(Hor) : a_\ast b^! = ch
}
\]
and the parabolic restriction and induction functors
\[
\xymatrix{
\Res = s_\ast r^![\dim N]: \cD(G)^G \ar[r] & \ar[l] \cD(H)^B \simeq \cD(H)^H : r_\ast s^![-\dim N] = \Ind
}
\]
These functors are closely related, but have different features. For example:
\begin{itemize}
\item The composite of $\hc$ followed by restricting to the diagonal $H\adjquot B$ in the Horocycle space is equivalent to $\Res$ (up to a shift). 
\item
The category $\cD(Y)$ carries a monoidal structure coming from convolution, and the functor $\hc$ is naturally monoidal. On the other hand, $\Res$ does not intertwine the convolution structures in general. 
\item The functor $\hc$ is easily seen to be conservative by an argument of Mirkovic and Vilonen \cite{MV} (the composite $ch \circ \hc$ is given by convolution with the Springer sheaf; in particular, the identity functor is a direct summand). On the other hand, the functor $\Res$ is only conservative in the case where no Levi subgroup of $G$ carries a cuspidal local system in the sense of Lusztig \cite{lusztig_intersection_1984} (this is the case for $G=GL_n$, for example, but not for $G=SL_2$). 
\item The functors $\Ind$ and $\Res$ restrict to exact functors on the level of abelian categories, but $\hc$ and $ch$ do not, in general.
\end{itemize}

\subsubsection{Springer theory and quantum Hamiltonian reduction}
 In \cite{Gun1, Gun2}, the category $\cD(\fg)^G$ is studied, along with the analogous functors to $\Res$ and $\Ind$ in the Lie algebra setting (which we continue to denote $\Res$ and $\Ind$). The category $\cD(\fg)^G$ is shown to decompose in to blocks indexed by cuspidal data. One such block is the Springer block; this can be described as the subcategory of $\cD(\fg)^G$ generated by the essential image of the functor $\Ind$. It is shown that the functor $\Res$ upgrades to an exact equivalence of abelian categories
\[
\ls{W}{\Res} : \cM(\fg)^G_{Spr} \xrightarrow{\sim} \cM(\fh)^W: \ls{W}{\Ind}
\]
on the Springer block. The inverse functor $\ls{W}{\Ind}$ to $\ls{W}{\Res}$ takes a $W$-equivariant object $\fM$ of $\cM(\fh)$ to the $W$-invariants of $\Ind(\fM)$. (This result can be thought of as an extension of the Springer correspondence, which identifies a block of the category of equivariant $\cD$-modules with support on the nilpotent cone with representations of $W$.) 

To state the conjectures below, we will assume the analogous results to \cite{Gun1, Gun2} in the setting of equivariant $\cD$-modules on $G$ (which the second named author intends to address in future work). In particular, we will assume we have an equivalence:
\[
\xymatrix{
\ls{W}{\Res} : \cM(G)^G_{Spr} \ar@{<->}[r]^\sim & \cM(H)^W: \ls{W}{\Ind}
}
\]
In particular, there is an extension of this equivalence to a functor (no longer fully faithful) on the level of dg-categories\footnote{Note that the source category $\cD(H)^W$ is the dg-derived category of its heart; though this is not the case for the target category, there is still a canonical functor from the dg-derived category of $\cM(G)^G$ to $\cD(G)^G$.}
\[
\xymatrix{
\ls{W}{\Ind}: \cD(H)^W \ar[r] & \cD(G)^G
}
\]

Now let us recall the functor of quantum Hamiltonian reduction and the Harish-Chandra homorphism. Consider the object
\[
\fD_{G\adjquot G} = \fD_G/\fD_G\ad(\fg) 
\]
which represents the functor of quantum Hamiltonian reduction. There is an exact functor of abelian categories
\[
\QHR:\cM(G)^G \to (\fD_{G\adjquot G})^G\module^\heartsuit
\]
which takes a strongly equivariant $\fD_G$-module to its $G$-invariants; it has a fully faithful left adjoint $\QHR^L$, which we extend to a functor on derived categories. By results of Levasseur and Stafford \cite{LS1, LS2} (or rather, the natural analogue in the group setting), the Harish-Chandra homomorphism defines an isomorphism of rings 
\[
rad:(\fD_{G\adjquot G})^G \simeq (\fD_H)^W.
\]
Note also there is a Morita equivalence between $\fD_H \# W$ and its spherical subalgebra $(\fD_H)^W$ which takes a $\fD_H \# W$-module to its $W$-invariants. These functors are compatible in the sense that there is a commutative diagram:
\[
\xymatrix{
\cD(H)^W \ar[d]^\wr_{(-)^W} \ar@/^1pc/[rrd]^{\ls{W}{\Ind}} && \\
(\fD_H)^W\module  & & \cD(G)^G \\
 (\fD_{G\adjquot G})^G\module \ar[u]_\wr^{rad} \ar@/_1pc/[rru]_{\QHR^L} && 
}
\]

\begin{remark}
In the case $G=GL_n$, there are no non-trivial cuspidal data (equivalently, $\QHR$ is conservative), and thus we have an equivalence of abelian categories
\[
\cM(G)^G \simeq \cM(H)^W
\]
However, note that this equivalence does not respect monoidal structures in general.
\end{remark}

\subsubsection{The nil-DAHA and sheaves on the coarse quotient} 
Let $W^\aff \simeq \Lambda_H \rtimes W$ denote the extended affine Weyl group, which acts on $\fh^\ast$ with $\Lambda_H$ acting by translation. The (degenerate) \emph{double affine nil-Hecke algebra} (nil-DAHA) is defined to be the subring
\[
\Nil_{W^\aff} \subseteq \left(\Sym(\fh)[\alpha^{-1} \mid \alpha \in \Phi]\right) \rtimes W^\aff
\]
generated by $\Sym(\fh)$ and the Demazure operators $\alpha^{-1}(1-s_\alpha)$ associated to affine simple roots $\alpha$ (see \cite[Chapter 4.3]{schubert book} for further details). The ring $\Sym(\fh) \rtimes W^\aff$ sits as a subring of $\Nil_{W^\aff}$; in particular there is a fully faithful functor
\[
\forg: \Nil_{W^\aff} \module \hookrightarrow \fD_H \rtimes W\module \simeq \cD(H)^W
\]
given by forgetting the action of the nil-Hecke algebra to the subring (the fully faithful property follows from the fact that both rings sit inside a common localization). Similarly, the spherical subalgebra $\Nil_{W^\aff}^{sph}$ is Morita equivalent to $\Nil_{W^\aff}$ and contains a copy of $(\fD_H)^W$.

More geometrically, the nil-DAHA represents the descent data for an object of $\QC(\fh^\ast)$ to the coarse quotient $\fh^\ast \GIT W^\aff$, whereas $\Sym(\fh)\rtimes W^\aff$ represents descent data to the stack quotient $\fh^\ast/W^\aff$ (see \cite{lonergan}).

The results of Ginzburg \cite{ginzburg whittaker} and Lonergan \cite{lonergan} identify the spherical nil-DAHA $\Nil_{W^\aff}^{sph}$ with bi-Whittaker differential operators $\fWh$, or alternatively, the loop rotation equivariant homology convolution algebra of the Langlands dual affine Grassmannian (with $\hbar$ formally set to 1). In fact, one can check that this is an isomorphism of bialgebroids, and thus there is an equivalence of monoidal categories
\[
\Nil_{W^\aff}\module \simeq \Nil^{sph}_{W^\aff} \module \simeq \fWh\module \simeq \cWh
\]
In particular, there is a copy of $Nil\module$ sitting inside $\cD(H)^W$, which we denote by $\cD(H)^W_{Nil}$. The following conjecture states that the Ng\^o functor is compatible with the functors given by Springer theory and the Harish-Chandra homomorphism.

\begin{conjecture}\label{quantum ngo induction}
There is a commutative diagram:
\[
\xymatrix{
\QC(\fh^\ast \GIT W^\aff) \ar@{<->}[d]^\wr \ar[r]^{\pi^\ast} & \QC(\fh^\ast)^{W^\aff} \ar@{<->}[d]^\wr \\
\Nil_{W^\aff}\module \ar[r]  \ar@{<->}[d]^\wr &  \ar@{<->}[d]^\wr \fD_H \rtimes W\module \ar[d] \ar@/^3pc/[dd]^{\ls{W}{\Ind}} \\
\Nil^{sph}_{W^\aff}\module \ar[r]  \ar@{<->}[d]^\wr & (\fD_H)^W\module \ar[d]_{\QHR^L} \\
\cWh \ar[r]^{\Ngo} & \cD(G)^G 
}
\]
\end{conjecture}

\begin{remark}
A remarkable feature of this diagram is that, while the functor 
\[
\ls{W}{\Ind}:\cD(H)^W \to \cD(G)^G
\]
relates two braided monoidal categories, it does not carry a monoidal structure; however, according to the conjecture, $\ls{W}{\Ind}$ is braided monoidal upon restriction to the full subcategory given by modules for the nil-Hecke algebra.
\end{remark}

\begin{remark}\label{remark quantum affine}
Recall that the Ng\^o functor arises from the Lagrangian correspondence (read from left to right)
\[
\xymatrix{
J & \ar[l] \chi^\ast J \ar[r] & I= T^\ast(G\adjquot G)
}
\]
whereas the functor $\Whit$ arises from the Lagrangian correspondence (read from right to left)
\[
\xymatrix{
J & \ar[l] \kappa^\ast(I) \ar[r] & I = T^\ast(G\adjquot G)
}
\]
While these diagrams are manifestly different in general in particular (in particular, the classical Ng\^o functor is not adjoint to the Whittaker functor), both diagrams have isomorphic affinizations:
\begin{equation}\label{affine diagram}
\xymatrix{J & \ar[l] J \simeq (\chi^\ast J)^{aff} \ar[r] & I^{aff} \simeq (T^\ast H) \GIT W}
\end{equation}
Diagram \ref{affine diagram} corresponds to an inclusion of rings
\[\xymatrix{
\C[J] & \ar[l] \C[T^\ast H]^W
}\]
which quantizes to
\[
\xymatrix{
\fWh \simeq \Nil_{W^\aff}^{sph} & \ar[l] (\fD_H)^W
}
\]
Thus, on the level of affinization, the functors $\Ngo$ and $\Whit$ correspond to the forgetful functor and the base change functor associated to the above inclusion of rings (in particular, they form an adjoint pair). It is remarkable that, while the stack $T^\ast(G\adjquot G)$ is far from affine, it's quantization is almost affine: the abelian category of $(\fD_H)^W$-modules (the quantum affinization) sits as a full subcategory (in fact, a direct summand) of $\cM(G\adjquot G)$ (and in the case $G=GL_n$, the two categories are equivalent, i.e. $T^\ast(G\adjquot G)$ is quantum affine). This explains the simpler form of the Ng\^o and Whittaker functors appearing in Conjecture \ref{quantum ngo induction}.
\end{remark}

\subsubsection{Very central $\cD$-modules}
The following definition was given in the PhD thesis of the second named author.
\begin{definition}\label{very central}
We say that an object $\fM \in \cM(G)^G$ is \emph{very central} if $\hc(\fM)$ is supported on the diagonal substack $H\adjquot B \subseteq Hor$.
\end{definition}
 Note that if $\fM$ is very central, then $\hc(\fM)$ is identified with the parabolic restriction $\Res(\fM)$. In particular, restricting $\ls{W}{\Res}$ to $\cM(G)^G_{vc}$ defines a fully faithful monoidal functor to the symmetric monoidal abelian category $\cM(H)^W$.

\begin{remark}
At the level of abelian categories, the functor $\hc$ takes an equivariant $\fD_G$-module $\fM$ to its $N$-average $(G\to G/N)_\ast \fM$. The very central property means that this $N$-average is supported on $H=B/N \subseteq G/N$. This property has been studied in \cite{Chen}.
\end{remark} 

\begin{conjecture}
\begin{enumerate}
\item The Ng\^o functor defines a fully faithful braided monoidal functor on abelian categories, whose essential image is given by $\cM(G)^G_{vc}$.
\item The essential image of $\ls{W}{\Res}$ restricted to $\cM(G)^G_{vc}$ is given by $\cM(H)^W_{Nil}$.
\end{enumerate}
\end{conjecture}
Note that the two statements are mutually equivalent given Conjecture \ref{quantum ngo induction}.

\subsubsection{Twisted Harish-Chandra systems and almost idempotent character sheaves}
Recall that there is an equivalence $\cWh \simeq \QC(\fh^\ast\GIT W^\aff)$; thus the Ng\^o functor defines a collection of orthogonal almost-idempotent objects of $\cD(G)^G$, given by the image of skyscraper sheaves of points $\cO_{[\theta]}$, where $[\theta]$ denotes a point of $\fh^\ast \GIT W^\aff$ corresponding to $\theta \in \fh^\ast\GIT W$. These objects are expected to be certain twisted forms of the Harish-Chandra system associated to $\theta$. 

To explain this more precisely, let $\cWh_{[\theta]}$ denote the category of\emph{admissible $\fWh$-modules with central character $[\theta]$}, i.e. the full monoidal subcategory of $\cWh$ consisting of objects whose (set-theoretic) support with respect to $\fZ\fg \simeq \C[\fh^\ast]^W$ is contained in $[\theta]$. As the fibers of $\fh^\ast \GIT W \to \fh^\ast \GIT W^\aff$ are discrete, there is a symmetric monoidal equivalence with sheaves on $\fc$ set-theoretically supported at $\theta$ (for any choice of lift $\theta$ of $[\theta]$):
\[
\cWh_{[\theta]} \simeq \fZ\fg\module_{\theta} \simeq \QC(\fc)_{\theta}
\]

Abstractly this category is symmetric monoidally equivalent to $\QC(\A^r)_{(0)}$, the subcategory of modules for a symmetric algebra generated by the augmentation.\footnote{In particular, the categories $\cWh_{[\theta]}$ are equivalent for all values of $\theta$. This result is not immediately apparent from the definition, and somewhat surprising given how the category of character sheaves $\cD(G)^G_{[\theta]}$ varies with the central character $[\theta]$.} By Koszul duality, this in turn is isomorphic to $\QC(\A^r[-1]) \simeq L\module$, where $L = \Sym(\C^r[1])$. It follows that the objects $\cO_{[\theta]}$ are orthogonal and almost idempotent with respect to the monoidal structure on $\cWh$ (i.e. idempotent up to a ``scalar'' given by the dg-vector space $L$). There is also an (actual) derived idempotent $\widecheck{\cO}_{[\theta]} \in \QC(\fh^\ast\GIT W^\aff)_{[\theta]}$ which corresponds to the augmentation module in $L\module$, or the $\cD$-module of delta functions in $\A^r$, considered as an object of $\QC(\A^r)_{(0)}$.

It follows that the image of  $\cO_{[\theta]}$ (respectively $\widecheck{\cO}_{[\theta]}$) are almost idempotent (respectively idempotent) objects in the monoidal category $\cD(G)^G$. We denote these objects by $\fE_{[\theta]}$ (respectively $\widecheck{\fE}_{[\theta]}$. 

Recall \cite{Ginzburg admissible} that the category of \emph{character sheaves} (or \emph{admissible modules}) with central character $[\theta]$ is the subcategory of $\cD(G)^G$ consisting of $\fD_G$-modules whose $\fZ\fg$-support is contained in $[\theta]$.\footnote{In this paper, we do not require character sheaves to be semisimple, or even coherent as $D$-modules.} It follows directly that the objects $\fE_{[\theta]}$ and $\widecheck{\fE}_{[\theta]}$ are examples of character sheaves with central character $[\theta]$; in fact, $\widecheck{\fE}_{[\theta]}$ is the unit object in the category of character sheaves.

These objects may be described more explicitly, assuming Conjecture \ref{quantum ngo induction}. Recall that we have a sequence of functors
\[
\cWh \simeq \QC(\fh^\ast\GIT W^\aff) \to \QC(\fh^\ast)^{W^\aff} \simeq \cD(H)^W
\]
Given a skyscraper sheaf $\cO_{[\theta]}$ in $\QC(\fh^\ast \GIT W^\aff)$, let $\fL_{[\theta]}$ denote the corresponding object of $\cD(H)^W$. Uniwinding the definitions, we see that $\fL_{[\theta]}$ is a certain $W$-equivariant flat connection of rank $W$ on $H$; for example, $\fL_{[0]}$ is an indecomposible unipotent flat connection on $H$, where the invariant differential operators $\Sym(\fh)$ act on a frame of sections as the module of coinvariants $\Sym(\fh)/\Sym(\fh)^W_+$. Similarly, the object $\widecheck{\fL}_{[\lambda]}$ is a certain infinite rank flat connection; for example, $\widecheck{\fL}_{[0]}$ is an ind-unipotent flat connection, which has a frame isomorphic to $\Sym(\fh^\ast) = \C[\fh]$, where the $\Sym(\fh)$ action is via constant coefficient differential operators. Thus we obtain the following:

\begin{proposition}
Assume Conjecture \ref{quantum ngo induction}; then we have almost idempotent objects
\[
\fE_{[ \theta]} \simeq \ls{W}{\Ind}(\fL_{[\theta]})
\]
and idemptotent objects
\[
\widecheck{\fE}_{[\theta]} \simeq \ls{W}{\Ind}(\widecheck{\fL}_{[\theta]})
\]
of $\cD(G)^G$, for each $[\theta] \in \fhx\GIT W^\aff$.
\end{proposition}
\begin{remark}
The objects $\widecheck{\fE}_{[\theta]}$ were studied recently by Chen; see, for example, Theorem 3.8 in \cite{Chen}, where it was shown that the $\fE_{[\theta]}$ were very central in the sense of Definition \ref{very central}.
\end{remark}

Finally, recall the \emph{Harish-Chandra system}
\[
\fM_{0} := \fD_{G}/\fD_G\left(\ad(\fg) + \fZ\fg_+\right)
\]
The fundmental results of Hotta and Kashiwara~\cite{HK} identify the Harish-Chandra system with the (Grothendieck)-Springer sheaf 
\[
\Ind(\fO_{0})\simeq \ls{W}{\Ind}(\C[W] \otimes \fO_0)
\]
where $\fO_{0}$ is the trivial rank one flat connection on $H$. Similarly, one can define $\fM_\theta$ for any $\theta \in \Spec(\fZ\fg)$, and there is an analogous description in terms of parabolic induction. Note that $\C[W] \otimes \fO_0$ is precisley the semisimplification of the $W$-equivariant flat connection $\fL_{[0]}$ (and there is an analogous statement for any $\theta$). Thus the Harish-Chandra system $\fM_{\theta}$ is the semisimplification of the almost idempotent object $\fE_{[\theta]}$. This justifies the name twisted Harish-Chandra system.

\subsection{Kac-Moody Groups and Coxeter Systems}\label{KM section}
Now let us explore some other examples of our construction of central actions on convolution categories from Theorem \ref{main intro}. We will have two closely related classes of examples: one topological, arising from Kac-Moody groups, and another combinatorial, associated to Coxeter systems. These examples are related to the motivating example, by taking the affine Kac-Moody group associated to $\Gv$.

\subsubsection{Toy examples}
Before discussing further, we give two examples to illustrate the basic principle of Theorem \ref{main intro}: for a groupoid $\cG$ acting on a space $X$ with quotient $Y$, $\cG$-equivariant sheaves on $X$, i.e., sheaves on $Y$, act centrally on modules for the convolution category of sheaves on $\cG$, i.e., sheaves of categories on $Y$. 

\begin{example} 
Let $\pi:X\to Y$ denote a map of finite sets, and $\cG=\XYX$. In this case the convolution algebra $(H=k[\cG],\ast)$ is the algebra of $|X|$ by $|X|$ block-diagonal matrices (with blocks labeled by $Y$), which is Morita equivalent to the commutative algebra $k[Y]$. We also consider the convolution category $(\cH=Vect(\XYX),\ast)$. In this case the inclusion of block-scalar matrices $Vect(Y)\hookrightarrow Vect(\XYX)$ identifies 
	$$\xymatrix{H\module\simeq Vect(Y)\ar[rr]^-{\sim}&& \cZ(Vect(\XYX))}$$ with the Drinfeld center of $(Vect(\XYX),\ast)$, categorifying the familiar identification of block-scalar matrices $k[Y]$ as the center of block-diagonal matrices $k[\XYX]$. 
\end{example}	
\begin{example} Let $G$ denote a finite group, and $X=pt\to Y=BG$, so that $G\simeq \XYX$. In this case the convolution algebra $H=(\CC[G],\ast)$ is the group algebra, and $H\module=Rep(G)$ is the symmetric monoidal category of representations. The Drinfeld center of the monoidal category $(Vect(G),\ast)$ is now the braided tensor category $Vect(G/G)$, which contains $Rep(G)\simeq Vect(pt/G)$ as the tensor subcategory of equivariant vector bundles supported on the identity. The latter is in fact a Lagrangian subcategory of $Vect(G/G)$ in the sense of~\cite{DGNO}. We expect our general construction provides (derived analogues of) Lagrangian subcategories as well. The action of $Vect(G)$ on a $Vect(pt)=Vect$ induces an action of its center
	$$\xymatrix{\cZ(Vect(G))=Vect(G/G)\ar[rr]&&End_{Vect(G)}(Vect)\simeq Rep(G)}$$
	which provides the desired left inverse.
\end{example}

\subsubsection{Kac-Moody groups}
Let $\bG$ denote a simply-connected Kac-Moody group, with Borel subgroup $\bB$ (or more generally parabolic subgroup $\bf P$).
The flag variety $\bG/\bB$ is an ind-projective ind-scheme of ind-finite type~\cite{Mathieu, Kumar}. We let $\cG_{\bG,\bB}=\bB\backslash \bG/\bB$ denote the corresponding
``Hecke" groupoid acting on $X_{\bG,\bB} = pt/\bB$.
In this setting, the convolution algebra $H_{\bG,\bB}$ is given by the 
equivariant homology ring $H_\ast(\quot{\bB}{\bG}{\bB})$ (considered as a dg-ring). 
The Kostant category $\cK_{\bG,\bB}=H_{\bG,\bB}\module$ has a symmetric monoidal structure arising from the ``cup coproduct'' on $H_{\bG,\bB}$. The convolution category $\cH_{\bG,\bB}$ is the (renormalized) Iwahori-Hecke category 
$\Dhol(\bB\backslash \bG/\bB)$ of equivariant ind-holonomic $\cD$-modules (or ind-constructible sheaves) on the affine flag variety. 

Theorem \ref{main intro} applies in this setting, giving the following:

\begin{theorem}\label{Kac-Moody intro}
There is a natural $E_2$ functor from the symmetric monoidal Kostant category $\cK_{\bG,\bB}$ to the center $\cZ(\cH_{\bG,\bB})$ of the Iwahori-Hecke category, together with a monoidal right inverse (and likewise for any parabolic $\bf P$ of $\bG$). Thus we have an instance of Diagram~\ref{basic diagram}:

\[
\xymatrix{
H_\ast(\quot \bB \bG \bB)\module\ar[r]_{E_2}\ar[d]_-{E_\infty}& \cZ(\Dhol(\quot \bB\bG\bB)) \ar@/_1pc/_-{E_1}[l]  \ar[d]^{E_1}\\
	H^\ast(pt/\bB)\module \ar[r]_{E_1} & \Dhol(\quot \bB\bG\bB)
}
\]

\end{theorem}

The objects appearing in the theorem carry combinatorial realizations in terms of the Coxeter system $(\mathbf{h},\mathbf{W})$ associated to $\bG$. Namely, the homology convolution algebra $H_{\bG,\bB}$ is isomorphic to the Kostant-Kumar nil-Hecke algebra of the Coxeter system  (see~\cite{KK,Kumar, Arabia,schubert book,ginzburg hecke}). Analogously, the Iwahori-Hecke category $\cH_{\bG,\bB}$ can be interpreted in terms of Soergel bimodules for $(\mathbf{h},\mathbf{W})$.

\begin{example}[Finite case]
	Consider the case where $\bG$  is a reductive algebraic group, so $(\bh,\bW)$ is a finite Coxeter system. In this case $H_{\bG,\bB}$ is the finite nil-Hecke algebra, acting on $\C[\bh]\module$ by Demazure operators. The category of $H_{\bG,\bB}$-modules is identified with $\C[\bh]^{\bW}\module$, or in other words sheaves the coarse quotient $\bh\GIT\bW$ of $\bh$ by $\bW$. The geometric setting is a differential-graded version of the combinatorial; forgetting about grading, we have $\C[\bh] = H^\ast_{\bB}(pt)$ and $\C[\bh]^\bW = H^\ast_{\bG}(pt)$. The result of Theorem \ref{main intro} is simply the linearity of the finite Hecke category $\cH=\Dhol(B\backslash G/B)$ over the $G$-equivariant cohomology ring.
\end{example}  

\begin{example}[$\cD$-modules on a reductive group]
Our main application, the quantum Ng\^o action, constructs a central action on the monoidal category $\cD(G)$ (or its Morita equivalent realization, $\HC$) via its Langlands dual realization on the loop Grassmannian. We can also apply the construction verbatim to $\cD(G)$, taking $\cG=G$ as a groupoid acting on $X=pt$ (a variant of the previous example with the equivariant flag variety as a groupoid on $pt/B$). In this case we find a central action of the Kostant category $\cK=H_\ast(G)\module$ on $\cD(G)$, which again is a Koszul dual form of linearity over the $G$-equivariant cohomology ring. This form of the Kostant category is manifestly different from (and less interesting than) the Ng\^o action of Whittaker $\cD$-modules; this example clearly demonstrates that our central actions depend on the presentation as a convolution category (rather than being intrinsic invariants of the monoidal category).
\end{example}

\subsubsection{Coxeter groups}
More generally Theorem~\ref{main intro} has a realization in the setting of a Coxeter group $\bW$ with reflection representation $\bh$ (for example, $\bW$ could be the Weyl group of $\bG$ and $\bh$ the Cartan). For $w\in \bW$ we let $\Gamma_w\subset \mathbf{h}\times \mathbf{h}$ denote the graph of the corresponding reflection. Let $$\Gamma_\bW=\coprod_{w\in \bW} \Gamma_w.$$ Then $\Gamma_\bW$ is an ind-proper groupoid acting on the scheme $\mathbf{h}$. This is the equivalence relation underlying the action of $\bW$ on $\mathbf{h}$ -- i.e., $\Gamma_\bW$ is the {\em adjacency groupoid} of $\bW\actson \mathbf{h}$ in the language of~\cite{lonergan}. We may still consider the (non-representable) quotient $\mathbf{h}/\Gamma_\bW$, i.e. the coarse (set-theoretic rather than stack theoretic) quotient, which we still denote $\mathbf{h}\GIT\bW$. 
Let $\omega(\Gamma_\bW)$ denote the convolution algebra of distributions, i.e. global sections of the Serre-dualizing complex on the singular ind-variety $\Gamma_{\mathbf{W}}$. On the other hand, 
it follows from the results of Lonergan~\cite{lonergan,lonergan2} that the algebra $\omega(\Gamma_\bW)$ is isomorphic to the nil-Hecke algebra
$$\omega(\Gamma_\bW)\simeq H_{\bh,\bW}$$
It follows from ind-proper descent~\cite{GR} that the category $\cK_{\mathbf{h},\mathbf{W}} = \omega(\Gamma_\bW)\module$ is equivalent to ind-coherent sheaves on $\mathbf{h}\GIT\mathbf{W}$, so that we have an equivalence
$$H_{\bh,\bW}\module\simeq \QC^!(\bh\GIT\bW).$$

For the Hecke category $\cH_{\mathbf{h},\bW}$ we may take ind-coherent sheaves $\QC^!(\Gamma_\bW)$ on the adjecency groupoid under convolution. Once again, Theorem \ref{main intro} applies in this setting, giving a diagram of the form Diagram~\ref{basic diagram}.

\begin{theorem}\label{Coxeter intro}
	There is a natural symmetric monoidal structure on modules $\cK_{\bh,\bW}=H_{\bh,\bW}\module$ for the nil-Hecke algebra compatible with the forgetful functor to $\CC[\mathbf{h}]$, and the action $\CC[\mathbf{h}]\to \cH_{\bh,\bW}$ on the Coxeter Hecke category lifts to a central action $\fz:H_W\module\to \cZ(\cH_W)$ with a monoidal right inverse $\fa$. Thus we have an instance of Diagram~\ref{basic diagram}:

\[
\xymatrix{
	H_{\bh,\bW}\module\ar[r]_{E_2}\ar[d]_-{E_\infty}& \cZ(\QC^!(\Gamma_\bW)) \ar@/_1pc/_-{E_1}[l]  \ar[d]^{E_1}\\
	\C[\bh]\module \ar[r]_{E_1} & \QC^!(\Gamma_\bW)
}
\]

\end{theorem}

\begin{remark}\label{dg vs graded remark}
	Making the connection between the Kac-Moody story and the Coxeter story more precise requires a number of modifications. The first issue arises from the fact that the convolution algebra $H_{\bG,\bB}$ in the Kac-Moody set-up is really a dg-algebra (and the corresponding $\cK_{\bG,\bB}$ is given by dg-modules), whereas in the Coxeter set-up, the nil-Hecke algebra $H_{\bh,\bW}$ is considered to be an ordinary algebra (and $\cK_{\bh,\bW}$ is its derived category of modules). This issue is fixed by considering an external grading on the convolution algebra and (dg-)modules with a compatible external grading. The external grading allows for a ``shearing'' equivalence between the two $\cK$-categories. On the level of convolution categories, adding the external grading corresponds to considering a mixed version of the Iwahori-Hecke category; this mixed category is equivalent to chain complexes of graded Soergel bimodules for $(\bh,\bW)$, which can be thought of as a certain modification of the combinatorial Hecke category $\cH_{\bW,\bh}$. Unfortunately, this mixed set-up does not seem to fall so neatly in to the set-up of Theorem \ref{main intro}. 
	\end{remark}

\begin{example}[Spherical affine case]

Let us explain how to connect the examples discussed in this section with our main motivation. We take for $\bG$ the affine Kac-Moody group associated to a reductive group $\Gv$, i.e., the extended form of the loop group $\LGv$, and for the parabolic $\bP$  the maximal parabolic corresponding to the complement of the ``extra'' node in the extended Dynkin diagram, i.e., the extended form of $\LGpv$. In particular $\bW=W^{\aff}$ is the affine Weyl group. In this way, the Iwahori-Hecke category $\cH_{\bG,\bP}$ is replaced by the spherical Hecke category (the renormalized Satake category). Similarly, the convolution algebra $H_{\bG,\bB}$ is replaced by its spherical subalgebra. In fact, there is a Morita equivalence between the nil-Hecke algebra and its spherical subalgebra~\cite{webster} so the $\cK$ categories are the same in the Iwahori and spherical settings  (see also~\cite{ginzburg whittaker} which constructs the equivalence from the Whittaker $\cD$-module perspective). 
The work of Lonergan and Ginzburg~\cite{lonergan,ginzburg whittaker} identifies the Kostant category $\cK$ with the full subcategory of $W_{\aff}$-equivariant quasicoherent sheaves on $\fhx$, on which the derived inertia action is trivial, or equivalently descend to the categorical quotient $\fhx\GIT W$ by the finite Weyl group (or by every finite parabolic subgroup of $W_{\aff}$).
\end{example}

\begin{remark}
The Morita equivalence between the spherical and full nil-DAHA can be understood from the Whittaker perspective as follows. Recall that there is a categorical Morita equivalence between the monoidal category $\HC$ and the (universal, monodromic) Hecke category 
\[
\widehat{\cH}_G = \cD(\quot{\oN}{G}{\oN})^{H\times H, wk}
\]
consisting of weakly $\overline{B}$, strongly $\overline{N}$ bi-equivariant $\cD$-modules (where $\overline{B}$ is the opposite Borel to $B$).
Under this Morita equivalence, the action of $\HC$ on the category $\cZ$ via Whittaker modules corresponds to the action of $\widehat{\cH}_G$ on 
\[
\cD( \quot{\oN}{G}{_\psi N})^{H,wk} \simeq \Sym(\fh)\module = \QC(\fh^\ast)
\]
In particular, there is a monoidal, monadic forgetful functor
\[
\cWh \simeq \End_{\HC}(\cZ) \simeq \End_{\widehat{\cH}_G}(\Sym(\fh)\module) \to \Sym(\fh)\module
\]
The $\Sym(\fh)$-ring corresponding to the monad is precisely the nil-DAHA (after unwinding the definitions, this is computed in \cite{ginzburg whittaker}). 
\end{remark}

\subsection{Further directions}\label{further section}
In this section we briefly mention some applications of the quantum Ng\^o action that we intend to pursue in future work.

\subsubsection{Langlands parameters}
The action of $\cWh\simeq \QC(\fhx\GIT W_{\aff})$ provides a notion of Langlands parameters for categorical representations of $G$. Indeed, we may identify the quotient complex analytically $$\fhx\GIT W_{\aff} \sim H^\vee\GIT W$$ with the affinization of the (Betti) stack $$\Gv/\Gv=\Loc_\Gv(D^\times)$$ of $\Gv$-local systems on the punctured disc. This identification should be closely related to the (de Rham) local geometric Langlands program~\cite{frenkel dennis, quantum langlands}.
Indeed it is expected (see~\cite[Example 1.23.1]{raskin W}) that 
$$\cWh(LG)\simeq \QC(Conn_\Gv(D^\times)):$$ the Hecke category of bi-Whittaker $\cD$-modules on the loop group $LG$ (i.e. the ``affine $\cW$-category") is symmetric monoidal, and equivalent to quasicoherent sheaves on the stack of $\Gv$-flat connections on the punctured disc. Thus
passing to ``Whittaker vectors" on categorical representations of $LG$ produces quasi-coherent sheaves of categories on the stack of $\Gv$-connections on $D^\times$ -- the geometric version of local Langlands parameters~\cite{quantum langlands}. This conjecture is a categorical analog of the Feigin-Frenkel description~\cite{FF} of the affine $\cW$-algebra, as our result is a categorical analog of the Kostant description of the finite $\cW$-algebra. Our proof of commutativity of $\cWh$ does not readily generalize to the affine setting, but we hope a deeper and cleaner understanding of $\cWh$ will prove useful in this regard. 

\subsubsection{Eigencategories for $\cWh$ and refined central character}
The Harish-Chandra system on $G/G$ (as in~\cite{HK}) is a reductive group ancestor of Beilinson-Drinfeld's quantized Hitchin system on the stack $Bun_G$ of $G$-bundles on an algebraic curve, and Lusztig's character sheaves~\cite{lusztig character, laumon character} are likewise the ancestors of automorphic sheaves in the geometric Langlands correspondence. Arinkin~\cite{arinkin thesis, arinkin paper} explained that the Hecke functors on $\cD(Bun_G)$ in the geometric Langlands correspondence appear naturally as an aspect of the quantized Hitchin system -- in Arinkin's paradigm, a quantization of completely integrable systems entails a deformation of symmetric monoidal categories, in this case deforming the translation symmetries of the classical system to the action of Hecke functors. We expect the action of $\cWh$ on $\cD(G/G)$, deforming the Ng\^o integration of the Hamiltonian flows, plays an analogous role for the Harish-Chandra system as the Hecke functors for the quantized Hitchin system. In particular character sheaves appear as $\cWh$-eigensheaves just as automorphic sheaves appear as Hecke eigensheaves. In particular, the action of $\cWh$ on $\cD(G/G)$ provides a refinement of the theory of central characters of character sheaves, as explained below.

Recall that the symmetric monoidal category $\cWh = \QC(\fh^\ast \GIT W^\aff)$ acts centrally on any $G$-category. Given a point $[\lambda] \in \fh^\ast/W^\aff$, we have an corresponding symmetric monoidal functor $\cWh \to \Vect$, i.e. a $\cWh$-module category $\Vect_{[\lambda]}$. For any $\cD(G)$-module category or $\HC$-module category $\cC$, we regard $\cC$ as a $\cWh$-module category via the Ng\^o functor and consider the categorical (co)invariants
\[
\cC^{\cWh,[\lambda]} = \Hom_{\cWh}(\Vect_{[\lambda]}, \cC) \simeq \cC \otimes_{\cWh} \Vect_{[\lambda]}
\]
For example, we have a braided monoidal category $\cD(G\adjquot G)^{\cWh, [\lambda]}$, a refined (or strict) version of the category of character sheaves with central character $[\lambda]$ (the usual category $\cD(G\adjquot G)^{\widehat{[\lambda]}}$ of character sheaves with a fixed central character corresponds to taking the completion at $[\lambda] \in \fh^\ast\GIT W^\aff$ rather than the fiber). One expects that the category $\cD(G\adjquot G)^{\cWh,[\lambda]}$ is ``more semisimple'' than $\cD(G\adjquot G)^{\widehat{[\lambda]}}$. We hope that these constructions will shed some light on the truncated Hecke and character sheaf categories defined by Lusztig \cite{lusztig cells, lusztig convolution} (see also \cite{BFO}).

\subsubsection{Character field theory and cohomology of character varieties}\label{TFT section}
We showed in~\cite{character2} (extending~\cite{character}) that the monoidal category $\HC$ controls the Borel-Moore homology of character varieties of surfaces, via the mechanism of a 3d topological field theory $\cX_G$, the {\em character field theory}.
Recall that given a topological surface $S$ the
character variety (or Betti space) $\Loc_G(S)$ is the derived stack of $G$-local systems on
$S$,  $$\Loc_G(S)\sim \{\pi_1(S)\to G\}/G.$$ 
The character theory is defined by prescribing that quantum Hamiltonian $G$-spaces ($\HC$-modules) define boundary conditions for $\cX_G$, and ``integrates" them on surfaces to obtain the homology of character varieties:

\begin{theorem}\cite{character2} The assignment $$\cX_G(pt)=\HC\modcat\simeq \cD(G)\modcat$$ satisfies the conditions of the Cobordism Hypothesis~\cite{jacob TFT} to define an oriented topological field theory, attaching a dg vector space to a closed surface. Moreover we have a canonical equivalence $$\cX_G(S)\simeq H_*^{BM}(\Loc_G(S))$$ with the Borel-Moore homology on the character variety, for $S$ an oriented closed surface.
\end{theorem}

We also prove a ``Hodge filtered" version of the theorem, which in particular defines a family $\cX_{\hbar,G}$ of topological field theories out of the $\hbar$-family of monoidal categories $\HC_\hbar$.

The quantum Ng\^o action of $\cWh_\hbar$ on $\HC_\hbar$ makes the entire character theory $\cX_{\hbar,G}$ linear over $\cWh_\hbar$, i.e. (for $\hbar\neq 0$) a family of topological field theories over $\fhx\GIT W_{\aff}$. In particular the Borel-Moore homology of $\Loc_G(S)$ sheafifies over $\fhx\GIT W_{\aff}$, with fibers defining new invariants, the {\em eigenhomology} of the character variety. We expect eigenhomologies of character varieties to be more accessible to combinatorial description.

The work of Hausel, Rodriguez-Villegas and Letellier \cite{H,HRV,HLRV}
has uncovered remarkable combinatorial patterns in the
cohomology of the character varieties, leading to a series of striking
conjectures. A central technique is counting points over finite fields, i.e., points of character varieties $\Loc_{G_q}(S)$ of the finite Lie groups
$G_q=G(\F_q)$. These counts are captured by the values of a 2d TFT ($G_q$ Yang-Mills theory), which assigns to a point the category $Rep(G_q)$ of representations of the finite group. Lusztig's Jordan decomposition of characters breaks up this category, and hence the counts on any surface, into blocks labeled by semisimple conjugacy classes in the dual group (i.e., informally speaking, over $H^\vee\GIT W$). The 3d character theory accesses the homology of character varieties (as opposed to the point count or $E$-polynomial) directly, and the decomposition over $\cWh$ (i.e. over $\fhx\GIT W^{\aff}\sim H^\vee\GIT W$) plays the role of the Jordan decomposition. This decomposition, which we will explore in a future paper, provided the original motivation for this work.

\subsection{Supersymmetric gauge theory}\label{SUSY}
We briefly indicate the interpretation of our constructions in the context of supersymmetric gauge theory, following discussions with Andy Neitzke, Tudor Dimofte and Justin Hilburn. See~\cite{BZ talk} for a slightly more leisurely discussion. Details will appear elsewhere.

To any 3d $\cN=4$ theory $\cZ$ is associated a holomorphic symplectic variety $\fM_\cZ$, its {\em Coulomb branch}, together with a deformation quantization $\CC_\hbar[\fM_\cZ]$ of its ring of functions, obtained as the algebra of supersymmetric local operators in the theory in $\Omega$-background~\cite{NW} (with quantization parameter $\hbar=\epsilon\in H^\ast(BS^1)$). See e.g.~\cite{BDG} and references therein. If $\cZ$ is a 3d {\em gauge} theory, one can define (using a Lagrangian description of $\cZ$) an integrable system $$\fM_\cZ\to \fc$$ with base the adjoint quotient of the Langlands dual of the gauge group.
The identification by~\cite{BFM} of the groupscheme $J$ of regular centralizers in terms of the equivariant homology of the Langlands dual Grassmannian is now understood (thanks to~\cite{teleman, BraFinkNa}) as describing the Coulomb branch of pure 3d $\cN=4$ gauge theory, while its quantization using $\Cx$-equivariant homology is an instance of the quantization in $\Omega$-background. 

However the {\em abelian group} structure of $J$ (and symmetric monoidal structure of its quantization), as well as the classical and quantum Ng\^o actions, are
best understood using 4d gauge theory -- specifically, in the spirit of Kapustin-Witten~\cite{KW}, as aspects of 
4d $\cN=4$ super-Yang-Mills (in the GL twist at $\Psi=\infty$). Indeed the base $\fc$ arises as the Coulomb branch of 4d SYM, while the characteristic polynomial map (in fact, a shifted integrable system)
$$\fgx/G\to\fc$$ arises from identifying the residual gauge symmetry of the theory on its Coulomb branch. 
The category $\QC(\fgx/G)$ is the monoidal (naturally $E_3$) category of (Wilson) line operators in the theory, and its deformation $\HC_\hbar$ is the monoidal category of line operators in the 4d $\Omega$-background (with $\epsilon_1=\hbar,\epsilon_2=0$).
The derived geometric Satake theorem of~\cite{BezFink} is thus interpreted as implementing S-duality for line operators, identifying the Wilson lines with 't Hooft lines (Hecke modifications). 
  
The Ng\^o map and its quantization are most naturally interpreted as providing an integration of the shifted integrable system $\fgx/G\to \fc$ (Ng\^o's ``mold") and its quantization. Rather than spell this structure out, we mention one of its consequences in terms of the familiar geometry of 3d Coulomb branches. The Ng\^o action provides symmetries of arbitrary BPS  boundary conditions for the 4d $\mc N=4$ theory and of their Coulomb branches (which produce holomorphic hamiltonian $G$-spaces). In particular one can pair two such boundary conditions, reducing the 4d theory on an interval to produce a 3d $\cN=4$ theory:

\begin{claim}
Let $\cZ$ denote any 3d $\cN=4$ theory obtained by reduction of 4d $\cN=4$ on an interval. Then the Coulomb branch $\fM_\cZ$ carries an integrable system $$\fM_\cZ\to\fc$$ which integrates to an action of the symplectic groupoid $J\to \fc$. Likewise the $\Omega$-deformed algebra $\CC_\hbar[\fM_\cZ]$ carries a quantum integrable system $\fZ_\hbar\fg\to \CC_\hbar[\fM_\cZ]$ which integrates to an action of $\cWh_\hbar$. In particular the category of modules for the quantized Coulomb branch sheafifies over $\fhx\GIT W^{\aff}$.
\end{claim}

The claim is the physical counterpart to Proposition~\ref{classical integration} and Corollary~\ref{quantum ham corollary} --- (classical and quantum) 
Hamiltonian reductions of $G$-spaces correspond to reductions of 4d $\cN=4$ along different pairs of boundary conditions. In particular {\em Whittaker} reduction, or restriction to the Kostant slice, corresponds to pairing with a Neumann boundary condition, i.e., {\em gauging} 3d $\cN=4$ theories with global symmetry (with gauge group the compact form of 
the dual group $\Gv$) - see~\cite{BDGH} for a closely related discussion. Thus the class of 3d $\cN=4$ theories obtained this way includes in particular all 3d $\cN=4$ gauge theories. However since the Kostant slice is contained in the regular locus, such theories don't probe the irregular locus of $\fgx$, and one doesn't need the Ng\^o construction to see the action of $J$, which follows immediately from the structure of hamiltonian $G$-space (or Langlands dually from the Braverman-Finkelberg-Nakajima construction of the Coulomb branch~\cite{BraFinkNa}).

\subsection{Acknowledgments} 
This project grew out of a joint project with David Nadler (parts of which appeared as~\cite{character2} and~\cite{hendrik}), and we would like to express our deep gratitude for his essential contributions. In particular the idea to quantize the commutative group-scheme $J$ of regular centralizers and the Ng\^o correspondence to a central action of a symmetric monoidal category is due to him. 

We are greatly indebted to Dario Beraldo and Sam Raskin for their help with the formalism of renormalized D-modules.
We would also like to thank Constantin Teleman for generously sharing his understanding of the relations between categorical representation theory, gauge theory and $J$, Simon Riche for discussion of mixed geometric Satake, Geoffroy Horel for his assistance with formality of Hopf algebras, Marco Gualtieri and James Pascaleff for sharing their ideas on Fukaya categories on symplectic groupoids, and Dima Arinkin, Dennis Gaitsgory, Victor Ginzburg, Gus Lonergan, and Ben Webster for their interest and useful discussions. 
DBZ would like to acknowledge the National Science Foundation for its support through individual grants DMS-1103525 and DMS-1705110. We would also like to acknowledge that part of the work was carried out at MSRI as part of the program on Geometric Representation Theory.


\section{Sheaf Theory: Ind-holonomic $D$-modules}\label{sheaf theory}

\subsection{DG categories}\label{dgcat}
We refer the reader to~\cite[I.1.5-8]{GR} as well as~\cite{BFN,BGT} for summaries of the basic properties of stable $\oo$-categories following~\cite{HA}. 
We now summarize the main points we will need.

Recall~\cite{HTT,HA} that $\Pr^L$ denotes the symmetric monoidal $\oo$-category of presentable $\oo$-categories with continuous (colimit preserving) functors, i.e., (by the adjoint functor theorem) functors which are left adjoints. Further $St\subset \Pr^L$ denotes the symmetric monoidal $\oo$-category of stable presentable $\oo$-categories. 

We will denote by $\DGCat_k$ the symmetric monoidal $\oo$-category of {\em cocomplete dg categories over $k$}, i.e., stable presentable $k$-linear $\oo$-categories. In other words $\DGCat_k$ consists of module categories for $k\module$ in $St$. 
We are mostly interested in the subcategory $\DGCat_k^c$ of compactly-generated dg categories with proper functors, i.e., continuous functors preserving compact objects, or equivalently functors that admit continuous right adjoints. The functors of taking compact objects and passing to Ind-categories define inverse symmetric monoidal equivalences of $\DGCat_k$
with the symmetric monoidal $\oo$-category $\DGCat_k^{sm}$ of small, idempotent-complete dg categories with exact functors.  By~\cite[Corollary 4.25]{BGT} $\DGCat_k^c$ (or equivalently $\DGCat^{sm}$ is presentable.

\subsection{Sheaf Theory Formalism.}
We will study monoidal properties of categories of sheaves on stacks. The geometric spaces that appear are ind-algebraic stacks and groupoids (Section~\ref{context}). We require a theory of sheaves that attaches to a stack $X$ a presentable DG category $Shv(X)$ with continuous pull-back and pushforward functors $p_*,p^!$ for maps $p:X\to Y$ of ind-finite type, satisfying base change and an adjunction $(p_*,p^!)$ in the case that $p$ is ind-proper.

Two important examples of such a theory of sheaves, developed in~\cite{GR}, are the theory of ind-coherent sheaves $IndCoh(X)$ and the theory of $\cD$-modules $\cD(X)$. Their properties are summarized in the following:

\medskip

\begin{theorem}~\cite[Theorem III.3.5.4.3, III.3.6.3]{GR} \label{GR sheaf theory}
There is a uniquely defined right-lax symmetric monoidal functor $IndCoh$ from the $(\infty,2)$-category whose objects are {\em laft} prestacks, morphisms are correspondences with vertical arrow ind-inf-schematic, and 2-morphisms are ind-proper and ind-inf-schematic, to the $(\infty,2)$ category of DG categories with continuous morphisms.
\end{theorem}

\medskip

The theorem encodes a tremendous amount of structure. Let us highlight some salient features useful in practice.
The theorem assigns a symmetric monoidal dg category $IndCoh(X)$ to any reasonable (locally almost of finite type) stack. The symmetric monoidal structure, the $!$-tensor product, is induced by $!$-pullback along diagonal maps. For an arbitrary morphism $p:X\to Y$ there is a continuous symmetric monoidal pullback functor $p^!:IndCoh(Y)\to IndCoh(X)$, while for $p$ schematic or ind-schematic there is a continuous pushforward $p_*:IndCoh(X)\to IndCoh(Y)$, which satisfies base change with respect to $!$-pullbacks. Moreover for $p$ ind-proper, $(p_*,p^!)$ form an adjoint pair. Furthermore, the formalism of {\em inf-schemes} greatly extends the validity of the construction. In particular the same formal properties holds for the theory of $\cD$-modules, defined by the assignment $X\mapsto \cD(X)=IndCoh(X_{dR})$, ind-coherent sheaves on the de Rham space of X.


For our applications we require a minor variation, the theory of ind-holonomic $\cD$-modules $\Dhol(X)$, the main instance of which is the renormalized Satake category $\Dhol(\uGrv)$ studied in~\cite{AG} (and, implicitly,~\cite{BezFink}). We will explain the appropriate modifications of the formalism of~\cite{GR} needed to establish the minimal functoriality of ind-holonomic $\cD$-modules we will require.

\subsubsection{Geometric context}\label{context}
We adopt the following geometric conventions: all schemes will be of almost finite type, and all algebraic stacks will be {\em laft} QCA stacks, as studied in particular in~\cite{finiteness}. In other words, an algebraic stack $X$ is a prestack whose diagonal is affine and which admits a smooth and surjective map from an affine scheme of almost finite type. 

By an {\em ind-algebraic stack} we refer to a prestack $X$ which is equivalent to a filtered colimit $X=\lim_{\rightarrow} X_i$ of algebraic stacks under closed embeddings.

In our applications $X$ will be realized as the quotient of an ind-scheme of ind-finite type by an affine algebraic group. The main example of interest is the equivariant affine Grassmannian $$X=\uGrv=G(\cO)\backslash G(\cK)/G(\cO)$$ of a reductive group $G$.

\subsection{Motivating Ind-Holonomic $\cD$-modules}\label{d-modules}

First recall (see e.g.~\cite{finiteness}) that for a scheme of finite type we have an equivalence $\cD(X)\simeq \Ind \cD_{coh}(X)$, and that we have a full stable subcategory $\cD_{coh,hol}\subset \cD_{coh}(X)$. Thus we have a fully faithful embedding $$\Dhol(X):=\Ind \cD_{coh,hol}(X)\subset \cD(X)$$ of ind-holonomic $\cD$-modules into all $\cD$-modules. Holonomic $\cD$-modules are preserved by $!$-pullback and $*$-pushforward for finite type morphisms, and carry a symmetric monoidal structure through $!$-tensor product for which $!$-pullback is naturally symmetric monoidal.

This picture persists for $X$ an ind-scheme of ind-finite type $X=\lim_{\rightarrow} X_i$, for example the affine Grassmannian $Gr=G(\cK)/G(\cO)$. The $(i_*,i^!)$ adjunction for a closed embeddings provides the alternative descriptions $$\cD(X)\simeq \lim_{\leftarrow,(-)^!} \cD(X_i)\simeq \lim_{\rightarrow,(-)_*} \cD(X_i).$$ As a result (by a general lemma of~\cite{DrG2}) $\cD(X)$ is compactly generated by coherent $\cD$-modules, which by definition are the pushforwards of coherent $\cD$-modules on the finite type closed subschemes $X_i$, and include the similarly defined holonomic $\cD$-modules. Note that with this definition the pullback of a holonomic $\cD$-module by an ind-finite type morphism (for example, the dualizing complex of an ind-scheme) is ind-holonomic but not necessarily holonomic (i.e. compact). 

For $X$ an algebraic stack, the situation (as studied in detail in~\cite{finiteness}) changes: coherent (and in particular holonomic) $\cD$-modules, defined by descent using a smooth atlas, are no longer compact in general. The category $\cD(X)$ is compactly generated by {\em safe} objects, which are coherent objects satisfying a restriction on the action of stabilizers (in the case of quotient stacks). One can thus measure the lack of safety of $X$ by the difference between $\cD(X)$ and the category $\cDv(X):=\Ind \cD_{coh}(X)$ of {\em ind-coherent} or {\em renormalized} $\cD$-modules. This is analogous to the difference between quasicoherent and ind-coherent sheaves on a derived stack measuring its singularities, with safe (respectively, coherent) $\cD$-modules taking on the role of perfect (respectively, coherent) complexes of $\cO$-modules.

\begin{example}\label{example-Dhol on classifying} 
Suppose $X=pt/G$ is the classifying stack of a reductive group. Let $\Lambda=C_*G\simeq \CC[\fg^*[-1]]^G$ and $S=C^*X\simeq \CC[\fg[2]]^G$ be the corresponding Koszul dual exterior and symmetric algebras. Then $$\cD(X)\simeq \Lambda\module\simeq QC(\fg[2]\GIT G)_0$$ is the completion of sheaves on the graded version of the adjoint quotient $\fg\GIT G\simeq \fh\GIT W$ at the origin. On the other hand we have $$\Dhol(X)=\Dhol(X)\simeq \Ind(Coh \Lambda)\simeq S\module\simeq QC(\fg[2]\GIT G)$$ is the ``anticompleted" version of the same category. 

This can also be described in terms of the corresponding homotopy type $X_{top}$ (as a constant prestack) and $\fX=\Spec C^*(X)$ the corresponding coaffine stack. We then have equivalences $$\cD(X)\simeq QC(X_{top})\simeq QC(\fX).$$ 
On the other hand we have the following description of renormalized sheaves:
$$\Dhol(X) \simeq C^*(X)\module.$$
In particular $\cD(X)$ is the completion of $\Dhol(X)$. 
\end{example}

We will be interested in a combined setting of ind-algebraic stacks. In this setting the category $\Dhol(X)$ (defined formally in the next section) is identified with the Ind-category of (coherent) holonomic $\cD$-modules, which are pushforwards of holonomic $\cD$-modules on algebraic substacks. Thus ind-holonomic $\cD$-modules form a full subcategory of {\em ind-coherent} (or renormalized) $\cD$-modules 
$\Dhol(X) = \Ind \cD_{coh}(X).$

\begin{example}
Our main motivating example is the equivariant affine Grassmannian $X=\uGrv$. The {\em renormalized Satake category} $\Dhol(\uGrv)$ of~\cite{AG} is a variant of the usual Satake category $\cD(\uGrv)$
which appears (implicitly) in the derived Satake correspondence of~\cite{BezFink}. It can be defined as the 
ind-category $\Ind(Shv_{lc}(\uGrv))$ of the category of {\em locally compact} sheaves on $\uGrv$, i.e., equivariant sheaves on the affine Grassmannian for which the underlying sheaves are constructible (hence compact).
In the language of $\cD$-modules, it is the Ind-category of the category of holonomic $\cD$-modules on $\uGrv$ - note that (as in the previous example) all coherent $\cD$-modules on $\uGrv$ are holonomic, in fact regular holonomic, hence identified with constructible sheaves. 
The renormalized Satake theorem~\cite{BezFink,AG} is an equivalence of monoidal categories
$$\Dhol(\uGrv)=\Dhol(\uGrv)\simeq IndCoh(\fg^{\vee}[2]/\Gv).$$ Dropping the renormalization of $\cD$-modules corresponds to imposing finiteness conditions on the right hand side.
\end{example}


\begin{remark}[Ind-constructible sheaves]
The notion of ind-holonomic $\cD$-modules has a natural analog in the setting of l-adic sheaves or constructible sheaves in the analytic topology. Namely on a scheme $X$ the compact objects in $Shv(X)$ are the constructible sheaves, but this is no longer the case on a stack. A {\em  locally compact} sheaf on a stack $X$ is a sheaf whose stalks are perfect complexes -- i.e., whose pullback under any map $pt\to X$ is compact. We denote $Shv(X)_{lc}\subset Shv(X)$ the full subcategory of locally compact sheaves, and define the category $\breve{Shv}(X)$ of renormalized, or ind-constructible, sheaves as $\Ind Shv(X)_{lc}$. It has $Shv(X)$ as a colocalization:
$$\xymatrix{\Xi: Shv(X)\ar[r]<.5ex> &\ar[l]<.5ex> \breve{Shv}(X):\Psi}$$ 
For example for $X=Y/G$ a quotient stack, $\breve{Shv}(X)$ can be identified with the Ind category of $G$-equivariant constructible complexes on $Y$ in the sense of Bernstein--Lunts \cite{bernsteinlunts}. 

The $!$-tensor structure on $Shv(X)$ respects locally compact objects, hence extends by continuity to define a symmetric monoidal structure on $\breve{Shv}(X)$, for which the functors $\Xi,\Psi$ upgrade to symmetric monoidal functors.

When $X$ is a finite orbit stack (for example, a quotient stack $Y/G$ where $G$ acts on $Y$ with finitely many orbits) or an ind-finite orbit stack such as $\uGrv$, every coherent complex on $X$ is regular holonomic. Thus, via the Riemann-Hilbert correspondence, $\Dhol(X)=\Dhol(X)\simeq \breve{Shv}(X)$. 
\end{remark}

\subsection{Formalism of ind-holonomic $\cD$-modules}

Recall~\cite{GRcrystals,GR,dario,raskin} the construction of the contravariant functor of $\cD$-modules $\cD^!$ on ind-schemes. Namely we start with the functor $$\cD^!:AffSch^{f.t.,op}\to \DGCat$$ of $\cD$-modules with $!$-pullback 
on schemes of finite type as constructed e.g. in~\cite{GRcrystals,GR}. We then right Kan extend to ind-schemes of ind-finite type (or more generally to {\em laft} prestacks). 

\begin{defn} The right-lax symmetric monoidal functor $\Dhol^!:QCA^{op}\to \DGCat$ is defined as the (symmetric monoidal) 
ind-construction 
$$\xymatrix{QCA^{op} \ar[rr]^-{\cD_{coh,hol}^!} && \DGCat^{sm} \ar[rr]^-{Ind} && \DGCat}$$ applied to the subfunctor of $\cD^!$ defined by coherent holonomic $\cD$-modules.
\end{defn}

Note that by construction $\Dhol(X)$ for a QCA stack is compactly generated by $\cD_{coh,hol}(X)$.

\begin{lemma}\label{QCA functoriality} For $p:X\to Y$ a finite type morphism of QCA stacks, we have continuous pullback and pushforward functors 
$$\xymatrix{p_*:\Dhol(X)\ar[r]<.5ex>& \ar[l]<.5ex> \Dhol(Y): p^!}$$ satisfying base change. Moreover for $p:X\to Y$ a proper morphism, $(p_*,p^!)$ form an adjoint pair. 
\end{lemma} 

\begin{proof}
Pullback and pushforward of holonomic $\cD$-modules on stacks under finite type morphisms remain holonomic. Hence the functors
$$\xymatrix{p_*:\cD_{coh,hol}(X)\ar[r]<.5ex>& \ar[l]<.5ex> \cD_{coh,hol}(Y): p^!}$$
extend by continuity to the ind-categories. The property of base-change can likewise be checked on the compact objects. 
\end{proof}

If we need to consider schemes beyond finite type, we first perform a left Kan extension to extend from affine schemes to all affines and then right Kan extend extend $\cD^!$ to all ind-schemes~\cite{raskin}. Another formulation~\cite{dario} is to consider schemes of pro-finite type or simply {\em pro-schemes}, schemes that can be written as filtered limits of schemes of finite type along affine smooth surjective maps. Again $\cD^!$ is extended from finite type schemes to pro-schemes as a left Kan extension, and then to ind-pro-schemes by a right Kan extension.

We are interested in objects such as the equivariant affine Grassmannian $\uGrv$, which is nearly but not quite an ind-finite type algebraic stack. Namely $\uGrv$ is the inductive limit (under closed embeddings) of stacks of the form $X/K$ where $X$ is a scheme of finite type and $K$ ($\LGpv$ in our setting) is an algebraic group acting on $X$ through a finite type quotient $K_f=K/K^u$ with pro-unipotent kernel $K^u$. Thus $X/K$ is a projective limit of finite type algebraic stacks under morphisms which are gerbes for unipotent group schemes. However the category of $\cD$-modules is insensitive to unipotent gerbes, so in particular the category of $\cD$-modules on $X/K$ is equivalent to that of the finite type quotient
$X/K_f$.

 Thus we make the following more modest variant of the constructions in~\cite{dario, raskin}:

\begin{defn} \begin{enumerate}
\item By a stack nearly of finite type we refer to an algebraic stack expressible as a projective limit of QCA stacks under morphisms which are gerbes for unipotent group schemes.
\item By an ind-nearly finite type stack, or simply {\em ind-stack}, we denote a prestack equivalent to an inductive limit of stacks nearly of finite type under closed embeddings. The symmetric monoidal category of ind-stacks is denoted $IndSt$.
\end{enumerate}
\end{defn}

\begin{defn}
The functor $\Dhol^!:IndSt^{op}\to \DGCat$ on ind-stacks is defined by first left Kan extending $\Dhol^!$ from QCA stacks to stacks nearly of finite type, and then right Kan extending to ind-nearly finite type stacks.
\end{defn}

\begin{proposition} The functor $\Dhol^!$ admits a right-lax symmetric monoidal structure extending that previously defined on QCA stacks.
\end{proposition}

\begin{lemma} \begin{enumerate}
\item For $\cX=\lim_\leftarrow X_n$ an inverse limit of stacks of finite type under unipotent gerbes, the functor $$\lim_{\leftarrow} \Dhol(X_n)\to \Dhol(X_i)$$ is an equivalence for any $i$. 
\item The assertions of Lemma~\ref{QCA functoriality} extend to morphisms of nearly finite type stacks.
\end{enumerate}
\end{lemma}

To calculate the abstractly defined functor $\Dhol$ on ind-stacks, we follow the strategy of~\cite{GR} (see also~\cite[Section 2]{GRindschemes}): 

\begin{proposition}\label{inductive ind-holonomic}
For $X$ an ind-stack, expressed as a filtered colimit of closed embeddings $i_n:X_n\hookrightarrow X$ with $X_n$ nearly of finite type, we have identifications $$\Dhol(X)\simeq \lim_{\leftarrow, i_{n}^!} \Dhol(X_n)\simeq \lim_{\rightarrow, i_{n,*}} \Dhol(X_n).$$
In particular $\Dhol(X)$ is compactly generated by pushforwards of coherent holonomic $\cD$-modules on the $X_n$.
\end{proposition}

\begin{proof}
The functor $\Dhol$ takes colimits in $IndSt$ to limits in $\DGCat$. Hence for an ind-stack $X=\lim_{\leftarrow, i_n} X_n$ written as a colimit of nearly finite type stacks under closed embeddings, 
we have an identification $\Dhol(X)\simeq \lim_{\rightarrow,i_n^!} \Dhol(X_n)$. Since the $X_n$ are nearly finite type stacks and $i_n$ are proper morphisms, we may apply proper adjunction to further identify the limit over the pullbacks with the colimit over their left adjoints,
$\Dhol(X)\simeq \lim_{\leftarrow, i_{n*}} \Dhol(X_n)$ as desired. 
\end{proof}

\begin{proposition} \label{IndHol adjunction}
 For $p:X\to Y$ an ind-finite type morphism in $IndSch$, we have continuous pushforward and pullback functors 
$$\xymatrix{p_*:\Dhol(X)\ar[r]<.5ex>& \ar[l]<.5ex> \Dhol(Y): p^!}$$ satisfying base change. 
For $p:X\to Y$ ind-proper, $(p_*,p^!)$ form an adjoint pair. 
\end{proposition}

\begin{proof}
Let us write $Y$ as the filtered colimit of closed embeddings of nearly finite type substacks $t_n:Y_n\hookrightarrow Y$, and $s_n:X_n=X\times_Y Y_n \hookrightarrow X$. Then by hypothesis we can further decompose $X_n$ as the colimit of substacks $i_{m,n}:X_{m,n}\hookrightarrow X_n$ with $p_{m,n}:X_{m,n}\to Y_n$ finite type.

A holonomic $\cD$-modules $\cF$ on $X$ can be represented as the pushforward of a holonomic $\cD$-module $\cF_{m,n}$ on some $X_{m,n}$. Hence $p_*\cF=p_{m,n*}\cF_{m,n}$ is holonomic. Thus pushforward on all $\cD$-modules restricts to a functor 
$$p_*:\cD_{coh,hol}(X)\to \cD_{coh,hol}(Y)$$ which thus extends by continuity to the ind-categories $\Dhol$. 

Pullback defines a functor $$p_{m,n}^!:\cD_{coh,hol}(Y_n)\to \cD_{coh,hol}(X_{m,n}),$$ and thus passing to ind-categories by continuity
$$p_{m,n}^!:\Dhol(Y_n)\to \Dhol(X_{m,n}).$$ By Proposition~\ref{inductive ind-holonomic}, these functors assemble to a continuous functor to the inverse limit category and on to the target, 
$$\xymatrix{\Dhol(Y_n)\ar[r]^-{p_n^!}& \Dhol(X_n)\ar[r]^-{s_{n,*}}& \Dhol(X)}$$ Finally by (finite type) base change the functors $s_{n,*}p_n^!\simeq p^!t_{n,*}$ assemble to a functor from the direct limit category $$\lim_{\rightarrow}\Dhol(Y_n)=\Dhol(Y)$$ to $\Dhol(X)$. The resulting functors inherit the base change property from their finite type constituents. 

\end{proof}

\begin{remark}[Bivariant functoriality]\label{remark-bivariant} The key 2-categorical extension theorem of Gaitsgory-Rozenblyum, ~\cite[Theorem V.1.3.2.2]{GR}, allows one to define functors out of correspondence 2-categories given 1-categorical data, namely a functor (in our case $\Dhol^!$) satisfying an adjunction and base change property for a particular class of morphisms (in our case ind-proper morphisms). Thus we find that the functor $\Dhol^!:IndSt^{op}\to \DGCat$ extends to a functor of $(\infty,2)$-categories
$$\Dhol:Corr_{ind-f.t,ind-prop}^{ind-prop}(IndSt)\to \DGCat^{(\infty,2)}.$$ 
\end{remark}

\section{Hecke algebras and Hecke categories} \label{Hecke section}

In this section we describe a general formalism for constructing symmetric monoidal categories acting centrally on convolution categories. We work in the setting of ind-holonomic $\cD$-modules on ind-stacks described above, since our main example is the renormalized Satake category $\Dhol(\uGrv)$ and more generally Hecke categories for Kac-Moody groups $\Dhol(\uP\backslash \uG/\uP)$. However the discussion of this section works identically when applied to the sheaf theories of ind-coherent sheaves $\QC^!$ or $\cD$-modules $\cD$ when restricted to {\em laft} prestacks, as in~\cite{GR}.

\subsection{Looping and delooping monoidal categories}

We recall the following fundamental feature of algebras and their module categories, due to Lurie (combining aspects of Theorems 6.3.5.5, 6.3.5.10 and 6.3.5.14 in~\cite{HA} -- Lurie also proves functoriality in $\cP$ which we omit). See~\cite[Section E.2]{AG} for a related discussion in the stable setting of dg categories.

Let us fix a presentable symmetric monoidal category $\cP\in \Pr^L$, and let $Cat_\cP:=\cP\modcat$ denote the symmetric monoidal category of $\cP$-module categories in $\Pr^L$. Thus for $\cP=k\module$ ($k$ a ring of characteristic zero) we have $Cat_\cP=\DGCat_k$, the symmetric monoidal category of presentable $k$-linear dg categories. We will be interested in applying the result for $\cP=\DGCat_k^c$ (see Section~\ref{dgcat}), so that an algebra $A$ in $\cP$ is a small monoidal dg category, or equivalently a compactly generated presentable dg category with proper monoidal structure, and $A\module$ is the $\oo$-category of $A$-module categories. 

\medskip

\begin{theorem}~\cite[Section 6.3.5]{HA}
There is a symmetric monoidal functor $$\Mod:Alg_{E_1}(\cP)\longrightarrow (Cat_\cP)_{\cP/}$$
from $E_1$-algebras in $\cP$ to $\cP$-categories under $\cP$ (i.e., $E_0$-$\cP$-categories), sending $A$ to the $\cP$-category $\rmodule{\cA}$ of right $A$-modules in $\cP$, pointed by $A$ itself. 
This functor admits a right adjoint $\Omega$, sending a pointed category $p:\cP\to\cM$ to $$\Omega_p(\cM,p)=End_{\cM}(p(1_\cP))\in Alg_{E_1}(\cP).$$ 
\end{theorem}

By iteratively applying Lurie's Dunn additivity theorem~\cite[Theorem 5.1.2.2]{HA} we may likewise loop and deloop between $E_n$-algebras in $\cP$ and $E_{n-1}$-monoidal $\cP$-categories. We spell out the case we will use:

\begin{corollary} \label{nonsense} 
\begin{enumerate}
\item Taking endomorphisms of unit objects defines a functor 
$$\Omega: Alg_{E_1}(Cat_\cP)\longrightarrow Alg_{E_2}(\cP)$$
from monoidal $\cP$-categories to $E_2$-algebras in $\cP$.
\item For $\cA\in Alg_{E_1}(Cat_\cP)$ a monoidal $\cP$-category, the $E_2$-morphism
$$1_\cA\ot -: End(1_\cA)\to End(Id_\cA)$$ given by applying $\Omega$ to the action of $\cA$ on itself admits a left inverse as an $E_1$-morphism 
$$act_{1_\cA}:End(Id_\cA)\to End(1_\cA), $$ given by the action of $End(Id_\cA)$ on the object $1_\cA$.
\end{enumerate}
\end{corollary}

\begin{proof}
The functor $\Omega$, by virtue of being the right adjoint to a symmetric monoidal functor, is itself right-lax symmetric monoidal.
Thus we can upgrade $\Omega$ to a functor 
$$\Omega: Alg_{E_1}(Cat_\cP)\longrightarrow Alg_{E_1}(Alg_{E_1}(\cP))\simeq Alg_{E_2}(\cP)$$
on monoidal $\cP$-categories, pointed by their units, to $E_2$-algebras in $\cP$.

We now apply this construction to the monoidal functor (morphism of $E_1$-algebras in $Cat_\cP$) given by the action of a monoidal category on itself, $$\otimes: \cA\to End(\cA),$$ obtaining an $E_2$-morphism $End(1_\cA)\to End(Id_\cA)$. 

The functor $\otimes$, considered as a morphism only of pointed $\cP$-categories, admits a left inverse
$$act_{1_\cA}: End(\cA)\longrightarrow \cA$$ obtained from acting on the unit of $\cA$ by endofunctors of $\cA$: the composite $$\xymatrix{\cA\ar[r]^-{\otimes}.  \ar@/^2pc/_-{1_\cA\otimes -}[rr] & End(\cA) \ar[r]^-{act_{1_\cA}} & \cA}$$ is identified with the identity functor of $\cA$ since $1_\cA$ is the monoidal unit.
Applying $\Omega$ to this morphism we obtain the desired left inverse morphism of $E_1$-algebras in $\cP$
$$act_{1_\cA}:End(Id_\cA)\to End(1_\cA).$$
\end{proof}

\subsection{Groupoids}

\begin{definition} By an {\em ind-proper groupoid} we refer to a groupoid object $\cG\actson X$ in ind-stacks, with ind-proper source and target maps $\pi_1,\pi_2:\cG\to X$.
\end{definition}

More precisely, the groupoid object is given by a simplicial object $\cG_\bullet$ satisfying a Segal condition resulting in an identification of the simplices with iterated fiber products:
 \begin{equation}\label{groupoid}
 \xymatrix{\cdots \ar[r]<1ex> \ar[r]<.5ex> \ar[r] \ar[r]<-.5ex> \ar[r]<-1ex> &
 \cG\times_X \cG\times_X \cG \ar[r]<.75ex> \ar[r]<.25ex> \ar[r]<-.25ex> \ar[r]<-.75ex> &
 \cG\times_X \cG \ar[r]<.5ex> \ar[r] \ar[r]<-.5ex> &
 \cG \ar[r]<.25ex> \ar[r]<-.25ex>&
 X}
 \end{equation}
See~\cite[Sections II.2.5.1, III.3.6.3]{GR} for a discussion of ind-proper groupoid objects.
We denote $p=(\pi_1,\pi_2):\cG\to X\times X$.

It will be convenient (but technically irrelevant) to think in terms of the (potentially very poorly behaved) quotient prestack $Y=|\cG_\bullet|=X/\cG$, so that $\cG_\bullet$ is identified with the \v{C}ech simplicial object $\{X\times_Y X\times_Y\cdots\times_Y X\}$.

\begin{remark}[Monoid/Segal objects] Our constructions apply equally well to monoid objects (also known as Segal or category objects) in stacks, rather than groupoids (the setting of the constructions in~\cite[Sections II.2.5.1, III.3.6.3]{GR}) - in other we will make no use of invertibility of morphisms. We use the language of groupoids for psychological reasons, for example to think of $\cG$ as the \v{C}ech construction on a mythical quotient stack $X\to Y$. 
\end{remark}
 
Our main example of an ind-proper groupoid will be the equivariant Grassmannian $\cG=\uGrv$ acting on $X=pt/\LGpv$, i.e., the \v{C}ech construction for the ind-proper, ind-schematic morphism $X=pt/LG^{\vee}_+\to Y=pt/L\Gv$ or its loop rotated version (see Section \ref{subsection-derived satake}).

For the remainder of this section we will fix an ind-proper groupoid $\cG \rightrightarrows X$.

\subsection{Hecke categories}\label{Hecke categories}

\begin{definition} The Hecke category attached to the ind-proper groupoid $\cG\actson X$ is $\cH:=\Dhol(\cG)$. 
\end{definition}

The construction of the monoidal structure, the convolution product, following the general mechanism discussed in~\cite[II.2.5.1, V.3.4]{GR} --- it is inherited on applying $\Dhol$ to the structure on $\cG$ of algebra object in correspondences. Since the pushforward under a proper map is a proper functor (it has a continuous right adjoint), the convolution product is proper, hence the Hecke algebra defines an algebra in $\DGCat_k^c$. 

Explicitly, given objects $A,B \in \cH$, their convolution is given by $A\ast B = p_{13\ast} p_{12}^!(A) \otimes^! p_{23}^! (B)$, where
\[
\xymatrix{
p_{12},p_{13},p_{23} : \cG\times_X \cG  \ar[r]<.5ex> \ar[r] \ar[r]<-.5ex> & \cG
}
\]
are the three projection maps. The diagonal embedding (unit map) $i:X\to \cG$ induces a monoidal functor $$\fd:\cR\to \cH$$ making the monoidal category $\cH$ into a {\em $\cR$-ring}, i.e., algebra object in $\cR$-bimodules. 

\subsection{Hecke algebras}
%
The groupoid $\cG$ defines a monad acting on $\cR=\Dhol(X)$ following the general mechanism discussed in~\cite[II.2.5.1, V.3.4]{GR} which we call the Hecke algebra $\uH$. The Hecke algebra is an algebra object structure on the functor $\pi_{2,*}\pi_1^!\simeq p^!p_*\in End(\cR)$. 

\begin{definition} The Kostant category associated to the groupoid $\cG$ is the category $\cK=\uH\module$ of $\uH$-modules in $\cR=\Dhol(X)$.
\end{definition}

 Alternatively, one can think of $\cK$ as the category of $\cG$-equivariant objects of $\Dhol(X)$. More precisely, since Diagram~\ref{groupoid} is a diagram of ind-stacks and ind-finite type maps, we can pass to $\Dhol$ and $!$-pullbacks to find the cosimplicial symmetric monoidal category $\Dhol(\cG_\bullet)$:
 $$\xymatrix{\cdots &\ar[l]<1ex> \ar[l]<.5ex> \ar[l] \ar[l]<-.5ex> \ar[l]<-1ex> 
 \Dhol(\cG\times_X \cG\times_X \cG)& \ar[l]<.75ex> \ar[l]<.25ex> \ar[l]<-.25ex> \ar[l]<-.75ex> 
 \Dhol(\cG\times_X \cG)& \ar[l]<.5ex> \ar[l] \ar[l]<-.5ex> 
 \Dhol(\cG) &\ar[l]<.25ex> \ar[l]<-.25ex>
 \Dhol(X)}$$

\begin{definition} The symmetric monoidal category $\Dhol(X)^{\cG}$ of $\cG$-equivariant sheaves on $X$ is the totalization $Tot(\Dhol(\cG_\bullet))$. 
\end{definition}

To identify $\Dhol(X)^\cG$ with $\uH$-modules in $\Dhol(X)$, we require the theory of monadic descent, in this setting due to Lurie \cite[Theorem 6.2.4.2]{HA} (see also \cite{1affine}, Appendix C). In general, if a cosimplicial category $\cC^\bullet$ satisfies the monadic Beck-Chevalley conditions, then we can identify the totalization of $\cC^\bullet$ with modules for a monad acting on $\cC^0$, whose underlying functor may be identified with the composite of one face map with the left adjoint of the other. In the case $\cC^\bullet = \Dhol(\cG_\bullet)$, these conditions are equivalent to the base change property for ind-proper morphisms in $\Dhol$ (see Remark \ref{remark-bivariant}), and the corresponding monad is precisely $\uH$. Thus we obtain the following result:
\begin{proposition} We have an identification $\cK\simeq \Dhol(X)^{\cG}$, and hence a symmetric monoidal structure on the Kostant category for which the forgetful functor $\cK\to\cR$ is symmetric monoidal.
\end{proposition}
 
\subsection{Symmetric monoidal structure vs. cocommutative bimonad}
One can view the (symmetric) monoidal structure on $\cK$ in terms of a (cocommutative) bimonad structure on $\uH$ in the sense of Moerdijk and Brugui\`eres-Virelizier, see~\cite{Bohm} (in fact it's naturally a Hopf monad). 

More precisely, the symmetric monoidal structure on $!$-pullback and oplax symmetric monoidal structure on ind-proper $*$-pushforward endow  $\uH=\pi_{2,*}\pi_1^!\simeq \pi^!\pi_*$ with a canonical oplax symmetric monoidal structure. In this way the endofunctor $\uH$ naturally upgrades to a cocommutative bimonad, i.e. an algebra object in the category of oplax symmetric monoidal endofunctors. In particular, $\uH(1_\cR)$ is naturally a cocommutative coalgebra object in $\cR$.

The monoidal structure on $\uH$-module is equivalent to the bimonad structure enhancing the monad $\uH$. Explicitly, given $\uH$-modules $M,N$ with structure maps $\uH M\to M$, $\uH N\to N$, we give $M\ot N$ a $\uH$-module structure with structure map $$\uH (M\ot N)\to \uH M \ot \uH N \to M\ot N$$
where the first morphism uses the oplax monoidal structure on $\uH$. 
Similarly, we have a natural transformation
$$\uH(\omega_X)=\pi_{2,*}\omega_{\cG}\simeq \pi_{2,*}\pi_2^!\omega_X\longrightarrow \omega_X.$$


\subsection{(Bi)monads vs. (bi)algebroids.}\label{monad vs algebra}
In the generality we're working, the Hecke algebra $\uH$ is only a monad, i.e., algebra object in endofunctors of $\Dhol(X)$. In the cases of practical interest however (in particular, for the equivariant Grassmannian) this reduces to an ordinary algebra, thanks to ``affineness" (or rather coaffineness). 

In general, if $R$ is any $k$-algebra object then we can monoidally identify continuous endofunctors of $R\module$ with $R$-bimodules, and thus a continuous monad $\uH$ acting on $\cR =R\module$ is the same thing as an algebra object $H$ in the monoidal category of $R$-bimodules. Unwinding the definitions, we observe that such an algebra object $H$ is nothing more than an $R$-ring, i.e. $H$ is itself a $k$-algebra object, together with a morphism of algebra objects $R\to A$. Moreover, the category of modules for $A$ in $R\module$ (thinking of $A$ as a monad acting on $R\module$) is the same thing as $A\module$, i.e. $A$-modules in $\Vect$ (see \cite{Bohm} Lemma 2.4). Moreover, if $R$ is a commutative ring, then a (cocommutative) bimonad structure on the monad $\uH$ is equivalent to the structure of a (cocommutative) $R$-bialgebroid on the $R$-ring $H$. In that case, we have an identification of left $R$-modules $H = \uH(R)$; in particular $H$ is a cocommutative $R$-coalgebra object.

Returning to the setting of an ind-proper groupoid $\cG$ acting on $X$, let us consider the case where the functor $p_{X\ast}:\Dhol(X) \to \Vect$ is monadic, so that $\Dhol(X) \simeq C^\ast(X)\module$ (this happens for example in the case when $X$ is the classifying stack of an algebraic group). 
In this case, the Hecke monad $\uH$ on $\Dhol(X)$ may be identified as a $C^\ast(X)$-ring which we denote $H$. By construction $H$ is given by global sections of the relative dualizing complex $\omega_{\cG/X} = \pi_1^! p^\ast(\C) \simeq \pi_1^! p^\ast(\C)$. 

Unwinding the definitions, we see that the $R$-ring structure on $H$ arises from convolution of (relative) chains on $\cG$. For example, the multiplication $H\otimes_R H \to H$ is given by direct image of chains along the ind-proper morphism
\[
\cG \times_{\pi_2,X,\pi_1} \cG \to \cG
\]
On the other hand, the $R$-coalgebra structure on $H$ arises from ``cup coproduct'' of chains. For example, the comultiplication $H \mapsto  H\otimes_R H$ on $H$ (where the commutative ring $R=R^{op}$ acts on both factors of $H$ by left multiplication) is given by pushforward associated to the diagonal map
\[
\cG \to  \cG \times_{\pi_1,X,\pi_1}  \cG
\]
Note that the fiber product involved in the cup coproduct is defined using $\pi_1$ on both factors, in contrast to the fiber product involved in convolution.

\begin{remark}
Note that as a dg vector space, $H$ may be unbounded in both directions. For example, in the case $\cG = \uGrv$, $X=pt/\LGpv$. Then $C^\ast(X)$ is the $G^\vee$ equivariant cohomology ring of a point (thus unbounded in positive cohomoglical degrees) and $C_{-\ast} (pt\times_X \cG) = C_{-\ast}(\cGr)$ is the (non-equivariant) homology of the affine Grassmannian (thus unboundeed in negative cohomological degrees). Equivariant formality gives an isomorphism of dg-vector spaces (in fact of $C^\ast(X)$-coalgebras) $$H \simeq C^\ast(X) \otimes C_{-\ast}(\cGr)$$
\end{remark}


%

\subsection{Modules for Hecke Categories}
We now consider a categorical analog of the above discussion.  

Consider the cosimplicial symmetric monoidal category $\Dhol(\cG_\bullet)$. We may pass to module categories, obtaining a cosimplicial symmetric monoidal category $\Dhol(\cG_\bullet)\modcat$. 

\begin{definition} The symmetric monoidal category $\Dhol(X)\modcat^{\cG}$ of $\cG$-equivariant module categories on $X$ is the totalization $Tot(\Dhol(\cG_\bullet)\modcat)$. 
\end{definition}

\begin{remark}[Algebra vs Monad, revisited]
As we noted in Remark~\ref{monad vs algebra}, we treat the Hecke algebra in general as a monad on $\Dhol(X)$, but in situations of interest this reduces to an algebra object in $C^*(X)$-bimodules. Here we chose to treat the Hecke category directly as an algebra in $\Dhol(X)$-bimodules. One could instead consider the monad on sheaves of categories on $X$ obtained by push-pull along $\cG$. Likewise the category $\Dhol(X)\modcat^{\cG}$ is an avatar for the category of $\cG$-equivariant sheaves of categories on $X$, with which it is connected by the localization-global sections adjunction, and which it would recover if we were in a 1-affine situation. Thus we can also consider it as an avatar of sheaves of categories on the quotient stack $Y=X/\cG$, which is the source of its symmetric monoidal structure.
\end{remark}

\begin{proposition} The cosimplicial category $\Dhol(\cG_\bullet)\modcat$ satisfies the monadic Beck-Chevalley conditions.
Moreover the associated monad on $\Dhol(X)\modcat$ is identified with the {\em Hecke category} $\cH=\Dhol(\cG)$ as an algebra in $\Dhol(X)$-bimodules via the diagonal map $\delta_*:\Dhol(X)\to \cH$. Thus we have an identification $\Dhol(X)\modcat^{\cG}\simeq \cH\modcat$.
\end{proposition}

\begin{proof}
The Beck-Chevalley conditions for $\Dhol(\cG_\bullet)\modcat$ follow from those for $\Dhol(\cG_\bullet)$ upon applying the functor $\modcat$.
\end{proof}

It follows that the category $\cH\modcat$ of $\cG$-equivariant $\Dhol(X)$-modules inherits a symmetric monoidal structure, such that the forgetful functor $\cH\modcat\to\cR\modcat$ is symmetric monoidal. The unit object is the $\cH$-module $\cR$ 
itself, which corresponds to the cosimplicial category $\Dhol(\cG_\bullet)$.

\subsection{Hecke algebras vs. Hecke categories}
We now compare descent for module categories with descent for sheaves. 
Given a $\cH$-module $\cM$, or equivalently $\cM^\bullet \in Tot(\Dhol(\cG_\bullet)\modcat)$, we define the {\em $\cG$-equivariant objects} $\cM^{\cG}$ to be 
$$\cM^{\cG}:=Hom_{\cH}(\Dhol(X), \cM).$$ Thus we have 
$$\cM^{\cG}\simeq Tot(Hom(\Dhol(\cG_\bullet), \cM^\bullet)).$$

\begin{proposition}\label{Hecke equivariance}
\begin{enumerate}
\item The $\cG$-equivariant objects in the $\cH$-module $\cR$ recover the category of $\cG$-equivariant sheaves on $X$, i.e.,
$$\cR^\cG\simeq \cK.$$
\item The resulting equivalence of $\cR^\cG$ with the endomorphisms of the unit $\cR$ of the symmetric monoidal
category $\cH\modcat$ lifts to a symmetric monoidal equivalence. 
\end{enumerate}
\end{proposition}

\begin{proof}
We apply the above definition in the case $\cR=\Dhol(X)$, which corresponds to $\cR^\bullet=\Dhol(\cG_\bullet)$: 
\begin{eqnarray*}
[\Dhol(X)]^{\cG}&:=&Hom_{\cH}(\Dhol(X),\Dhol(X))\\
&\simeq& Hom_{Tot(\Dhol(\cG_\bullet)\modcat)}(\Dhol(\cG_\bullet),\Dhol(\cG_\bullet))\\
&\simeq& Tot(\Dhol(\cG_\bullet))\\
&\simeq& \Dhol(X)^{\cG}.
\end{eqnarray*}

Tracing through the identifications above, we see that the symmetric monoidal structure on $Tot(\Dhol(\cG_\bullet))$ coming from tensor product of sheaves is identified with the symmetric monoidal structure on endomorphisms of the unit in $Tot(\Dhol(\cG_\bullet)\modcat)$, as claimed.
\end{proof}

Our main result asserts that $\cG$-equivariant sheaves give central objects in the groupoid category $\cH$. This central action can be thought of as expressing the linearity of convolution on $\cG=\XYX$ over sheaves on the (possibly ill-behaved) quotient $Y=X/\cG$.

\begin{theorem}\label{groupoid center}
Let $\cG$ denote an ind-proper groupoid acting on an ind-stack $X$, $\uH$ the corresponding monad on $\cR=\Dhol(X)$, $\cK=\uH\module$ the Kostant category and $\cH=\Dhol(\cG)$ the groupoid category. 
Then there is a canonical $E_2$-morphism $\fz$ with a monoidal left inverse $\fa$ ($\fa\circ \fz\simeq Id$), 
$$\xymatrix{\cK\ar[r]_-{\fz}& \ar@/_1pc/_-{\fa}[l]\cZ(\cH)}$$ 
lifting the diagonal map $\fd:\cR\to \cH$:

$$\xymatrix{\cK\ar[r]_{E_2}\ar[d]_-{E_\infty}& \ar@/_1pc/_-{E_1}[l] \cZ(\cH)\ar[d]^{E_1}\\
\cR \ar[r]_{E_1} & \cH  }$$
\end{theorem}

\begin{proof}
We apply Corollary~\ref{nonsense} in the setting of the presentable symmetric monoidal category
$\cP=\DGCat_k^c$ of compactly generated dg categories with proper morphisms.
For the algebra object $\cA\in Alg_{E_1}(Cat_\cP)$ (which in our case happens to be a commutative algebra object) we 
take $$\cA=\cH\modcat\simeq \cR\modcat^{\cG}$$ to be the category of modules for the Hecke category, i.e., $\cG$-equivariant $\cR$-modules. The center $End(Id_\cA)$ of $\cA$ is identified with the center $\cZ(\cH)$ of the monoidal category $\cH$. We have identified $$End(1_\cA)\simeq\uH\module\simeq \Dhol(X)^{\cG}$$ as categories. We need to show that this identification can be upgraded to an $E_2$ identification, hence obtaining the desired $E_2$-morphism from $\cK=\uH\module$ to $End(Id_\cA)=\cZ(\cH)$. However we have seen in Proposition~\ref{Hecke equivariance} that the identification is in fact naturally $E_\infty$. Thus Corollary~\ref{nonsense} provides the desired morphisms $\fz$ and $\fa$. 

To conclude the theorem, we only need to establish that the morphism $\fz$ lifts the morphism $\fd$ (i.e., the commutativity of the above diagram). Note that the monoidal functor $\fd: \cR \to \cH$ (which defines the structure of $\cH$ as an $\cR$-module) induces a corresponding functor $\End(\cR) \to \End(\cH)$ which we still denote by $\fd$. By construction, the functor $\fz: \cK \to \cZ(\cH)$
takes an object of $\cK$, represented by a $\cH$-linear endomorphism $F:\cR \to \cR$ to $\fd(F)$, which has the structure of an $\cH\otimes \cH^{op}$-linear endomorphism of $\cH$, i.e. an object of $\cZ(\cH)$. In other words, we have a commutative diagram
\[
\xymatrix{
\End_{\cH}(\cR) \ar[r]^{\fz} \ar[d] & \End_{\cH \otimes \cH^{op}}(\cH) \ar[d] \\
\End(\cR) \ar[d]_{act_{1_\cR}} \ar[r]^{\fd} & \End(\cH) \ar[d]^{act_{1_\cH}} \\
\cR \ar[r]^\fd & \cH
}
\]
as required.

\end{proof}

\section{Sheaf theory: Filtered $D$-modules}\label{filtered D-mod section}
In the previous two sections, we considered categories of ind-holonomic $D$-modules in the setting of ind-proper groupoid stacks. The main example was the equivariant affine Grassmannian $\uGrv \rightrightarrows pt/\LGpv$ associated to the group $\Gv$. In this section we will discuss the relevant sheaf theory for the Langlands dual side, which involves finite dimensional geometry associated to the group $G$. In this setting we will be using the category of all (not-necessarily holonomic) $D$-modules (rather than its ind-holonomic variant)
\[
\cD(Y) = \QC^!(Y_{dR})
\]
and we will need to understand the degeneration of this category to $\QC(T^\ast Y)$ quasi-coherent sheaves
on the cotangent bundle.

\subsection{Categorical Representation Theory}\label{categorical representation theory}
Let $G$ be a fixed affine algebraic group. In this subsection, we give a brief overview of the theory of $G$-actions in the setting of dg or stable, presentable $\infty$-categories (see \cite{BD,frenkel dennis,1affine} as well as \cite[Section 3]{dario} and the references therein for further details).

Consider the category $\cD(G)$ of $\cD$-modules of $G$, equipped with the convolution monoidal structure. A strong $G$-category is, by definition, a module category for $\cD(G)$. Examples of such include $\cD(X)$ for a stack $X$ with a $G$-action, and $\fU(\fg)\module$. Given a strong $G$-category $\cC$, we have its (strong) invariants $\cC^G = \Hom_{\cD(G)}(\Vect,\cC)$. This is computed as the totalization of a cosimplicial category
\[
 \xymatrix{
\cdots& \ar[l]<.75ex> \ar[l]<.25ex> \ar[l]<-.25ex> \ar[l]<-.75ex> 
\cC \otimes \cD(G) \otimes \cD(G) & \ar[l]<.5ex> \ar[l] \ar[l]<-.5ex> 
\cC \otimes \cD(G) &\ar[l]<.25ex> \ar[l]<-.25ex>
\cC
}
\]

A weak $G$-category is defined to be a module category for the convolution category $\QC(G)$, and we denote the weak invariants of a weak $G$-category $\cC$ by $\cC^{G,w} := \Hom_{\QC(G)}(\Vect,\cC)$, which can be computed using a similar diagram. Given a weak $G$-category $\cC$ its weak invariants $\cC^{G,w}$ naturally carries an action of the rigid symmetric monoidal category $\Rep(G) = \QC(pt/G) = \Hom_{\QC(G)}(\Vect,\Vect)$. It is a result of Gaitsgory \cite{1affine} that $pt/G$ is a 1-affine stack: quasi-coherent sheaves of categories on $pt/G$ are identified with module categories $\QC(pt/G) = \Rep(G)$.  By descent, sheaves of categories on $pt/G$ are identified with module categories for $(\QC(G),\ast)$, leading to the following interpretation of 1-affineness:
\begin{theorem}[Gaitsgory's 1-affineness] \label{BG1affine}
The $\QC(G)$-$\Rep(G)$ bimodule $\Vect$ defines a Morita equivalence between the monoidal categories $(\QC(G), \ast)$ and $(\Rep(G), \otimes)$. 
\end{theorem}
In other words a weak $G$-category can be recovered from its weak invariants as a $\Rep(G)$-module category.


If $\cC$ is a strong $G$-category, then in particular it is a weak $G$-category, and we have
\[
\cC^{G,w} = \Hom_{\QC(G)}(\Vect,\cC) \simeq \Hom_{\cD(G)}(\fU\fg\module,\cC)
\]
In the case $\cC=\cD(X)$ for a smooth stack $X$ with a $G$-action, we have identifications $\cD(X)^G \simeq \cD(X/G)$, and $\cD(X)^{G,w} \simeq \cD(\qw XG) = \QC(X_{dR}/G)$ (this is smooth descent). Note that
\[
\cD(X)^G = \Hom_{\cD(G)}(\Vect,\cD(X)) \simeq \Hom_{\HC}(\Rep(G),\cD(X)^{G,w})
\]
This is a derived rephrasing of familiar equivalence between strongly equivariant $\cD$-modules and weakly equivariant $\cD$-modules for which the quantum moment map is identified with the derivative of the $G$-action.

Consider the monoidal category of Harish-Chandra bimodules:
\[
\HC = \Hom_{\cD(G)}(\fU\fg\module,\fU\fg\module) = \left(\fU\fg\bimod \right)^G \simeq \cD(\wqw GGG)
\]
Objects of $\HC$ are given by $\fU\fg$-bimodules in $\Rep(G)$, together with an identification of the adjoint $\fU(\fg)$-action with the derivative of the $G$-action.\footnote{Note that, while the abelian category heart of $\HC$ is a full subcategory of $\fU\fg$-bimodules, in the derived setting strong equivariance is data not a condition (even for $G$ connected).} If $\cC$ is a strong $G$-category then $\cC^{G,w} = \Hom_{\cD(G)}(\fU\fg\module,\cC)$ is naturally a $\HC$-module category. Using the 1-affineness of $pt/G$, we have
\begin{theorem}[Beraldo  \cite{dario}]\label{thm-morita}
The $\cD(G)$-$\HC$-bimodule $\fU\fg\module$ defines a Morita equivalence between the monoidal categories $\cD(G)$ and $\HC$.
\end{theorem}
In other words, a strong $G$-category can be recovered from its weak invariants as a $\HC$-module category. 

\begin{corollary}\label{center of HC}
There are equivalences of $E_2$-categories
\[
\cZ(\cD(G)) \simeq \cD(G\adjquot G) \simeq \cZ(\HC)
\]
\end{corollary}

\begin{remark}
The forgetful functor 
\[
\cD(G\adjquot G) \simeq \cZ(\HC) \to \HC \simeq (\fU\fg \otimes \fU\fg)^G
\]
takes a $G$-equivariant $D$-module on $G$ to its underlying $\fU\fg \otimes \fU\fg$-module via the algebra map $\fU\fg \otimes \fU\fg \to \fD_G$; the $G$-equivariant structure ensures that the adjoint action of $\fU\fg$ is integrable.
\end{remark}

As an example of a strong $G$-category, suppose $K$ is an algebraic subgroup of $G$. The homogeneous space $G/K$ carries a $G$-action, and thus $\cD(G/K)$ is a strong $G$-category. The corresponding $\HC$-module is the category of Harish-Chandra $(\fg,K)$-modules 
\[
(\fg,K)\module = \fU(\fg)\module^{K} \simeq \cD(\wq GG/K)
\]

In particular, the symmetries of $(\fg,K)\module$ as an $\HC$-module category can be identified as follows
\[
\End_{\HC}((\fg,K)\module) \simeq \End_{\cD(G)}(\cD(G/K)) \simeq \cD(G/K \times G/K)^G \simeq \cD(\quot KGK)
\]
where the right hand side is considered as a monoidal category with respect to convolution.

\subsection{Graded and filtered lifts of categories}
In the case $G=\G_m$, Gaitsgory's Theorem  \ref{BG1affine} says that for a given category $\cC$ in $\DGCat$, the following data are equivalent:
\begin{itemize}
\item A quasi-coherent sheaf of categories on $pt/\G_m$ whose pullback to $pt\to pt/\G_m$ is identified with $\cC$.
\item A weak $\G_m$ action on $\cC$.
\item A module category $\cC_{gr}$ for $\Rep(\G_m) = \Vect_{gr}$ with an identification $\cC_{gr} \otimes_{\Vect_{gr}} \Vect \simeq \cC$.
\end{itemize}
We will refer to the category $\cC_{gr}$ as a \emph{graded lift} of $\cC$, and the forgetful functor $\cC_{gr} \to \cC$ as a \emph{degrading functor}. For example, if $A$ is a graded algebra (i.e. algebra object in $\Vect_{gr}$), then the category $A\module_{gr}$ consisting of dg-modules for $A$ equipped with an external grading, is a graded lift of $A\module$.

Similarly, the 1-affineness of $\A^1/\G_m$\footnote{Recall that in this paper the action of $\G_m$ on a vector space (for example, $\A^1 = \Spec\C[t]$) has weight 2.} implies that the following data are equivalent:
\begin{itemize}
\item A quasi-coherent sheaf of categories on $\A^1/\G_m$ whose pullback to $pt \simeq \A^1-\{0\}/\G_m\to \A^1/\G_m$ is identified with $\cC$.
\item A module category $\cC_{t,gr}$ for $\QC(\A^1/\G_m) = \C[t]\module_{gr}$ with an identification 
\[
\cC_{t,gr}[t^{-1}] = \C[t,t^{-1}]\module_{gr} \otimes_{\C[t]\module_{gr}} \cC_{t,gr} \simeq \cC.
\]
\end{itemize}
We refer to $\cC_{t,gr}$ as a \emph{filtered lift} of the category $\cC$. To such a data, we have an associated graded category $\cC_{t=0,gr}$, and also an associated asymptotic category $\cC_t$ which is a degrading of $\cC_{t,gr}$. For example, if $A= \bigcup_{i\in \Z} A_{\leq i}$ is a filtered algebra, then the Rees algebra $A_t := \bigoplus_{i\in \Z} A_{\leq i}t^i$ is a graded $\C[t]$-algebra, and the category $A_t\module_{gr}$ is a filtered lift of $A\module$. The associated graded category is the category of graded modules for the associated graded algebra $A_{t=0}$. The associated asymptotic category  $A_t\module$ is given by (ungraded) modules for the Rees algebra.

\begin{example}
Suppose $X$ is an Artin stack; then the category $\cD(X)=\QC^!(X_{dR})$ has a filtered lift $\cD_{t,gr}(X)$ given by $\QC(X_{Hod})$, where $X_{Hod} \to \A^1/\G_m$ is the Hodge stack of $X$. The associated graded category is given by
\[
\cD_{t=0,gr}(X) \simeq \QC^!(T[1]X)_{gr} \simeq \QC(T^\ast X)_{gr}
\]
In the case when $X$ is a smooth affine algebraic variety, we have $\cD(X) = \fD_X\module$, and $\cD_{t,gr}(X)$ is equivalent to $\fD_{X,t}\module_{gr}$, where $\fD_{X,t}\module_{gr}$ is the Rees algebra of $\fD_{X}$, as explained above.
In particular, returning to the main setting of this section, we have monoidal categories $\HC_{t,gr}$, $\cD_{t,gr}(G)$ which define filtered lifts of $\HC$ and $\cD(G)$. The same proof as in \cite{dario} gives that $\HC_{t,gr}$ is Morita equivalent to $\cD_{t,gr}(G)$, and the center of $\HC_{t,gr}$ is $\cD_{t,gr}(G\adjquot G)$.
\end{example}

\subsection{Shearing}
In the examples relevant to this paper, the original filtered algebra $A$ will be supported in cohomological degree $0$ (i.e. it is an ordinary algebra, not a dg-algebra). In that case, the Rees algebra $A_t$ also sits in cohomological degree $0$, but carries a non-trivial external (weight) grading for which the Rees parameter $t$ has weight 2. We will be interested in another form of the Rees algebra $A_\hbar = \bigoplus A_{\leq i}\hbar^i$ where an element of homogeneous weight $i$ sits in cohomogical degree $i$; in particular, the Rees parameter $\hbar$ now sits in cohomological degree 2. Note that $A_\hbar$ is a dg-algebra in general, even when the original algebra $A$ lives in cohomological degree $0$. 

\begin{remark}
Throughout this paper, $\C[t]$ will always refer to a polynomial algebra in which the variable $t$ has cohomological degree $0$ and weight 2; on the other hand, $\C[\hbar]$ always refers to a polynomial algebra in which $\hbar$ has cohomological degree $2$ and weight $2$.
\end{remark}

The categories of graded $A_t$-modules and of graded $A_\hbar$-modules are related by the notion of \emph{shearing}. The fundamental result is: 
\begin{lemma}
There is a symmetric monoidal autoequivalence of $\Vect_{gr}$ called \emph{shearing}
defined by
\[
M = \bigoplus_i M_i \mapsto M^{\fatslash} := \bigoplus_{i \in \Z} M_i[-i]
\]
with inverse
\[
N = \bigoplus N_i \mapsto N^\fatbslash := \bigoplus_{i \in \Z} N_i[i]
\]
\end{lemma}
Note that $\fatslash$ has the property that it takes a an ordinary graded vector space (i.e. a graded dg-vector space concentrated in cohomological degree $0$) to a dg-vector space for which the weight on the cohomology agrees with the cohomological degree.

\begin{remark}[Formality]\label{formality remark}
The shearing autoequivalence is related to a well-known criterion for formality of a dg-algebra. 

Recall that taking cohomology objects defines a symmetric monoidal endofunctor $H^\ast$ of $\Vect$ which takes $A$ to $\bigoplus_{i\in \Z} H^i(A)[-i]$.\footnote{Note that the underlying functor of $H^\ast$ is equivalent to the identity functor, but it carries an interesting monoidal structure.}  A (co)algebra object $A$ in $\Vect$ is called formal if there is an equivalence of (co)algebra objects $A\simeq H^\ast(A)$. Now suppose $R$ is an (co)algebra object of $\Vect$ which carries an external grading such that the weight of $H^i(R)$ is $i$. Then $R^\fatbslash$ is concentrated in cohomological degree $0$, and in particular $R^\fatbslash$ is formal: $R^\fatbslash \simeq H^0(R^\fatbslash)$. It follows that the original (co)algebra is formal: $R \simeq \bigoplus H^i(R)[-i]$.
\end{remark}

Twisting by the shearing autoequivalence leads to the following result:
\begin{lemma}
There is an equivalence of graded, symmetric monoidal categories 
\[
\fatslash: \C[t]\module_{gr} \simeq \C[\hbar]\module_{gr}
\]
\end{lemma}
In particular, given a filtered lift of a category $\cC$, there is a corresponding $\C[\hbar]\module_{gr}$-module category $\cC_{\hbar,gr}$ with an equivalence $\cC_{t,gr} \simeq \cC_{\hbar,gr}$. Thus there is a ``sheared'' degrading functor to the associated dg-asymptotic category: 
\[
\xymatrix{
\cC & \ar[l] \cC_{t,gr} \ar[r] & \cC_\hbar.
}
\]
\begin{remark}
In the case $\cC = A\module$ for a filtered ordinary algebra $A$ (i.e. concentrated in cohomological degree $0$), we have that $\cC = A\module$, $\cC_{t,gr} = A_t\module_{gr}$, and $\cC_t$ carry a natural $t$-structure, and are each equivalent to the dg derived category of the corresponding abelian categories appearing as the heart. On the other hand, $\cC_\hbar = A_\hbar\module$ does not carry a $t$-structure which makes the degrading functor $\cC_{t,gr} \to \cC_\hbar$ $t$-exact in general.
\end{remark}

\subsection{Filtered categorical representation theory}
We have monoidal categories $\HC_{t,gr}$, $\cD_{t,gr}(G)$ which are filtered lifts of $\HC$ and $\cD(G)$. The same proof as in \cite{dario} gives the following:
\begin{theorem}\label{Morita equivalence filtered}
There is a Morita equivalence between the monoidal categories $\HC_{t,gr}$ and $\cD_{t,gr}(G)$. 
There is an $E_2$-monoidal equivalence of categories
\[
\cZ(\cD_{t,gr}(G)) \simeq \cD_{t,gr}(G\adjquot G) \simeq \cZ(\HC_{t,gr})
\]
\end{theorem}
Applying the degrading functors, we get the corresponding statement for the $\hbar$-versions: $\cZ(\HC_{\hbar}) \simeq \cD_\hbar(G\adjquot G)$. 

\section{The Spherical Hecke Category and quantum Ng\^o action}\label{quantum section}
In this section, we translate the results of Section~\ref{Hecke section} through the Geometric Satake equivalence. Throughout this section $G$ will be a complex reductive group with a fixed Borel subgroup $B$, $N=[B,B]$, and $H=B/N$. The corresponding Lie algebras are denoted $\fg$, $\fb$, $\fn$, and $\fh$ respectively. The Langlands dual group will be denoted $\Gv$, with loop group $\LGv$ and arc group $\LGpv$. 

\subsection{The Characteristic Polynomial Map and Kostant Section}\label{char poly section}
Following \cite{Ngo}, Section 2, let us recall some constructions arising from the the diagram of stacks
\begin{equation}\label{Kostant section}
\xymatrix{
	\fg^\ast/G \ar[r]_{\chi} & \ar@/_1pc/[l]_{\kappa} \fc
}
\end{equation}
where $\chi \circ \kappa = id_{\fc}$. Here $\chi$ is the canonical map $\fg^\ast/G \to \fc := \Spec(\Sym(\fg)^G)$, which we call the characteristic polynomial map. The Kostant section, $\kappa$ can be constructed as follows. Let $\psi: \fn \to \C$ denote a character which is non-zero on every simple root space, and denote by $\mu: \fg^\ast \to \fn^\ast$ the projection map (which is also the moment map for the adjoint action of $N$ on $\fg^\ast$). Then Kostant \cite{Kostant Whittaker} showed that the action of $N$ on $\mu^{-1}(\psi)$ is free, and the composite 
\[
\xymatrix{
\fg^\ast \GIT_{\psi} N := \mu^{-1}(\psi)/N \ar[r]& \fg^\ast/G \ar[r]^\chi& \fc
}
\]
is an isomorphism, providing the desired section $\kappa$ of $\chi$.

The restriction of $i$ to the regular locus $\fg^\ast_\reg/G \to \fc$ is a gerbe for the abelian group scheme $J \to \fc$, trivialized by $\kappa$ (thus $J = \fc \times_{\kappa}\fc$). Alternatively, $J$ may be realized as $\kappa^\ast  I$ where $ I$ the inertia stack of $\fg^\ast/G$: informally
\[
 I = \left\{(g,x) \in G\times \fg^\ast \mid coAd_g(x)=x \right\}/G
\]

Now consider the multiplicative group $\G_m$ acting on $\fg^\ast$ by scaling with weight 2 (throughout this paper, the scaling action of $\G_m$ on a vector space will always have weight 2, or equivalently, polynomial rings will be considered as graded rings generated in degree 2). This action commutes with the coadjoint action and the characteristic polynomial map $\chi$ is equivariant for the $\G_m$ action, where $\G_m$ acts on $\fc$ by twice the exponents of the Lie algebra $\fg$. It is not immediately clear that the Kostant section is equivariant for this $\G_m$ action, as $\mu^{-1}(\psi)/N$ is not preserved under scaling. However, as explained in \cite[Section 2]{Ngo}, there is a diagram of stacks
\[
\xymatrix{
	\fg^\ast/G\times \G_m \ar[rr]_{\chi/\G_m} && \ar@/_1pc/[ll]_{\kappa/\G_m} \fc/\G_m
}
\]
where the equivariance data of $\kappa/\G_m$ is defined via the homomorphism $\G_m \to G\times \G_m$ given by $(2\rho,1)$, where $2\rho$ refers to the sum of the simple coroots.

In order to explain why the $2\rho$ appears above let us give another construction of the Kostant slice, which has the additional advantage of not requiring a choice of the character $\psi$. Let $\fn' = \fn/[\fn,\fn]$ denote the maximal abelian quotient, so $\ch = (\fn')^\ast$ is identified with the space of characters of $\fn$. The torus $T=B/N$ acts on $\ch$, which has a one dimensional weight space for each negative simple root. There is a unique open dense orbit $\ch^\circ$ on which $T$ acts simply transitively; the elements of $\ch^0$ correspond precisely to the possibly choices of $\psi$ above.

Any choice of $\psi \in \ch^\circ$ defines a slice to the $T$-action on $\mu^{-1}(\ch^\circ)/N$. Thus the composite
\[
\mu^{-1}(\psi)/N \to \mu^{-1}(\ch^\circ)/N \to \mu^{-1}(\ch^\circ)/B
\]
is an isomorphism. Note that $\G_m$ acts on the right hand side compatibly with the map to $\fg^\ast/G$. If we use the isomorphism above to translate the $\G_m$-action to $\mu^{-1}(\psi)/N$, we see that under the map $\mu^{-1}(\psi)/N \to \fg^\ast/G$ has a $\G_m$-equivariant structure using the homomorphism $\G_m \to \G_m \times G$ given by $(1,2\rho)$, recovering the description above.  

\subsection{The group scheme of regular centralizers}\label{J section}

Recall that the fiber product $J = \fc \times_{\fg^\ast/G,\kappa} \fc$, which is a priori a groupoid acting on $\fc$, is in fact a commutative group scheme over $\fc$. Its fiber over an element $a\in \fc$ is the centralizer of $\kappa(a) \in \fg$.
\begin{lemma}
	We have an isomorphism of groupoids over $\fc$
	\[
	J \simeq \GITquot{N_\psi}{T^\ast G}{\lsub{\psi}N}
	\]
\end{lemma}
\begin{proof}
(See also~\cite[Theorem 6.3]{teleman}.)
Note that the operation of Hamiltonian reduction is a composite of taking a closed fiber and a quotient by a group action. As both these operations commute with fiber products, we have 
	\[
	J = (\fg^\ast \GIT_{\psi} N) \times_{\fg^\ast/G} (\fg^\ast \GIT_{\psi} N) \simeq \GITquot{N_\psi}{(\fg^\ast \times_{\fg^\ast/G} \fg^\ast)}{\lsub \psi N}
	\]
compatible with the projection maps to $\fc$, as required. 
\end{proof}

Note that $\QC(J)$ has a monoidal structure arising from the  convolution diagram
\[
\xymatrix{
J\times J & \ar[l] J\times_\fc J \ar[r]& J
}
\]
As the group structure on $J$ is commutative, this monoidal structure is naturally symmetric. As $J$ is affine, $\QC(J) = \C[J]\module$, where $\C[J]$ has the structure of a commutative and cocommutative Hopf algebra over $\C[\fc]$. 

As in the previous section, we can identify $\fg^\ast\GIT_{\psi} N$ with $\mu^{-1}(\ch^\circ)/B$. As the latter carries a $\G_m$-action, so does the fiber product
\[
J = \mu^{-1}(\ch^\circ)/B  \times_{\fg^\ast/G}  \mu^{-1}(\ch^\circ)/B
\]
In particular, the coordinate ring of $J$ is a graded ring (note that the grading is only in even degrees, but is generally unbounded in both positive and negative degrees). Thus we have a symmetric monoidal category $\QC(J)_{gr}$ of graded $\C[J]$-modules (with respect to convolution).

\subsection{Bi-invariant differential operators: the quantum characteristic polynomial map}
Recall that the ring of bi-invariant differential operators
\[
\fZ\fg = (\fD_{G})^{G\times G} = \fU\fg^G
\]
is a commutative ring, which is identified with the center of left invariant differential operators $\fU\fg = (\fD_G)^G$. We write $\cZ = \fZ\fg\module$ for the symmetric monoidal category of modules. The filtration on $\fD_G$ by order of differential operator defines PBW filtrations on $\fU\fg$ and $\fZ\fg$, and we write $\fU_t\fg$ and $\fZ_t\fg$ for the corresponding Rees algebras. The Duflo/Harish-Chandra isomorphisms define equivalences of filtered algebras 
\[
\fZ\fg \simeq \Sym(\fg)^G \simeq \Sym(\fh)^W
\]
Thus we have $\cZ_{t,gr} \simeq \QC(\fc\times \A^1)_{gr}$. 

There is a natural monoidal functor
\[
\xymatrix{
\Char_{t,gr}: \cZ_{t,gr} \ar[r] & \HC_{t,gr} 
}
\]
given by
\[
\xymatrix{
\fZ_t\fg\module_{gr} \ar[r]& \cD_{t,gr}(\wqw GGG) \ar[r] &(\fU_t\fg\module_{gr})^{G,wk} \\
\fM \ar@{|->}[r] & \fD_{G,t} \otimes_{\fZ_t\fg} \fM \ar[r] & \fU_t\fg \otimes_{\fZ_t\fg} \fM
}
\]
Setting $t=0$, we recover the symmetric monoidal functor
\[
\Char_{t=0,gr} = \chi_{gr}^\ast: \QC(\fc)_{gr} \to \QC(\fgx/G)_{gr}
\]
Thus $\Char_{t,gr}$ is thought of as a quantization of the characteristic polynomial map $\chi$.


\subsection{Whittaker modules: the quantum Kostant slice}
We consider a twisted variant of the category $(\fg,K)\module$ defined in Subsection \ref{categorical representation theory}.

Let $\psi:\fn =Lie(N) \to \C$ be a Lie algebra character. This gives rise to a monoidal functor $\cD(N) \to \Vect$ (a ``categorical character''); we denote the corresponding $\cD(N)$-module category $\Vect_\psi$. Given a strong $N$-category $\cC$, we define the $(N,\psi)$-semi-invariants $\cC^{N,\psi} \simeq \Hom_{\cD(N)}(\Vect_\psi,\cC)$. In particular, we have the category $\cD(X/_\psi N) \simeq \cD(X)^{N,\psi}$ of $(N,\psi)$-twisted equivariant $\cD$-modules on a $N$-space $X$. We also have the category $(\fg,N,\psi)\module = \fU\fg\module^{N,\psi}$ of $(N,\psi)$-Whittaker modules studied in~\cite{Kostant Whittaker}, consisting of $\fU(\fg)$-modules with a compatible action of $N$, together with an identification of the deriviative of the $N$-action with the $\fU(\fn)$-action twisted by $\psi$. 

Given an object of $(\fg,N,\psi)\module$, its space of (derived) 
$N$-invariants 
 (known as Whittaker vectors) carries an action of the center $\fZ\fg$ of $\fU\fg$, and we have the following extension of the results of~\cite{Kostant Whittaker}, known as the Skryabin equivalence  \cite{Premet} (see also \cite[Theorem 6.1]{Gan-Ginzburg}):

\begin{theorem}[Skryabin's equivalence,] 
Suppose $\psi$ is generic. Then the functor of taking Whittaker vectors is a $t$-exact equivalence of categories
\[
(\fg,N,\psi)\module \xrightarrow{\sim} \cZ
\]
\end{theorem}
\begin{remark}
The object $\fU\fg \otimes_{\fU\fn} \C_\psi$ is a compact generator of the category of Whittaker modules, which represents the functor of taking Whittaker invariants. The theorem can be interpreted as saying that $\fU\fg \otimes_{\fU\fn} \C_\psi$ is a projective generator of the abelian category of Whittaker modules, and its endomorphism ring $\fU\fg \GIT_{\psi} N$ is isomorphic to $\fZ\fg$.
\end{remark}

Using the Skryabin equivalence, we have an action of the monoidal category $\HC$ on $\cZ \simeq (\fg,N,\psi)\module$, which can be considered as a quantum form of the Kostant slice.
In \cite{BezFink}, the authors define a filtered lift of this $\HC$-module category, i.e. an action of $\HC_{t,gr}$ on $\cZ_{t,gr}$, or equivalently, a monoidal functor
\[
\xymatrix{
\Whit_{t,gr}: \HC_{t,gr} \to \End_{\QC(\A^1_t/\G_m)}(\cZ_{t,gr}) \simeq \fZ_t\fg \otimes_{\C[t]} \fZ_t\fg\module_{gr}
}
\]
where the right hand side has a monoidal structure coming from identifying with $\fZ_t\fg$-bimodules in $\C[t]\module$. Specializing to $t=0$ we recover the functor of restriction under the graded Kostant slice: 
\[
\xymatrix{
\QC(\fg^\ast/G)_{gr} \ar[r]^{\kappa^\ast_{gr}}& \QC(\fc)_{gr}  \ar[r]^{\Delta_\ast} & \QC(\fc\times \fc)_{gr}
}
\]

\begin{remark}\label{remark Kazhdan}
Defining the grading on the quantum Kostant slice is not immediate as the Whittaker equation $n.m = \psi(n)m$ is not homogeneous  (for an element $m$ in a $\fU\fg$-module $M$, and $n \in \fn$). One approach is given by the \emph{Kazhdan filtration} on $\fU\fg$ (see \cite{Gan-Ginzburg}). The Rees algebra of the Kazhdan filtration is isomorphic to the usual (PBW) Rees algebra $\fU_t\fg$ as plain algebras, but the grading is defined by the homomorphism $ (id,2\rho^\vee):\G_m \to G\times \G_m$. In particular, the category $\HC_{t,gr}$, which consists of $G\times \G_m$-weakly equivariant $\fU_t\fg$-modules may thought of in terms of the Rees algebra of either filtration. With respect to the Kazhdan filtration, the Whittaker equation is homogeneous of degree $0$, and thus we can define a graded lift of the category of Whittaker modules as required. Alternatively, one can proceed as in the classical case in Section \ref{char poly section} and consider a certain localization of the category of $B$-integral $\fU\fg$-modules with a factorization of the action of $\fU_t\fn$ through the quotient $\fU_t\fn/[\fn,\fn] \simeq \C[\ch \times A^1_t]$. One can use this latter approach to define an (ungraded) dg-version of Whittaker modules, i.e. an action of $\HC_\hbar$ on $\cZ_\hbar$.
\end{remark}


\subsection{The Whittaker category}
The Whittaker category is a monoidal category which quantizes the group scheme $J \to \fc$. To motivate the definition below, note by \cite{BFN}
\[
\QC(J) = \QC(\fc\times_{\fg^\ast/G} \fc) \simeq \End_{\QC(\fg^\ast/G)}(\QC(\fc))
\]
\begin{definition}
The  \emph{Whittaker category} is the monoidal category $ \cWh = \End_{\HC}(\cZ)$. It has a filtered lift given by
$
\cWh_{t,gr} = \End_{\HC_{t,gr}}(\cZ_{t,gr}).
$
\end{definition}

Note that under the Morita equivalence of Theorem \ref{thm-morita}, the $\HC$-module category $(\fg,N,\psi)\module$ corresponds to the $\cD(G)$-module category $\cD(G/_{\psi} N)$. Thus we identify
\[
\cWh \simeq \End_{\cD(G)}(\cD(G/_{\psi} N)) \simeq \cD(\quot{N_\psi}{G}{_\psi N}).
\]
Similarly, there is a filtered version
\[
\cWh_{t,gr} \simeq \End_{\cD_{t,gr}(G)}(\cD_{t,gr}(G/_{\psi} N)) \simeq \cD_{t,gr}(\quot{N_\psi}{G}{_\psi N}).
\]
(one should use the grading on $\fD_{G,t}$ induced by the Kazhdan filtration to make sense of this).

In general, the Drinfeld center of a monoidal category acts by endomorphisms on any module. In particular, there is a monoidal functor
\[
\Whit_{t,gr}: \cZ(\HC_{t,gr}) \simeq \cD_{t,gr}(G\adjquot G) \to \cWh_{t,gr}
\]
Unwinding the definitions, we see that this functor is given by a composite
\[
\cD_{t,gr}(G\adjquot G) \to \cD_{t,gr}(G\adjquot N) \to \cD_{t,gr}(\quot{N_\psi}{G}{_\psi N}) \simeq \fWh_{t}\module_{gr}
\]
which is identified with the Whittaker functor appearing in \cite{ginzburg whittaker} (see the next section for the algebra $\fWh_t$).

\subsection{Bi-Whittaker differential operators}
 The category $\cWh \simeq \cD(\NGNpsi)$ contains a distinguished object $$\fD_{\quot{N_\psi}{G}{_\psi N}}  = \fD_G \otimes_{\fU\fn^L \otimes \fU\fn^R}( \C_{-\psi} \otimes \C_{\psi})$$
The Skyrabin equivalence implies that this object (which represents the functor of taking left and right Whittaker vectors) is a compact generator of $\cD(\quot{N_\psi}{G}{_\psi N})$, which moreover is a projective object in the heart of the $t$-structure. Consider its endomorphism ring, which is identified with the bi-Whittaker differential operators as studied in \cite{ginzburg whittaker}
\[
\fWh := \End_{\cD(\quot{N_\psi}{G}{_\psi N})}(\fD_{\quot{N_\psi}{G}{_\psi N}}) \simeq \left(\fD_{\quot{N_\psi}{G}{_\psi N}}\right)^{N\times N}
\]
Applying the same argument in the filtered setting, we get a graded algebra $\fWh_t$ which is the Rees algebra with respect to the Kazhdan filtration (see \cite{ginzburg whittaker}) on $\fWh$\footnote{Warning: the filtration on $\fWh$ (or equivalently, the grading on the Rees algebra $\fWh_t$) is unbounded in both directions in general.}. 

We record these results in the following proposition.
\begin{proposition}
There are equivalences of categories
\[
\cWh \simeq \cD(\quot{N_\psi}{G}{_\psi N}) \simeq \fWh\module
\]
with a corresponding filtered lift
\[
\cWh_{t,gr} \simeq \cD_{t,gr}(\NGNpsi) \simeq \fWh_t\module_{gr} 
\]
\end{proposition}

The monoidal structure on $\cWh_{t,gr}$ can be recovered from a $\fZ_t\fg$-bialgebroid structure on the ring $\fWh_t$. First note that there is a map of rings $\fZ_{t}\fg \to \fWh_t$, and the corresponding forgetful functor on modules coincides with the manifestly monoidal functor 
\[
\End_{\HC_{t,gr}}(\cZ_{t,gr}) \to \End_{\cZ_{t,gr}}(\cZ_{t,gr}) \simeq \cZ_{t,gr}
\]
where we use the quantum characteristic polynomial map $\cZ_{t,gr} \to \HC_{t,gr}$. Thus the corresponding monad acting on $\cZ_{t,gr}$ is just given by the graded $\fZ_t\fg$-ring $\fWh_t$ (see \ref{monad vs algebra} for details on how a monad acting on a module category can be regarded as a ring). The monoidal structure on the forgetful functor endows the monad itself with an oplax monoidal structure, which, according to \cite{Bohm}, precisely corresponds to a (graded) $\fZ_t\fg$-bialgebroid structure on the (graded) $\fZ_t\fg$-ring $\fWh_t$.  This bialgebroid in fact is a Hopf algebroid (though we will not need this fact) which specializes to the commutative and cocommutative graded Hopf algebra $\C[J]$ after setting $t=0$.
One can recover the monoidal structure on $\fWh_t\module_{gr}$ naturally from the bialgebroid structure using the comultiplication in the usual way. 

\begin{remark}\label{abelian vs derived}
We were not able to locate a reference for bialgebroids in the dg/homotopical setting. However, our present situation may be expressed purely in terms of the usual theory in discrete abelian categories as follows. The monoidal (dg)-category $\cWh_{t,gr}$ can be recovered as the dg derived category of the heart its $t$-structure, which is a right-exact Grothendieck abelian monoidal category. The forgetful functor defines an exact, monadic, monoidal functor to the abelian category of graded $\fZ_t\fg$-modules, so the results in \cite{Bohm} apply verbatim to recover the monoidal structure on the abelian category of graded $\fWh_t$-modules (and thus on the dg-category $\cWh_t$) in terms of the $\fZ_t\fg$-bialgebroid structure.
\end{remark}

As a consequence of Remark \ref{abelian vs derived}, we obtain the following:
\begin{proposition}\label{cocommutative}
If the (discrete) graded bialgebroid $\fWh_t$ is cocommutative, then $\cWh_{t,gr}$ (and thus $\cWh$) carries a symmetric monoidal structure.
\end{proposition}
\begin{proof}
The abelian category of modules for a cocommutative bialgebroid over a commutative ring is naturally sysmmtric monoidal. This structure carries through to the derived category, as required. 
\end{proof}
\begin{remark}
In the next section we will see that the $\fZ_t\fg$-bialgebroid $\fWh_t$ is indeed cocommutative, and thus the Whittaker category $\cWh = \fWh\module$ is symmetric monoidal.
\end{remark}

\subsection{Derived geometric Satake and the Kostant/Whittaker category}\label{subsection-derived satake}
In this subsection we will apply the results of Section \ref{Hecke section} in the setting of the equivariant Grassmannian for $G^\vee$ to derive results about the Whittaker category via the derived geometric Satake theorem of Bezrukavnikov and Finkelberg \cite{BezFink}.

We take $X=pt/\LGpv\rtimes \Gm$, $\uGrv=\LGpv\backslash \LGv/\LGpv\rtimes \G_m$. Note that $X$ is an ind-stack and $\uGrv$ an ind-proper groupoid acting on $X$. Let $\cH_\hbar = \Dhol(\uGrv)$ denote the spherical Hecke category, and $\cR_\hbar = \Dhol(X)$. Note that there is an isomorphism $R_\hbar=H^\ast(X) \simeq \fZ_\hbar\fg$, thus we may identify $\cR_\hbar = H^\ast(X)\module$ with $\cZ_\hbar = \fZ_\hbar\fg\module$.


\begin{theorem}~\cite{BezFink} \label{Bez-Fink}
	There is an equivalence of monoidal categories $\cH_\hbar \simeq \HC_\hbar$ giving rise to a commutative diagram
	\[
	\xymatrix{
		\cZ_{t,gr} \ar[d]_{\Char_{t,gr}} \ar[r] &\cZ_\hbar \ar[r]^\sim \ar[d] & \ar[l] \cR_{\hbar} \ar[d]^{\fd}\\
		\HC_{t,gr} \ar[r] \ar[d]_{\Kost_{t,gr}} &\HC_\hbar \ar[r]^\sim \ar[d] & \ar[l] \cH_{\hbar} \ar[d]^{\fa} \\
		\End(\cZ_{t,gr}) \ar[r]&\End(\cZ_\hbar) \ar[r]^\sim & \ar[l] \End(\cR_{\hbar})	
}\]
where the left column is a graded lift of the right two.
\end{theorem}
\begin{remark}
Given algebra objects $S$ and $A$ (in some closed symmetric monoidal category $\cC$, say) we say that $A$ is an augmented $S$-ring if there is an algebra homomorphism $S\to A$, and $S$ carries the structure of an $A$-module, such that the composite 
\[
S \to A \to \End(S)
\]
is the structure map for $S$ as a module over itself. The Geometric Satake Theroem may be interpreted as saying that $\HC_\hbar$ and $\cH_\hbar$ are equivalent as augmented $(\cZ_\hbar \simeq \cR_\hbar)$-rings.
\end{remark}

Recall the Kostant category $\cK_\hbar$ is the symmetric monoidal category $\Dhol(X)^{\uGrv}$ which is monoidally identified with $\End_{\cH_\hbar}(\cR_{\hbar})$. On the other hand, the (dg-asymptotic) Whittaker category is the monoidal category given by $\cWh_{\hbar} = \End_{\HC_{\hbar}}(\cZ_{\hbar})$. In particular, Theorem \ref{Bez-Fink} implies that there is an equivalence between $\cWh_{\hbar}$ and $\cK_{\hbar}$; thus, $\cWh_{t,gr}$ defines a graded lift of $\cK_{\hbar}$. Thus we obtain the following:

\begin{corollary}\label{corollarykostantwhittaker}
There is a commutative diagram of monoidal categories
\[
\xymatrix{
 \cWh_{t,gr} \ar[d] \ar[r] & \cWh_\hbar \ar[r]^\sim \ar[d] & \ar[l] \cK_{\hbar} \ar[d]  \\
 \cZ_{t,gr} \ar[r] & \cZ_\hbar \ar[r]^\sim & \cR_{\hbar} 
}
\]
where the horizontal arrows are monoidal degrading functors. In particular, the dg-Whittaker category $\cWh_\hbar$ is symmetric monoidal.
\end{corollary}

Recall that the Drinfeld center of a monoidal category acts by endomorphisms on any module category. Identifying these actions on either side of Theorem \ref{Bez-Fink} we obtain:

\begin{corollary}\label{corollary action Kostant}
There a commutative diagram:
\[
\xymatrix{
\cD_{t,gr}(G\adjquot G) \ar[r] \ar[d]^{\Whit_{t,gr}} & \cD_\hbar(G\adjquot G) \ar[r] \ar[d]^{\Whit_\hbar} & \ar[l] \cZ(\cH_\hbar) \ar[d]^\fa \\
\cWh_{t,gr} \ar[r] & \cWh_{\hbar} \ar[r] & \ar[l] \cK_\hbar
}
\]
\end{corollary}

Using Corollary \ref{corollarykostantwhittaker}, we can recover a theorem of Bezrukavnikov-Finkelberg-Mirkovic \cite{BFM}, identifying the homology convolution algebra of the affine Grassmannian. 
\begin{corollary}\label{corollaryBFimpliesBFM}\cite{BFM}
There is an equivalence of graded Hopf algebroids 
\[
\fWh_t \simeq H^\ast(H_\hbar) \simeq H^{\LGpv \rtimes \G_m}_{-\ast}(\cGr)
\]
In particular, setting $t=0$ we have an equivalence of Hopf algebras.
\[
\C[J] \simeq H_{-\ast}^{\LGpv}(\cGr)
\]
\end{corollary}
\begin{proof}
We have $\cWh_\hbar = \fWh_\hbar\module$ and $\cK_\hbar = H_\hbar\module$, so Corollary \ref{corollarykostantwhittaker} gives rise to an equivalence $\fWh_\hbar \simeq H_\hbar$ of $(\fZ_\hbar\fg\simeq R_\hbar)$-rings (or equivalently monads acting on $\cZ_\hbar \simeq \cR_\hbar$). Moreover, as the forgetful functors carry a monoidal structure, the ring objects $\fWh_\hbar$ and $H_\hbar$ admit $(\fZ_\hbar\fg\simeq R_\hbar)$-bialgebroid structures, and the equivalence $\fWh_\hbar \simeq H_\hbar$ respects this structure.

The existence of the graded lift $\cWh_{t,gr} \to \cWh_\hbar$ means that the dg-bialgebra $\fWh_\hbar$ arises from the graded bialgebra $\fWh_t$ (which we consider to be in cohomological degree 0) by shearing so that the $\G_m$-weight is equal to the cohomological degree. In particular, $\fWh_\hbar$ and thus $H_\hbar$ is formal as a bialgebroid by Remark \ref{remark formality}
. In other words, the homology of $\fWh_\hbar$ (and thus of $H_\hbar$) is isomorphic to $\fWh_t$ as graded bialgebroids, as claimed.
\end{proof}

\begin{corollary}\label{corollarysymmetricmonoidal}
The monoidal category $\cWh_{t,gr}$ upgrades to a symmetric monoidal category. In particular, $\cWh = \cD(\quot{N_\psi}{G}{_\psi N})$ is symmetric monoidal.
\end{corollary}
\begin{proof}
By Corollary \ref{corollaryBFimpliesBFM} the $\fZ_t\fg$-coalgebra structure on $\Wh_t \simeq H_{-\ast}(\uGrv)$ is cocommutative (it is the ``cup coproduct'' arising from pushforward under diagonal maps). Thus the result follows from Proposition \ref{cocommutative}. 

\end{proof}

\subsection{The Quantum Ng\^o map}
Applying Theorem \ref{groupoid center} to the Ind-proper groupoid $\uGrv \rightrightarrows X$, and interpreting the results using Theorem \ref{Bez-Fink}, Corollary \ref{corollarykostantwhittaker}, and Corollary \ref{corollary action Kostant} we obtain the following result.
\begin{theorem}\label{thm-quantum-ngo-hbar}
There is a canonical $E_2$-morphism $\Ngo_\hbar:\cWh_\hbar\to \cZ(\HC_\hbar)$ which fits in to a diagram:
\[
\xymatrix{
\cWh_\hbar \ar[r]_-{\Ngo_\hbar}\ar[d]_-{}& \ar@/_1pc/_-{\Whit_\hbar}[l] \cD(G\adjquot G) \ar[d]^{}\\
\cZ_\hbar \ar[r]_-{\Char_\hbar} & \HC_\hbar  }
\]
\end{theorem}


Let us try to understand the functor $\Ngo_\hbar$ more explicitly. 
Composing $\Ngo_\hbar$ with the monoidal forgetful functor $\cD_\hbar(G\adjquot G) \simeq \cZ(\cD_\hbar(G)) \to \cD_\hbar(G)$, we obtain a functor
\[
\tNgo_\hbar: \cWh_\hbar = \fWh_{\hbar}\module \to \fD_\hbar\module
\]
All our constructions have taken place in the category $\DGCat_k$, and all functors appearing are continuous. It follows that the functor $F_\hbar$ above is represented by a $\fD_{G,\hbar}-\fWh_\hbar$-bimodule $\fB_\hbar$. Applying the forgetful functors given by the vertical arrows in Theorem \ref{thm-quantum-ngo-hbar}, we see that there is an equivalence of underlying $\fU_\hbar\fg\otimes \fU_\hbar\fg - \fZ_\hbar\fg$-bimodules:
\[
\fB_\hbar \simeq F_\hbar(\fWh_\hbar) \simeq \Char_\hbar(\fWh_\hbar) \simeq \fU\fg_{\hbar}\otimes_{\fZ_\hbar\fg}\fWh_\hbar
\]
In other words, there is a left $\fD_{\hbar,G}$-module structure on  $\fU\fg_{\hbar}\otimes_{\fZ_\hbar}\fWh_\hbar$ commuting with the right $\fWh_\hbar$ action. In particular, there is a map of $\fU_t\fg$-bimodules
\[
\fD_{G,\hbar} \to  \fU\fg_{\hbar}\otimes_{\fZ_\hbar\fg}\fWh_\hbar
\]
given by acting on the distinguished element $1\otimes 1$. This map of can be thought of as the quantiziation of the Ng\^o morphism
\[
\chi^\ast(J) \to T^\ast G
\]
(this will be explained more precisely in Remark \ref{remark classical Ngo}).

\subsection{Graded lift of the quantum Ng\^o map}\label{ss graded lift}
The goal of this subsection is to sketch a proof of the following result, which claims a graded lift of the functor $\Ngo_\hbar$ given by Theorem \ref{thm-quantum-ngo-hbar}.

\begin{theorem}\label{thm-quantum-ngo-t}
There is a $t$-exact $E_2$-monoidal functor
\[
\Ngo_{t,gr}: \cWh_{t,gr} \to \cD_{t,gr}(G\adjquot G)
\]
lifting the monoidal functor $\Char_{t,gr}: \cZ_{t,gr} \to \HC_{t,gr}$.
\end{theorem}

The idea is to deduce Theorem \ref{thm-quantum-ngo-t} from Theorem \ref{thm-quantum-ngo-hbar} using certain formality properties.

Let us first construct the composite 
\[
\tNgo_{t,gr}: \cWh_{t,gr} \too \cD_{t,gr}(G\adjquot G) \too \cD_{t,gr}(G)
\]
Such a functor will be represented by a certain graded $\fD_{G,t}-\fWh_t$-bimodule, $\fB_t$, whose underlying  $\fU_t\fg\otimes \fU_t\fg - \fZ_t\fg$-bimodule is isomorphic to $\fU_t\fg \otimes_{\fZ_t\fg} \fWh_t$. Recall from the comments following Theorem \ref{thm-quantum-ngo-hbar} that we have a corresponding dg-bimodule $\fB_\hbar$. Let us note the following:

\begin{lemma}\label{lemma formal}
The object $\fB_\hbar$ is formal as a $\fD_{G,\hbar}-\fWh_\hbar$-bimodule.
\end{lemma}
\begin{proof}
By construction, $\fU_\hbar\fg$, $\fZ_\hbar\fg$, and $\fWh_\hbar$ all carry compatible pure external gradings (i.e. the weight of the external grading on the $i$th cohomology object is equal to $i$). Note also that $\fU_\hbar\fg$ is free as a $\fZ_\hbar\fg$-module, by a theorem of Kostant. Thus the $\fU_\hbar\fg\otimes \fU_\hbar\fg - \fZ_\hbar\fg$-bimodule
\[
\fU\fg_{\hbar}\otimes_{\fZ_\hbar\fg}\fWh_\hbar
\]
carries a pure grading, so in particular is formal as a $\fU_\hbar\fg\otimes \fU_\hbar\fg - \fZ_\hbar\fg$-bimodule (Kostant's theorem implies that the tensor product as graded algebras is the same as the tensor product in the dg-derived category). On the other hand, we have an equivalence $\fD_{G,\hbar} \simeq \cO(G) \rtimes \fU_\hbar\fg$, where $\cO(G)$ is in pure degree $0$. It follows that the $\fU_\hbar\fg$-module isomorphism from $\fB_\hbar$ to its homology is automatically a $\fD_{G,\hbar}$-module isomorphism, as required.
\end{proof}

Lemma \ref{lemma formal} is equivalent to the statement that $\fB_{\hbar}$ carries a pure external grading as a (dg) $\fD_{G,\hbar}-\fWh_\hbar$-bimodule. In particular, $\tNgo_\hbar$ lifts to a functor
\[
\tNgo_{\hbar,gr}: \cWh_{\hbar,gr} \to \cD_{\hbar,gr}(G)
\]
or equivalently, after shearing,
\[
\tNgo_{t,gr}: \cWh_{t,gr} \to \cD_{t,gr}(G)
\]
Note that $\tNgo_{t,gr}$ is $t$-exact as it is a lift of $\Char_{t,gr}$, which is $t$-exact due to the flatness of $\fU_t\fg$ over $\fZ_t\fg$. 

\begin{remark}
The functor $\tNgo_{t,gr}$ is represented by a graded $\fD_{G,t}-\fWh_t$-bimodule $\fB_t$ (sitting in cohomological degree $0$); this is just the formal (dg)-bimodule $\fB_\hbar$ where the cohomological grading is reinterpreted as an external grading.
\end{remark}

To deduce Theorem \ref{thm-quantum-ngo-t} boils down to equipping the bimodule $\fB_{t}$ with extra structure, corresponding to the fact that the functor $\tNgo_{t,gr}$ factors through the center, and the factorization $\Ngo_{t,gr}$ carries an $E_2$-monoidal structure. To simplify matters, let us consider the restriction of the (for now, still hypothetical) functor $\Ngo_{t,gr}$ to the subcategory $\cWh_{t,gr}^{proj,\heartsuit}$ of $\cWh_{t,gr}$ consisting of graded, projective $\fWh_t$-modules in the heart of the $t$-structure. Such objects are, in particular, projective modules (and thus free by the Quillen-Suslin theorem) over $\fZ_t\fg$; we will denote the category of such modules by $\cZ_{t,gr}^{fr,\heartsuit}$. Let us also consider the category $\HC_{t,gr}^{fr,\heartsuit}$ of graded Harish-Chanrdra bimodules (in the heart of the $t$-structure) which are free as left (or equivalently, right) modules over $\fU_t\fg$ (such objects are necessarily of the form $\fU_t\fg \otimes V(k)$ where $V$ is a representation of $G$, and $(k)$ indicates grading shift). Note that $\HC_{t,gr}^{fr,\heartsuit}$ is a discrete, exact category which sits fully faithfully in the dg-category $\HC_{t,gr}$, similarly for $\cZ_{t,gr}^{fr,\heartsuit}$ and $\cWh_{t,gr}^{proj}$; these categories form the heart of a weight structure on the corresponding dg-categories, in the sense of \cite{Bondarko}.

The graded Ng\^o functor (assuming it exists) must restrict to a braided monoidal functor
\[
\Ngo_{t,gr}^{fr,\heartsuit}: \cWh_{t,gr}^{proj,\heartsuit} \to \cZ(\HC_{t,gr}^{fr})
\]
On the other hand, the dg-category $\cWh_{t,gr}$ can be recovered as the category of complexes in the additive category $\cWh_{t,gr}^{proj}$:
\[
\cWh_{t,gr} \simeq K(\cWh_{t,gr}^{fr})
\]
Assuming certain properties of the functor $K$ which takes an additive category to its category of complexes, one may recover the functor $\Ngo_{t,gr}$ from its restriction to $\cWh_{t,gr}^{proj,\heartsuit}$. 

Now let us explain how to construct $\Ngo_{t,gr}^{fr,\heartsuit}$ from $\Ngo_\hbar$ (which was constructed in Theorem \ref{thm-quantum-ngo-hbar}). Consider the subcategory
$
\HC_\hbar^{fr}
$
 consisting of direct sums and cohomological shifts of objects of the form $\fU_\hbar\fg \otimes V$, where $V$ is a finite dimensional representation of $G$. Note that $\HC_{\hbar}^{fr}$ is a non-stable, additive, $\C$-linear $\infty$-category, and its homotopy category $H^0\HC_{\hbar}^{fr}$ is a discrete additive category. Similarly, we define $\cZ_\hbar^{fr}$ to be the subcategory consisting of direct sums and shifts of $\fZ_\hbar\fg$, and $\cWh_{\hbar}^{fr}$ to be the full subcategory consisting of direct sums and shifts of $\fWh_\hbar$. Finally, let $\cD_\hbar(G\adjquot G)^{fr}$ be the full subcatgory of $\cD_\hbar(G\adjquot G)$ such that the essential image of the forgetful functor to $\HC_\hbar$ is contained in $\HC_\hbar^{fr}$.

\begin{lemma}\label{lemma homotopy category}
There is a monoidal equivalence of categories
\[
H^0(\HC_\hbar^{fr}) \simeq \HC_{t,gr}^{fr,\heartsuit}
\]
Analogous results hold for $\cZ^{fr}$, $\cWh^{fr}$, and $\cZ(\HC^{fr})$ (as braided monoidal categories).
\end{lemma}
\begin{proof}[Proof (Sketch)]
	The objects $\fU_\hbar\fg \otimes V[k]$ form a skeleton of $H^0\HC_\hbar^{fr}$, where $V$ ranges over a skeleton of $\Rep(G)^\heartsuit$, and $k$ ranges over the integers. These objects correspond to $\fU_t\fg\otimes V (k)$ in $\HC_{t,gr}^{fr}$. We observe that the morphism sets agree, as required.
	\end{proof}
\begin{remark}
	Note that the cohomological shift functor $[1]$ is taken to the grading shift $(1)$ under the equivalences of Lemma \ref{lemma homotopy category}.
\end{remark}
	
Note that $\Char_\hbar$ takes objects of $\cZ_\hbar^{fr}$ to $\HC_\hbar^{fr}$; as $\fWh_\hbar$ is itself free over $\fZ_\hbar\fg$, we see that $\Ngo_\hbar$ (which lifts $\Char_\hbar$) takes objects of $\cWh_\hbar^{fr}$ to $\cZ(\HC_\hbar^{fr})$. In fact, we claim that $\Ngo_\hbar$ restricts to a $E_2$-monoidal functor
\[
\Ngo_\hbar^{fr}: \cWh_{\hbar}^{fr} \to \cZ(\HC_\hbar^{fr})
\]

Taking the homotopy categories and using Lemma \ref{lemma homotopy category}, we obtain a braided monoidal functor 
\[
\Ngo_{t,gr}^{fr,\heartsuit}: \cWh_{t,gr}^{fr,\heartsuit} \to \cZ(\HC_{t,gr}^{fr,\heartsuit})
\]
The functor $\Ngo_{t,gr}$ is then obtained by taking the functor $K$.

\begin{remark}\label{mixed remark}
	Under geometric Satake, the subcategory $\HC_\hbar^{fr}$ corresponds to the full subcategory $\cH^{pure}_\hbar$ of $\cH_\hbar = \Dhol(\uGr)$ consisting of direct sums of intersection cohomology complexes $IC_\mu$ on orbits $\uGr_\mu$ (in fact the equivalence is proved by first identifying these subcategories). The graded lift $\HC_{t,gr}$ of $\HC_\hbar$ corresponds to a \emph{mixed} Satake category. The word ``mixed'' is used in the sense of Beilinson-Ginzburg-Soergel \cite{BGS} (see also \cite{riche} for a mixed version of the derived geometric Satake eqivalence in the modular setting); the weight of the grading corresponds to weight as in Deligne's theory of weights or in mixed Hodge theory. The reconstruction of the functor $\Ngo_{t,gr}$ from $\Ngo_\hbar$ mirrors the construction of a mixed category by taking the dg-category of complexes of the additive category of a suitable subcategory of pure objects--see e.g. \cite{Rider, Achar-Riche}.

\end{remark}

\begin{remark}
Note that the quantum Hamiltonian reduction of $\fB$ (considered as an equivariant left $\fD_G$-module) is equivalent to $\fWh$ as a right $\fWh$-module. It follows that the left action of $(\fD_{G\adjquot G})^G \simeq (\fD_T)^W$ is given by a ring homomorphism $(\fD_T)^W \to \fWh$. Thus the Ng\^o functor, at the level of quantum Hamiltonian reduction, is given by the forgetful functor from $\fWh$ to $(\fD_T)^W$. This is the basis for Conjecture \ref{quantum ngo induction}.
\end{remark}

Lemma \ref{lemma formal} is equivalent to the statement that $\fB_{\hbar}$ carries a pure external grading as a (dg) $\fD_{G,\hbar}-\fWh_\hbar$-bimodule. In particular, $\tNgo_\hbar$ lifts to a functor
\[
\tNgo_{\hbar,gr}: \cWh_{\hbar,gr} \to \cD_{\hbar,gr}(G)
\]
or equivalently, after shearing,
\[
\tNgo_{t,gr}: \cWh_{t,gr} \to \cD_{t,gr}(G)
\]
Note that $\tNgo_{t,gr}$ is $t$-exact as it is a lift of $\Char_{t,gr}$, which is $t$-exact due to the flatness of $\fU_t\fg$ over $\fZ_t\fg$.

\begin{remark}\label{remark classical Ngo}
The bimodule structure gives a morphism
\[
\fD_{G,t} \to \fB_t = \fU_t\fg \otimes_{\fZ_t\fg} \fWh_t
\]
If we set $t=0$, then the monoidal category $\HC_{t=0,gr}\simeq \QC(\fg^\ast/G)_{gr}$ upgrades to a symmetric monoidal category, and the action of $\HC_{t=0,gr}$ on $\cZ_{t=0,gr}$ upgrades to the symmetric monoidal functor $\kappa^\ast:\QC(\fg^\ast/G) \to \QC(\fc)$. It follows that the monad $\fWh_{t=0} = \cO(J)$ is in fact a commutative (and cocommutative) Hopf algebra object in $\QC(\fc)_{gr}$, and thus $\fB_{t=0} = \QC(\chi^\ast J)$ is a cocommutative Hopf algebra object in $\QC(\fg^\ast/G)_{gr}$. It follows formally that the structure map $\fD_{G,t} \to \fB_{t}$ arising from the Ng\^o map is in fact a morphism of Hopf algebroids over $\cO(\fg^\ast)$
\[
\cO(T^\ast G) \to \cO(\chi^\ast(J))
\]
Thus there is a morphism of groupoids over $\fg^\ast$:
\[
\chi^\ast(J) \to T^\ast G
\]
which factors thorugh the centralizer subgroup $\cI \rightarrow T^\ast G$. To see that this agrees with the Ng\^o homomorphism, as constructed in \cite{Ngo} (using Hartog's lemma), it suffices to check that they agree on the regular semisimple locus. In terms of sheaves on the Grassmannian, this corresponds to a certain localization of $\cH_{\hbar=0}$ over $R_{\hbar=0}$, after which $\cH_{\hbar=0}$ and $\cW_{\hbar=0}$ become Morita equivalent (this corresponds to the fact that $\QC(\fg^\ast/G)$ becomes equivalent to $\Rep(J) = \QC(\fc/J)$ after localizing, and this latter category is Morita equivalent to $\QC(J)$ under convolution, by 1-affineness). Thus, after this localization, the Ng\^o map is just the natural braided monoidal functor from a symmetric monoidal category to its Drinfeld center. The corresponding functor arising from Ng\^o's construction can also be characterized in this way, so the two constructions must agree.
\end{remark}

\subsection{Example: the abelian case}
	Suppose $G=T$ is an algebraic torus, and $\Gv = \Tv$ the dual torus. In this case, everything can be made very explicit. First, note that the conclusions of Theorem \ref{thm-quantum-ngo-t} are clear: the $\cW$-category is just $\cD_\hbar(T)$ under convolution, which is   symmetric monoidal as $T$ is commutative; the Ng\^o map $\cD_\hbar(T) \to \cD_\hbar(T\adjquot T)$ is just the natural map from a symmetric monoidal category in to its own Drinfeld center. Explicitly, this situation is controlled by the cocommutative Hopf algebroid $\cD_{T,\hbar}$: the $\cW$-category $\cD_\hbar(T)$ is its category of modules (which is symmetric monoidal), the category of Harish-Chandra bimodules is given by $\cD_{T,\hbar}$-comodules in $\fU_\hbar(\ft)\module$, and $\cD_\hbar(T\adjquot T)$ can be thought of as Yetter-Drinfeld modules for $\cD_{T,\hbar}$, which identifies as the center of $\cD_\hbar(T)$ and of $\HC_{T,\hbar}$ (in fact, the two monoidal categories are Morita equivalent).
	
	Let $\Lambda = \Hom(T,\G_m) \subset \ft^\ast$ denote the character lattce of $T$, and consider the action groupoid $\Gamma_{\Lambda,\hbar}$ of $\Lambda$ acting on $\ft^\ast \times \A^1_\hbar$ by $n \cdot (\lambda,\hbar) = (\lambda + \hbar n,\hbar)$. The corresponding convolution algebra $H_{\Lambda,\hbar} = \Sym(\ft[-2] \oplus \C.\hbar) \rtimes \C[\Lambda]$ is a cocommutative Hopf algebroid over $R_\hbar \simeq \Sym(\ft \oplus \C.\hbar)$; the corresponding convolution category $\cH_{\Lambda,\hbar}$ of sheaves on $\Gamma_{\Lambda,\hbar}$ is identified with $H_{\Lambda,\hbar}\comod$, and $\cK_{\Lambda,\hbar}$ with $H_{\Lambda_\hbar}\module$. 
	
	It is easy to check that $\fD_{\hbar,T}$ coincides with $H_{\Lambda,\hbar}$ as Hopf algebroids over $\fU_\hbar(\ft) =R_\hbar$. Thus $\HC_{T,\hbar}$ identifies with $\cH_{\Lambda,\hbar}$, and $\cD_{\hbar}(T)$ with $\cWh_{\Lambda,\hbar}$.
	
	On the other hand, the affine Grassmannian for $T^\vee$ is equal to $\Lambda$ (we only care about the reduced scheme structure here). The spherical Hecke category $\Dhol(\uGrv_{T^\vee})$ is identified with $\cH_{\Lambda,\hbar}$, and the convolution bialgebra of chains $C_\ast (\uGrv)$ with $H_{\Lambda,\hbar}$. These identifications explicitly establish the renormalized Satake equivalence of Theorem \ref{Bez-Fink} in this setting. Note that the inclusion of local systems in to the spherical Hecke category is an equivalence in this case, corresponding to the fact that the Kostant section (which is just the map $\ft^\ast \to \ft^\ast /T = \ft^\ast \times BT$) is surjective.

\end{document}